\documentclass[11pt]{amsart}
\usepackage{amssymb}
\usepackage{mathtools}
\usepackage{color}

\newtheorem{theo}{Theorem} 

\newtheorem{lemma}{Lemma}[section]
\newtheorem{prop}[lemma]{Proposition}
\newtheorem{corol}[lemma]{Corollary}

\newtheorem{claim}[lemma]{Claim}

\theoremstyle{remark}
\newtheorem{remark}[lemma]{Remark}

\newtheorem*{notations}{Notations}

\theoremstyle{definition}
\newtheorem{defi}[lemma]{Definition}

\newcommand{\la}{\left\langle}
\newcommand{\ra}{\right\rangle}

\newcommand{\NN}{\mathbb{N}}
\newcommand{\RR}{\mathbb{R}}

\newcommand{\eps}{\varepsilon}

\newcommand{\DDD}{\mathcal{D}}

\newcommand{\MMM}{\mathcal{M}}

\newcommand{\OOO}{\mathcal{O}}
\newcommand{\PPP}{\mathcal{P}}
\newcommand{\SSS}{\mathcal{S}}
\newcommand{\TTT}{\mathcal{T}}
\newcommand{\UUU}{\mathcal{U}}
\newcommand{\VVV}{\mathcal{V}}

\newcommand{\ZZZ}{\mathcal{Z}}
\newcommand{\tZZZ}{\widetilde{\mathcal{Z}}}

\newcommand{\tPsi}{\widetilde{\Psi}}

\newcommand{\vA}{A}

\newcommand{\ve}{\vec{e}}
\newcommand{\vell}{\boldsymbol{\ell}}

\newcommand{\tc}{\tilde{c}}

\newcommand{\ty}{\tilde{y}}

\newcommand{\tPhi}{\widetilde{\Phi}}

\newcommand{\tu}{\widetilde{u}}

\newcommand{\unu}{\underline{u}}

\newcommand{\epsb}{\overline{\eps}}

\newcommand{\tQ}{\widetilde{Q}}

\newcommand{\tK}{\widetilde{K}}

\newcommand{\loc}{\rm loc}
\newcommand{\hdot}{\dot{H}^1}

\newcommand{\EMPH}[1]{\medskip\noindent\textit{#1}.}

\DeclareMathOperator{\supp}{supp}

\DeclareMathOperator{\Imax}{I_{\max}}

\newcommand{\ds}{\displaystyle}

\numberwithin{equation}{section} 

\title[Nonradial waves]{Solutions of the focusing nonradial critical wave equation with the compactness property}
\author[T.~Duyckaerts]{Thomas Duyckaerts$^1$}
\author[C.~Kenig]{Carlos Kenig$^2$}
\author[F.~Merle]{Frank Merle$^3$}
\thanks{$^1$LAGA, Universit\'e Paris 13 (UMR 7539). Partially supported by ANR Grant SchEq, ERC Grant Dispeq and ERC Advanced Grant  no. 291214, BLOWDISOL}
\thanks{$^2$University of Chicago. Partially supported by NSF Grant DMS-0968472 and by NSF Grant DMS-1265249}
\thanks{$^3$Cergy-Pontoise (UMR 8088), IHES. Partially supported by ERC Advanced Grant  no. 291214, BLOWDISOL}

\date{\today}

\begin{document}

%

\begin{abstract}
 Consider the focusing energy-critical wave equation in space dimension $3$, $4$ or $5$. In a previous paper, we proved that any solution which is bounded in the energy space converges, along a sequence of times and in some weak sense, to a solution \emph{with the compactness property}, that is a solution whose trajectory stays in a compact subset of the energy space up to space translation and scaling. It is conjectured that the only solutions with the compactness property are stationary solutions and solitary waves that are Lorentz transforms of the former. In this note we prove this conjecture with an additional non-degeneracy assumption related to the invariances of the elliptic equation satisfied by stationary solutions. The proof uses a standard monotonicity formula, modulation theory, and a new channel of energy argument which is independent of the space dimension.
\end{abstract}

\maketitle

\tableofcontents

\section{Introduction}
In this work we consider the energy-critical focusing nonlinear wave equation in space dimension $N=3,4,5$:
\begin{equation}
\label{CP}
\left\{ 
\begin{gathered}
\partial_t^2 u -\Delta u-|u|^{\frac{4}{N-2}}u=0,\quad (t,x)\in I\times \RR^N\\
u_{\restriction t=0}=u_0\in \hdot,\quad \partial_t u_{\restriction t=0}=u_1\in L^2,
\end{gathered}\right.
\end{equation}
where $I$ is an interval ($0\in I$), $u$ is real-valued, $\hdot:=\hdot(\RR^N)$, and $L^2:=L^2(\RR^N)$.

The equation is locally well-posed in $\hdot\times L^2$. If $u$ is a solution, we will denote by $(T_-(u),T_+(u))$ its maximal interval of existence. On $(T_-(u),T_+(u))$, the following two quantities are conserved:
$$ E[u]=E(u(t),\partial_tu(t))=\frac 12\int |\nabla u|^2dx+\frac{1}{2}\int (\partial_tu)^2dx-\frac{N-2}{2N}\int |u|^{\frac{2N}{N-2}}dx$$
(the energy) and 
$$ P[u]=P(u,\partial_t u(t))=\int u\nabla udx$$
(the momentum). 

Denote by $\Sigma$ the set of non-zero stationary solutions of \eqref{CP}:
\begin{equation}
 \label{eq07}
\Sigma:=\Big\{Q\in \hdot(\RR^N)\setminus\{0\}\quad\text{s.t.}\quad-\Delta Q=|Q|^{\frac{4}{N-2}}Q\Big\}.
\end{equation} 
The only radial elements of $\Sigma$ are $\pm\lambda^{\frac{N-2}{2}}W(\lambda x)$, $\lambda>0$ where the ground state $W$ is given by
\begin{equation}
 \label{defW}
W(x)=\frac{1}{\left(1+\frac{|x|^2}{N(N-2)}\right)^{\frac{N-2}{2}}}.
\end{equation} 
In \cite{DuKeMe13Pa}, the authors have proved the soliton resolution for spherically symmetric solutions of equation \eqref{CP} in the case $N=3$. Namely, any bounded radial solution $u$ of \eqref{CP} has an asymptotic expansion of the following form
$$ u(t,x)=\sum_{j=1}^J \frac{\iota_j}{\lambda_j(t)^{\frac{1}{2}}} W\left(\frac{x}{\lambda_j(t)}\right)+v_L(t,x)+\eps(t,x),$$
where $v_L$ is a solution to the linear wave equation, $J$ is a natural number ($J\geq 1$ if $T_+(u)$ is finite), $\iota_j\in \{-1,+1\}$, and, as $t\to T_+(u)$, 
$$0<\lambda_1(t)\ll \ldots \ll\lambda_J(t),\quad (\eps(t),\partial_t\eps(t))\longrightarrow 0 \text{ in }\hdot\times L^2.$$
The proof is based on the classification of radial solutions of \eqref{CP} that do not satisfy an exterior energy estimate, and the ``channel of energy'' method, which consists, in a contradiction argument, in bounding from below the $\hdot\times L^2$ norm of the solution outside a well-chosen light cone. This work uses several properties that are specific to the radial case (in particular, it relies heavily on the fact that $W$ is the only non-zero radial stationary solution of \eqref{CP} up to scaling), and to space dimension $3$, where the exterior energy estimates for the free wave equation are the strongest.

Much less is known in higher dimensions, and in the nonradial case. The existence of elements of $\Sigma$ that are not spherically symmetric, and with arbitrary large energy was proved by Ding  \cite{Ding86} using a variational argument. More explicit constructions of such solutions are available in \cite{dPMPP11,dPMPP13}. However, only existence results are available, and the elements of $\Sigma$ are not classified. 

Other particular solutions of \eqref{CP} are solitary waves given by Lorentz transform of stationary solutions: if $Q\in \Sigma$ and $\vell\in \RR^N$ satisfies $|\vell|<1$, then
$$Q_{\vell}(t,x)=Q\left(\left(-\frac{t}{\sqrt{1-|\vell|^2}}+\frac{1}{|\vell|^2} \left(\frac{1}{\sqrt{1-|\vell|^2}}-1\right)\vell\cdot x\right)\vell+x\right)=Q_{\vell}(0,x-t\vell)$$
is a global, non-scattering, bounded solution of \eqref{CP}, travelling in the direction $\vell$ (here and in the sequel $|\cdot|$ is the Euclidean norm on $\RR^N$).

We expect that the soliton resolution for \eqref{CP} is still valid without the radiality assumption, namely that any solution $u$ that is bounded for positive time can be written, as $t\to T_+(u)$, as a finite sum of solitary waves modulated by space translations and scaling, a linear solution, and a remainder that goes to $0$ in $\hdot\times L^2$. One major difficulty in the proof of this conjecture is the lack of classification of solutions of the stationary equation.

The main result of our previous paper, \cite[Theorem 1]{DuKeMe13Pa}, is a first step in the classification of arbitrary large, bounded, nonradial solutions of \eqref{CP}. It implies that for any bounded solution of \eqref{CP}, there exists a sequence of time $t_n\to T^+(u)$ such that $u(t_n)$ converges, in some weak sense and up to scaling and space-translation, to the initial data $Q_{\vell}(0)$ of a solitary wave. The proof of \cite{DuKeMe13Pa} is based on the notion of solutions of \eqref{CP} \emph{with the compactness property}, which first appears in \cite{GlMe95P} and plays an important role in the compactness/rigidity method initiated in \cite{KeMe06}  (See also \cite{Keraani06}, \cite{KeMe08}, \cite{TaViZh08}).
\begin{defi}
\label{D:compactness}
We say that a solution $u$ of \eqref{CP} has the \emph{compactness property} when there exists $\lambda(t)>0$, $x(t)\in \RR^N$, defined for $t\in (T_-(u),T_+(u))$ such that:
$$K=\left\{ \left( \lambda^{\frac{N}{2}-1}(t)u\left(t,\lambda(t)\cdot+x(t)\right), \lambda^{\frac N2}(t)\partial_t u\left(t,\lambda(t)\cdot+x(t)\right)\right),\; t\in (T_-(u),T_+(u))\right\}$$
has compact closure in $\hdot\times L^2$.
 \end{defi}
The null solution, as well as the solitary waves $Q_{\vell}$, with $Q\in \Sigma$, $|\vell|<1$ have the compactness property. We conjecture (\emph{rigidity conjecture for solutions with the compactness property}) that these are the only solutions of \eqref{CP} with the compactness property.  

This conjecture was settled in \cite[Theorem 2]{DuKeMe11a} for radial solutions. Again, the uniqueness of the radial stationary solution $W$ plays an important role in the proof. Without the radial assumption, one has the following weaker result from \cite{DuKeMe13Pa}:
\begin{prop}
\label{P:maintheo1}
 Let $u$ be a nonzero solution with the compactness property, with maximal time of existence $(T_-,T_+)$. Then
\begin{enumerate}
\item \label{I:def_vell}$E[u]>0$ and $|\vell|<1$, where
$\vell=-\frac{P[u]}{E[u]}.
$
 \item \label{I:half_global} $T_-=-\infty$ or $T_+=+\infty$.
\item \label{I:theo_virial} there exist two sequences $\{t_n^{\pm}\}_n$, two elements $Q^{\pm}$ of $\Sigma$ such that $\lim_{n\to +\infty}t_n^{\pm}=T_{\pm}$ and
\begin{multline}
 \label{theo_CV}
\lim_{n\to\infty}\left\|\lambda^{\frac{N}{2}-1}\left(t_n^{\pm}\right)u\left(t_n^{\pm},\lambda\left(t_n^{\pm}\right)\cdot +x\left(t_n^{\pm}\right)\right)-Q_{\vell}^{\pm}\left(t_n^{\pm}\right)\right\|_{\hdot}\\
+\left\|\lambda^{\frac{N}{2}}\left(t_n^{\pm}\right)\partial_t u\left(t_n^{\pm},\lambda\left(t_n^{\pm}\right)\cdot +x\left(t_n^{\pm}\right)\right)-\partial_tQ_{\vell}^{\pm}\left(t_n^{\pm}\right)\right\|_{L^2}=0.
\end{multline}
\end{enumerate}
\end{prop}
It is essential, in order to prove the soliton resolution conjecture, to improve the classification of solutions with the compactness property. In this paper, we prove the rigidity conjecture for solutions with the compactness property, under an additional nondegeneracy assumption on the energy functional at the stationary solution $Q^+$ given by Proposition \ref{P:maintheo1}. This condition is related to the invariances of $\Sigma$. 

If $Q\in \Sigma$, then $x\mapsto Q(x+b)$, where $b\in \RR^N$, $x\mapsto Q(Px)$, where $P\in \OOO_N$, $x\mapsto \lambda^{N/2-1}Q(\lambda x)$, where $\lambda>0$ and $x\mapsto \frac{1}{|x|^{N-2}}Q\left(\frac{x}{|x|^2}\right)$ are also in $\Sigma$ ($\OOO_N$ is the classical orthogonal group). We will denote by $\MMM$ the group of isometries of $L^{\frac{2N}{N-2}}$ (and $\hdot$) generated by the preceding transformations. We will see that $\MMM$ defines a $N'$-parameter family of transformations in a neighborhood of the identity, where $$N'=2N+1+\frac{N(N-1)}{2}.$$ 
If $Q\in \Sigma$ we let 
\begin{equation}
 \label{eq03}
L_{Q}=-\Delta -\frac{N+2}{N-2}|Q|^{\frac{4}{N-2}}
\end{equation} 
be the linearized operator around $Q$. Let
\begin{equation}
 \label{eq04} \ZZZ_{Q}=\left\{ f\in \hdot(\RR^N)\text{ s.t. }L_{Q}f=0\right\}
\end{equation} 
and 
\begin{multline}
 \label{eq05}
\widetilde{\ZZZ}_{Q}=\mathrm{span}\Big\{(2-N)x_jQ+|x|^2\partial_{x_j}Q-2x_jx\cdot\nabla Q,\partial_{x_j}Q,\; 1\leq j \leq N,\\
   (x_j\partial_{x_k}-x_k\partial_{x_j})Q, \; 1\leq j<k\leq N,\;\frac{N-2}{2}Q+x\cdot\nabla Q\Big\}.
\end{multline} 
The vector space $\widetilde{\ZZZ}_{Q}$ is the null space of $L_{Q}$ generated by the family of transformations $\MMM$ defined above, so that $\widetilde{\ZZZ}_{Q}\subset\ZZZ_{Q}$ (see Lemma \ref{L:A12} for a rigorous proof). We note that $\tZZZ_Q$ is of dimension at most $N'$, but might have strictly lower dimension if $Q$ has symmetries. For example, 
$$\tZZZ_W=\Big\{ \frac {N-2}{2}W+x\cdot \nabla W,\; \partial_{x_j}W,\;1\leq j\leq N\Big\}$$
is of dimension $N+1$. We will make the following non-degeneracy assumption
\begin{equation}
 \label{ND}
\ZZZ_{Q}=\widetilde{\ZZZ}_{Q}.
\end{equation} 
If $Q$ satisfies \eqref{ND} and $\theta\in \MMM$ then $\theta(Q)$ also satisfies \eqref{ND}. Furthermore, $W$ satisfies \eqref{ND} (see \cite[Remark 5.6]{DuMe08}).

The main result of this paper is the following:
\begin{theo}
\label{T:maintheo2}
 Let $u$ be a non-zero solution with the compactness property. Assume that $Q^+$ or $Q^-$ (given by Proposition \ref{P:maintheo1}) satisfies the non-degeneracy assumption \eqref{ND}. Then there exists $Q\in \Sigma$ such that $u =Q_{\vell}$, where 
  $\vell=-P[u]/E[u]$ satisfies $|\vell|<1$ by Proposition \ref{P:maintheo1}.
 \end{theo}

The nondegeneracy assumption \eqref{ND} is classical in spectral theory and geometrical analysis. We have learned from M. Del Pino \cite{delPinoPC} that C. Musso and J. Wei \cite{MuWe14} have recently established the nondegeneracy of the solutions constructed in \cite{dPMPP11,dPMPP13}. There is no known example of a stationary solutions of \eqref{CP} that does not verify \eqref{ND}.

The main new ingredient in the proof of Theorem \ref{T:maintheo2} is an exterior energy argument. Unlike our previous papers on equation \eqref{CP} using a similar method, the space dimension is not restricted to $N=3$ (see \cite{DuKeMe11a,DuKeMe12b,DuKeMe12c,DuKeMe13}) or to odd space dimension \cite{DuKeMe12}, but works the same way in any low space dimension. We refer to the sketch of proof below for more details. As usual, the restriction $N\leq 5$ is merely to avoid a nonlinearity with a low regularity, and it is very likely that the proof adapts to higher dimensions using the results of \cite{BuCzLiPaZh13} and additional technicalities to deal with the fact that the potential $|Q|^{\frac{4}{N-2}}$ is not $C^1$ if $N\geq 6$.

With an additional a priori bound on the $\dot{H}^1\times L^2$ norm of the solution $u$, the stationary solution $Q$ in the conclusion of Theorem \ref{T:maintheo2} is equal (up to scaling, space-translation and sign change), to the ground state $W$: see Corollary \ref{C:W} in Section \ref{S:proof} below. This implies the rigidity theorem \cite[Theorem 2]{DuKeMe12}. We take this opportunity to mention that there is a mistake in the statement of \cite[Theorem 2]{DuKeMe12}. We refer to Corollary \ref{C:W} for a corrected version of this result. See also the corrected arXiv version.

\medskip

We next sketch the proof of Theorem \ref{T:maintheo2}, which is given in Section \ref{S:proof}. 

The first step (see \S \ref{SS:reduction}) is to use the Lorentz transformation to reduce to the case of a zero momentum solution. For this we need to know that the Lorentz transform of a solution of \eqref{CP} with the compactness property is a solution of \eqref{CP} with the compactness property. This fact, proved in Section \ref{S:Lorentz} is not obvious since the Lorentz transformation mixes space and time variables. In this section, we also clarify a few facts about Lorentz transformation of solutions of \eqref{CP}. This also uses precise properties of the Cauchy problem for \eqref{CP} (proved in Section \ref{S:prelim}).

We next apply Proposition \ref{P:maintheo1} to $u$ (see \S \ref{SS:stat}). Since by the first step $\vell=-P[u]/E[u]=0$, this yields a stationary solution $Q\in \Sigma$ and a sequence $t_n\to T_+(u)$ such that $\big\{(u(t_n),\partial_tu(t_n))\big\}_n$ converges to $(Q,0)$ up to space translation and scaling. Reversing time if necessary we can assume that $Q$ satisfies the nondegeneracy assumption \eqref{ND}. We must prove that $u$ equals $Q$ (after a fixed translation and scaling). We argue by contradiction: if it is not the case, we construct in \S \ref{SS:compact}, using a continuity argument, a solution $w$ of \eqref{CP} with the compactness property which has the energy of $Q$, is close to $Q$ for positive times, but is not a stationary solution.

We next use the main result of Section \ref{S:modulation} that states that any nonstationary solution $w$ which has the energy of a stationary solution $Q$ and remains close to $Q$ satisfies $T_+(w)=+\infty$ and has an asymptotic expansion of the form 
\begin{equation}
 \label{asym_exp}
w(t)=S+e^{-\omega t}Y+O(e^{-\omega^+t}), \quad t\to +\infty,
\end{equation} 
where $S\in \Sigma$, $Y$ is an eigenfunction of the linearized operator $L_S$, $-\omega^2$ is the corresponding negative eigenvalue (with $\omega>0$) and $\omega^+>\omega$. It is in this part of the proof that we use the nondegeneracy assumption \eqref{ND}. The proof uses modulation theory in the spirit of \cite{DuMe08}, where this result is proved for $Q=W$. However new technical difficulties arise because $L_S$ might have more than one negative eigenvalue and more invariances must be taken into account. 

We finally reach a contradiction (see \S \ref{SS:end_of_proof}) by proving that there is no solution $w$ with the compactness property and the expansion \eqref{asym_exp}. This is the core of the proof of Theorem \ref{T:maintheo2}. The idea is to use a channel of energy argument which is based on exterior energy estimates for the linearized equation $\partial_t^2h+L_Sh=0$ instead of the free wave equation $\partial_t^2u-\Delta u=0$. This argument, which is the main novelty of the paper, has the advantage of working in any space dimension $N\geq 3$, whereas the usual channel of energy method depends very strongly on the dimension $N$.

Apart from the sections mentioned above, Section \ref{S:stationary} is dedicated to preliminaries on stationary solutions. We recall there a result of Mehskov \cite{Meshkov89} on the decay of eigenfunction of the linearized operator at a stationary solution which is crucial in our proof. We thank T. Cazenave for helpful discussions on this topic and for mentioning Meshkov's work to us.

\medskip

We conclude this introduction by giving some references on related works. The defocusing energy-critical wave equation was treated in many papers, including \cite{Grillakis90b,Grillakis92,ShSt93,ShSt94,Kapitanski94,BaSh98,Nakanishi99b,BaGe99,Tao05NY}. The works \cite{KeMe08,DuMe08} classify the dynamics of the focusing equation below and at the energy threshold $E[W]$ (see also \cite{DuMe09b}).  For the classification of the dynamics of solutions with energy $E[u]<E[W]+\eps$, see \cite{KrSc07,DuKeMe11a,DuKeMe12,KrNaSc13a,KrNaSc13b,KrNaSc13P}. For examples of nonscattering bounded solutions of \eqref{CP} in this energy range, see e.g \cite{KrScTa08, HiRa12, DoKr13, KrSc12P} and references therein. The works \cite{DuKeMe12b,DuKeMe13,DuKeMe13Pa} classify the dynamics of large energy solutions. Finally, we would like to point out that the exterior energy estimates and the channel of energy argument were also used in the context of wave maps \cite{KeLaSc13P,CoKeLaSc12Pa,CoKeLaSc12Pb,Cote13P} and subcritical or 
supercritical wave equations \cite{
DuKeMe12c,Shen12P}.

\begin{notations}
If $N$ is an integer and $R>0$, we will denote by $B^N(R)$ be the ball of $\RR^{N}$, centered at the origin with radius $R$. 

We let $\OOO_N=\OOO_N(\RR)$ be the orthogonal group, which is the group of $N\times N$ real orthogonal matrices such that $A^{T}A$ is the identity matrix, and by $\SSS\OOO_N$ the special orthogonal group, i.e. the subset of $A\in \OOO_N$ such that $\det A=1$. 

We denote by $\SSS=\SSS(\RR^N)$ the space of Schwartz functions on $\RR^N$.
\end{notations}

\section{Preliminaries on well-posedness}
\label{S:prelim}
In this subsection, we recall the exact definition of a \emph{solution} of \eqref{CP} and give some sufficient conditions for a function $u$ to be a solution. These conditions will be used essentially in Section \ref{S:Lorentz} on the Lorentz transformation. The classical Cauchy theory for \eqref{CP} uses the space $L^{\frac{N+2}{N-2}}(\RR,L^{\frac{2(N+2)}{N-2}}(\RR^N))$. We will rather use the Cauchy theory developped in \cite{KeMe08}, based on the space $L^{\frac{2(N+1)}{N-2}}(\RR^{N+1})$ which is invariant by Lorentz transform. Let us emphasize that both Cauchy theories give the same definition of solution of \eqref{CP} (see Claim \ref{C:L5L10} below).

We first recall the Strichartz estimates \cite{GiVe95}. 
Let $I$ be an open interval containing $0$ and $(w_0,w_1)\in \hdot\times L^2$. Let
$$ w(t)=\cos(t\sqrt{-\Delta})w_0+\frac{\sin(t\sqrt{-\Delta})}{\sqrt{-\Delta}}w_1+\int_0^t\frac{\sin\left( (t-s)\sqrt{-\Delta} \right)}{\sqrt{-\Delta}}h(s)\,ds,\quad t\in I.$$
Then, if $D_x^{1/2}h\in L^{\frac{2(N+1)}{N+3}}(I\times \RR^N)$, we have
\begin{multline}
 \label{Strichartz1}
 \sup_{t\in I} \left\|(w,\partial_tw)(t)\right\|_{\hdot\times L^2}+\left\|D_x^{1/2} w\right\|_{L^{\frac{2(N+1)}{N-1}}(I\times \RR^N)}+\|w\|_{L^{\frac{2(N+1)}{N-2}}(I\times \RR^N)}+\|w\|_{L^{\frac{N+2}{N-2}}(I,L^{\frac{2(N+2)}{N-2}}(\RR^N))}\\
 \leq C\left(  \left\|(w_0,w_1)\right\|_{\hdot\times L^2}+\left\|D_x^{1/2} h\right\|_{L^{\frac{2(N+1)}{N+3}}(I\times \RR^N)}\right),
\end{multline} 
and if $h\in L^{1}(I,L^2(\RR^N))$
\begin{multline}
 \label{Strichartz2}
 \sup_{t\in I} \left\|(w,\partial_tw)(t)\right\|_{\hdot\times L^2}+\left\|D_x^{1/2} w\right\|_{L^{\frac{2(N+1)}{N-1}}(I\times \RR^N)}+\|w\|_{L^{\frac{2(N+1)}{N-2}}(I\times \RR^N)}+\|w\|_{L^{\frac{N+2}{N-2}}(I,L^{\frac{2(N+2)}{N-2}}(\RR^N))}\\
 \leq C\left(  \left\|(w_0,w_1)\right\|_{\hdot\times L^2}+\left\|h\right\|_{L^{1}\left(I,L^2(\RR^N)\right)}\right).
\end{multline} 
\begin{defi}
\label{D:solution}
 Let $I$ be an open interval containing $0$, and $(u_0,u_1)\in \hdot\times L^2$. We say that $u$ is a solution of \eqref{CP} in $I$ if 
 $$(u,\partial_tu) \in C^0(I,\hdot\times L^2),\; u\in L^{\frac{2(N+1)}{N-2}}_{\loc}\left( I,L^{\frac{2(N+1)}{N-2}}\left( \RR^N \right) \right),\; D_x^{1/2}u\in L^{\frac{2(N+1)}{N-1}}_{\loc}\left( I,L^{\frac{2(N+1)}{N-1}}\left( \RR^N \right) \right)$$
 and 
 \begin{equation}
  \label{Duhamel}
  u(t)=\cos\left( t\sqrt{-\Delta} \right)u_0+\frac{\sin\left( t\sqrt{-\Delta} \right)}{\sqrt{-\Delta}}u_1+\int_0^t\frac{\sin\left( (t-s)\sqrt{-\Delta} \right)}{\sqrt{-\Delta}}|u|^{\frac{4}{N-2}}u(s)\,ds.
 \end{equation} 
\end{defi}
Recall from \cite{KeMe08} that for any initial data $(u_0,u_1)\in \hdot\times L^2$, there is a unique solution $u$ of \eqref{CP} defined on a maximal interval of definition $I_{\max}(u)=(T_-(u),T_+(u))\subset \RR$, that satisfies the blow-up criterion 
\begin{equation}
\label{Bup_criterion}
T_+(u)<\infty\Longrightarrow \|u\|_{L^{\frac{2(N+1)}{N-2}}\left((0,T_+(u))\times \RR^N\right)}=+\infty.
\end{equation} 
More precisely, if $u\in L^{\frac{2(N+1)}{N-2}}\left((0,T_+(u))\times \RR^N\right)$, then $T_+(u)=+\infty$ and $u$ scatters, in $\hdot\times L^2$, to a linear solution as $t\to +\infty$. 

Note that by the Strichartz estimate \eqref{Strichartz1}, any solution $u$ belongs to the space 
$$L^{\frac{N+2}{N-2}}_{\loc}\left(I_{\max}(u),L^{\frac{2(N+2)}{N-2}}(\RR^N)\right).$$ 
Since by H\"older and Sobolev inequalities
$$\|u\|_{L^{\frac{2(N+1)}{N-2}}(I\times\RR^N)}^{\frac{2(N+1)}{N-2}}\leq C\sup_{t\in I}\|(u(t),\partial_t u(t))\|_{\hdot\times L^2}^{\frac{N}{N-2}}\|u\|_{L^{\frac{N+2}{N-2}}(I,L^{\frac{2(N+2)}{N-2}})}^{\frac{N+2}{N-2}},$$
we also have, in view of the Strichartz estimate \eqref{Strichartz2}, the following variants of the blow-up and scatttering criteria:
\begin{equation}
\label{Bup_criterion2}
T_+(u)<\infty\Longrightarrow \|u\|_{L^{\frac{N+2}{N-2}}\left((0,T_+(u)),L^{\frac{2(N+2)}{N-2}}\left(\RR^N\right)\right)}=+\infty,
\end{equation} 
and if $u\in L^{\frac{N+2}{N-2}}\left((0,T_+(u)), L^{\frac{2(N+2)}{N-2}}\left(\RR^N\right)\right)$, then $T_+(u)=+\infty$ and $u$ scatters to a linear solution as $t\to +\infty$. 
\begin{remark}
 \label{R:distrib}
 Let $u$ be a solution of \eqref{CP} in the sense of Definition \ref{D:solution}. If $(u_0,u_1)\in C_0^{\infty}(\RR^N)$, then $u\in C^{\infty}(I_{\max}\times \RR^N)$ and is a classical solution of \eqref{CP}. By the Strichartz estimates and a density argument, one can check that if $(u_0,u_1)$ is a general element of $\hdot\times L^2$, the corresponding solution $u$ satisfies $(\partial_t^2-\Delta)u=|u|^{\frac{4}{N-2}}u$ in the sense of distributions.
\end{remark}

We next give three sufficient conditions for a function $u$ to be a solution.
The first one is \cite[Remark 2.14]{KeMe08} and we omit the proof.
\begin{claim}
 \label{C:solution_sequence} Let $u\in L^{\frac{2(N+1)}{N-2}}(I\times \RR^N)$ be such that $(u,\partial_tu)\in C^0(\hdot\times L^2)$. Assume that there exists a sequence $(u_k)$ of solutions of \eqref{CP} such that
 \begin{gather*}
 \sup_{t\in I}\|(u-u_k,\partial_tu-\partial_tu_k)(t)\|_{\hdot\times L^2}\underset{k\to\infty}{\longrightarrow}0\\
  \sup_{k}\|u_k\|_{L^{\frac{2(N+1)}{N-2}}(I\times \RR^N)}<\infty.
 \end{gather*} 
 Then $u$ is a solution of \eqref{CP}.
\end{claim}
\begin{claim}
\label{C:L5L10}
 Let $I$ be an open interval containing $0$, and $(u_0,u_1)\in \hdot\times L^2$. Assume that $u\in L^{\frac{N+2}{N-2}}_{\loc}(I,L^{\frac{2(N+2)}{N-2}}(\RR^N))$ satisfies the integral equation \eqref{Duhamel}. Then $u$ is a solution of \eqref{CP}. 
\end{claim}
\begin{proof}
By the definition of a solution, it is sufficient to check:
$$ u\in L^{\frac{2(N+1)}{N-2}}_{\loc}\left(I,\RR^N\right),\quad D_x^{1/2}u\in L^{\frac{2(N+1)}{N-1}}_{\loc}\left(I,\RR^N\right).$$
Since by our assumptions on $u$, $|u|^{\frac{4}{N-2}}u\in L^1_{\loc}\left(I,L^2(\RR^N)\right)$, this follows immediately from the Strichartz estimate \eqref{Strichartz2}.
\end{proof}
We next prove that a solution of \eqref{CP} in the distributional sense, that satisfies an appropriate space-time bound, is also a solution of \eqref{CP} in the sense of Definition \ref{D:solution}.
\begin{lemma}
 \label{L:sol_distrib}
 Let $(u_0,u_1)\in \hdot\times L^2$, $I$ be an open interval such that $0\in I$, 
 $$u\in L^{\frac{N+2}{N-2}}_{\loc}\left(I,L^{\frac{2(N+2)}{N-2}}(\RR^N)\right)\text{ and }(u,\partial_tu)\in C^0\left(I,\hdot\times L^2\right).$$
 Assume furthermore $(u,\partial_tu)_{\restriction t=0}=(u_0,u_1)$ and 
 \begin{equation}
  \label{distrib_u}
  \partial_t^2u-\Delta u=|u|^{\frac{4}{N-2}}u\text{ in }\DDD'(I\times \RR^N).
 \end{equation} 
 Then $u$ is a solution of \eqref{CP}.
\end{lemma}
\begin{proof}
 In view of Claim \ref{C:L5L10}, it suffices to check that $u$ satisfies the integral equation \eqref{Duhamel}. Let $I_+=I\cap (0,+\infty)$.  We prove \eqref{Duhamel} for $t\in I_+$, the proof of \eqref{Duhamel} for $t\in I\cap (-\infty,0)$ is exactly the same. Let
 $$v(t)=u(t)-\cos(t\sqrt{-\Delta})u_0-\frac{\sin(t\sqrt{-\Delta})}{\sqrt{-\Delta}}u_1 .$$
 Then $(v,\partial_t v)\in C^0(I,\hdot\times L^2)$, $(v,\partial_tv)_{\restriction t=0}=(0,0)$, and
 \begin{equation}
  \label{distrib_v}
  \partial_t^2v-\Delta v=|u|^{\frac{4}{N-2}}u\text{ in }\DDD'(I\times \RR^N).
 \end{equation} 
 Let $h\in C^{\infty}_0(I_+\times \RR^N)$. Let, for $t\in \RR$, 
 $$ H(t)=-\int_t^{+\infty}\frac{\sin\left( (t-s)\sqrt{-\Delta} \right)}{\sqrt{-\Delta}}h(s)\,ds,$$
 so that $H\in C^{\infty}\left(\RR^{N+1}\right)$ (with compact support in $x$), $H(t)=0$ for large $t$ and $\partial_t^2H-\Delta H=h$.
 Let $\varphi\in C^{\infty}(\RR)$ such that $\varphi(\sigma)=1$ if $\sigma\geq 1$, and $\varphi(\sigma)=0$ if $\sigma \leq \frac{1}{2}$.  If $a\in (0,1]$, we let
 $$H^a(t,x)=\varphi\left( \frac{t}{a} \right)H(t,x).$$
 Note that $H^a\in C^{\infty}_0\left(\RR^{N+1}\right)$. By \eqref{distrib_v},
 \begin{equation}
  \label{P6}
  \iint_{\RR^{N+1}} v(t,x)(\partial_t^2-\Delta)H^a(t,x)\,dtdx=\iint_{\RR^{N+1}} |u|^{\frac{4}{N-2}}u(t,x)H^a(t,x)\,dtdx.
 \end{equation} 
 By dominated convergence and Fubini's Theorem, 
 \begin{multline*}
  \lim_{a\to 0} \iint_{\RR^{N+1}} |u|^{\frac{4}{N-2}}u(t,x)H^a(t,x)\,dtdx=\int_{\RR^N}\int_0^{+\infty} |u|^{\frac{4}{N-2}}u(t,x)H(t,x)\,dtdx\\
 \qquad\qquad\qquad\qquad\qquad\quad=-\int_0^{+\infty} \int_0^s \int_{\RR^N} |u|^{\frac{4}{N-2}}u(t,x)\frac{\sin\left( (t-s)\sqrt{-\Delta} \right)}{\sqrt{-\Delta}}h(s,x)\,dx\,dt\,ds\\
 =-\int_{\RR^N} \int_0^{+\infty}\int_0^s \frac{\sin\left( (t-s)\sqrt{-\Delta} \right)}{\sqrt{-\Delta}}|u|^{\frac{4}{N-2}}u(t,x)\,dt\,h(s,x)\,ds\,dx,
 \end{multline*}
where at the last line we have also used the self-adjointness of $\frac{\sin\left( (t-s)\sqrt{-\Delta} \right)}{\sqrt{-\Delta}}$. As a conclusion, the right hand-side of \eqref{P6} satisfies:
\begin{multline}
 \label{P7}
  \lim_{a\to 0} \iint_{\RR^{N+1}} |u|^{\frac{4}{N-2}}u(t,x)H^a(t,x)\,dtdx\\
  =\int_{\RR^N} \int_0^{+\infty}\int_0^t \frac{\sin\left( (t-s)\sqrt{-\Delta} \right)}{\sqrt{-\Delta}}|u|^{\frac{4}{N-2}}u(s,x)\,ds\,h(t,x)\,dt\,dx.
\end{multline} 
We next consider the left-hand side of \eqref{P6}:
 \begin{multline}
  \label{P8}
   \iint_{\RR^{N+1}} v(t,x)(\partial_t^2-\Delta)H^a(t,x)\,dtdx\\
   =\iint_{\RR^{N+1}}v(t,x)\left( \frac{1}{a^2}\varphi''\left( \frac{t}{a} \right)H(t,x)+\frac{2}{a}\varphi'\left( \frac{t}{a} \right)\partial_tH(t,x)+\varphi\left( \frac{t}{a}\right)h(t,x)  \right)\,dtdx.
 \end{multline}
 Assume that we have proved:
 \begin{gather}
  \label{P9}
  \lim_{a\to 0}\iint_{\RR^{N+1}}v(t,x)\frac{1}{a^2}\varphi''\left( \frac{t}{a} \right)H(t,x)\,dtdx=0\\
 \label{P10}
  \lim_{a\to 0}\iint_{\RR^{N+1}}v(t,x)\frac{1}{a}\varphi'\left( \frac{t}{a} \right)\partial_tH(t,x)\,dtdx=0.
 \end{gather}
Then, by \eqref{P6}, \eqref{P7} and \eqref{P8}
\begin{equation}
 \label{P11}\int_{\RR^N} \int_0^{+\infty}\int_0^t \frac{\sin\left( (t-s)\sqrt{-\Delta} \right)}{\sqrt{-\Delta}}|u|^{\frac{4}{N-2}}u(s,x)\,ds\,h(t,x)\,dt\,dx=
 \int_{\RR^N} \int_0^{+\infty} v(t,x)h(t,x)\,dt\,dx.
\end{equation} 
Since $h$ is arbitrary in $C^{\infty}_0\left(  I_+\times \RR^N\right)$,  we deduce, in view of the definition of $v$, the desired integral formula \eqref{Duhamel}.

It remains to check \eqref{P9} and \eqref{P10}. We only prove \eqref{P9}, the proof of \eqref{P10} is similar. Using that $\partial_t v\in C^0(I,L^2)$ and $v_{\restriction t=0}=\partial_tv_{\restriction t=0}=0$ almost everywhere, we deduce 
$$\forall t\in I,\quad v(t)\in L^2(\RR^N)\text{ and }\lim_{t\to 0}\frac{1}{t}\|v(t)\|_{L^2}=0.$$
Let $\eps>0$ and $a_0$ such that $\|v(t)\|_{L^2}\leq \eps t$ for $t\in (0,a_0]$. Then (using that $\varphi''(t/a)=0$ for $t\geq a$ or $t \leq 0$), 
$$\left|\iint_{\RR^{N+1}}v(t,x)\frac{1}{a^2}\varphi''\left( \frac{t}{a} \right)H(t,x)\,dtdx\right|\leq C\int_0^a \frac{\eps t}{a^2}\,dt\leq C\eps,$$
which concludes the proof of \eqref{P9}, and thus of Lemma \ref{L:sol_distrib}.
\end{proof}

\section{Properties of stationary solutions}
\label{S:stationary}
This section concerns the set $\Sigma$ of non-zero stationary solutions of \eqref{CP}. More precisely, in \ref{SS:Kelvin_transfo}, we give the asymptotics, for large $x$, of an element $Q$ of $\Sigma$. We also study the set $\MMM$ of transformations, mentioned in the introduction, leaving $\Sigma$ invariant. Subsection \ref{SS:linearized} concerns the linearized operator $L_Q$. Finally, in \ref{SS:modulation}, under the nondegeneracy assumption \eqref{ND}, we choose modulation parameters in $\MMM$ in order to satisfy some orthogonality properties.
\subsection{Kelvin transformation and asymptotic behaviour}
\label{SS:Kelvin_transfo}
Recall that $\Sigma$ is the set of non-zero functions $Q$ in $\hdot(\RR^N)$ such that 
\begin{equation}
\label{eq02}
-\Delta Q=|Q|^{\frac{4}{N-2}}Q
 \end{equation} 
in the sense of distributions on $\RR^N$.

We fix an arbitrary one to one map $\zeta$ from $\{(i,j)\in \NN^2,\; 1\leq i<j\leq N\}$ to $\left\{1,2,\ldots, \frac{N(N-1)}{2}\right\}$. If $c=\left(c_1,\ldots,c_{\frac{N(N-1)}{2}}\right)\in \RR^{\frac{N(N-1)}{2}}$, we write
\begin{equation}
 \label{Pc}
P_c=\exp\left([p_{i,j}]_{1\leq i,j\leq N}\right) \in \SSS\OOO_N,
 \end{equation}
 where $p_{i,i}=0$, $p_{i,j}=c_{\zeta(i,j)}$ if $i<j$, $p_{i,j}=-c_{\zeta(j,i)}$ if $j<i$. This defines a parametrization of the special orthogonal group $\SSS\OOO_N$ by $\RR^{\frac{N(N-1)}{2}}$ in a neighborhood of the identity matrix.

Let $A=(s,a,b,c)\in \RR^{N'}=\RR\times \RR^N\times \RR^N\times \RR^{\frac{N(N-1)}{2}}$. We let, for $f\in \hdot$,
 \begin{equation}
 \label{def_thetaA}
 \theta_{A}(f)(x)=e^{\frac{(N-2)s}{2}} \left|\frac{x}{|x|} -a |x|\right|^{2-N}f\left( b+\frac{e^sP_c(x-a|x|^2)}{1-2\langle a,x\rangle +|a|^2|x|^2} \right).
\end{equation}
\begin{prop}
\label{P:A1}
Let $Q\in \Sigma$. Then
 \begin{enumerate}
  \item \label{I:Qsmooth}$Q\in C^{\infty}(\RR^N)$ if $N=3,4$ and $Q\in C^{4}(\RR^5)$ if $N=5$.
 \item\label{I:harmonic} 
We have:
 $$\forall \alpha\in \NN^N,\text{ s.t. }|\alpha|\leq 4,\;  \exists C_{\alpha}>0,\quad \left|\partial_{x}^{\alpha}Q(x)\right|\leq C_{\alpha}|x|^{-N+2-|\alpha|},\quad |x|\geq 1.$$
 \item \label{I:Kelvin}The function 
 $$\tQ: x\mapsto \frac{1}{|x|^{N-2}}Q\left(\frac{x}{|x|^2}\right)$$
 is also in $\Sigma$. Furthermore,
 $$\|\tQ\|_{\hdot}^2=\|\tQ\|_{L^{\frac{2N}{N-2}}}^{\frac{2N}{N-2}}=\|Q\|_{L^{\frac{2N}{N-2}}}^{\frac{2N}{N-2}}=\|Q\|_{\hdot}^2.$$
 \item \label{I:transformations}Let $A=(s,a,b,c)\in \RR^{N'}=\RR\times \RR^N\times \RR^N\times \RR^{\frac{N(N-1)}{2}}$. Then the function $\theta_A(Q)$ 
is in $\Sigma$.
\item \label{I:compo} If $A_1,A_2\in B^{N'}(\eps)$ ($\eps>0$ small), then 
$$ \theta_{A_1}\circ\theta_{A_2}=\theta_{A_3},\quad (\theta_{A_1})^{-1}=\theta_{A_4},$$
where $A_3,A_4\in \RR^{N'}$ 
and the maps $(A_1,A_2)\mapsto A_3$ and $A_1\mapsto A_4$ are $C^{\infty}$ from $\left(B^{N'}(\eps)\right)^2$ (respectively $B^{N'}(\eps)$) to a neighborhood of $0$ in $\RR^{N'}$.
 \end{enumerate}
\end{prop}
\begin{remark}
 \label{R:A2}
In the cases $N=3,4$ when the nonlinearity is smooth,  the estimates of point \eqref{I:harmonic} holds for all multi-index $\alpha$. Furthermore, one can adapt the proof of this estimate to prove
\begin{equation}
\label{limite}
Q(x) =\frac{1}{|x|^{N-2}}P\left( \frac{x}{|x|^2} \right)+\OOO\left(\frac{1}{|x|^{k+N-1}}\right),\quad |x|\to \infty,
\end{equation} 
where $P$ is a homogeneous harmonic polynomial of degree $k\geq 0$. This polynomial can be a non-zero constant. In this case, $|x|^{N-2}Q(x)$ converges to some non-zero real number. This is the case of the explicit radial stationary solution $W$. If the degree of $P$ is positive, then $|x|^{N-2}Q(x)$ tends to $0$ as $|x|\to \infty$. The existence of solutions of \eqref{eq02} such that \eqref{limite} holds with nonconstant $P$ follows from the existence of changing sign solutions of \eqref{eq02}, proved in \cite{Ding86}, and the Kelvin transformation given  by \eqref{I:Kelvin}. To our knowledge, the existence of solutions of \eqref{eq02} such that  \eqref{limite} holds with $P$ of arbitrary degree is still open.
\end{remark}
\begin{remark}
 Point \eqref{I:transformations} of the proposition gives a parametrization of an open neighborhood of the identity in $\MMM$. Note that it includes space translations ($s=0$, $a=c=0$), scaling ($a=b=c=0$) and space rotations ($a=b=0$, $s=0$), as well as additional tranformations which can be constructed by conjugating space translations with the Kelvin transformation ($b=c=0$, $s=0$).  However, the Kelvin tranformation defined in \eqref{I:Kelvin} cannot be described by this parametrization. 
\end{remark}

By \cite{Trudinger68}, if $Q$ is an $\hdot$ solution of \eqref{CP}, then $Q$ is locally bounded. By Sobolev inequalities, point \eqref{I:Qsmooth} follows. The remainder of this subsection is devoted to points \eqref{I:harmonic}, \eqref{I:Kelvin}  and \eqref{I:transformations}.

\subsubsection{Kelvin transformation}
\label{SS:Kelvin}
We first prove:
\begin{lemma}
 \label{L:A3}
 Let $Q\in L^{\frac{2N}{N-2}}(\RR^N)\cap\hdot(\RR^N)$ such that \eqref{eq02} holds in the sense of distributions on $\RR^N\setminus \{0\}$. Then $Q\in (C^{\infty}\cap\hdot)(\RR^N)$ if $N=3,4$ and $Q\in  (C^{4}\cap\hdot)(\RR^5)$ if $N=5$. Furthermore $Q$ satisfies \eqref{eq02} in the classical sense on $\RR^N$.
\end{lemma}
\begin{proof}
By \eqref{I:Qsmooth}, it is sufficient to prove that $Q$ satisfies \eqref{eq02} in the sense of distributions on $\RR^N$. Let $\varphi\in C_0^{\infty}(\RR^N)$. 
 
 Let $\psi\in C_0^{\infty}(\RR^N)$ such that $\psi(x)=1$ is $|x|\leq 1$ and $\psi(x)=0$ is $|x|\geq 2$. Then
 \begin{multline}
 \label{divided}
 \int Q\Delta \varphi=\int Q\Delta\bigg[\left(\psi\left(\frac{x}{\eps}\right)+1-\psi\left( \frac{x}{\eps} \right)\right) \varphi(x) \bigg]\\
=-\int |Q|^{\frac{4}{N-2}}Q \left(1-\psi\left(\frac{x}{\eps}\right)\right) \varphi(x)\,dx+\int Q\Delta \left(\psi\left(\frac{x}{\eps}\right)\varphi(x)\right)\,dx
\end{multline}
where we have used, in the first integral of the last line, the fact that $Q$ satisfies \eqref{eq02} in the sense of distributions outside the origin. By the dominated convergence theorem,
$$\lim_{\eps\to 0^+}\int |Q|^{\frac{4}{N-2}}Q(x) \left(1-\psi\left(\frac{x}{\eps}\right)\right) \varphi(x)\,dx=\int |Q|^{\frac{4}{N-2}}Q(x) \varphi(x)\,dx.$$
Moreover
\begin{multline}
\label{Leibniz}
 \int Q(x)\Delta \left(\psi\left(\frac{x}{\eps}\right)\varphi(x)\right)\,dx=\int Q(x)\frac{1}{\eps^2}\Delta \psi\left(\frac{x}{\eps}\right)\varphi(x)\,dx\\
 +\frac{2}{\eps}\int Q(x)\nabla \psi\left( \frac{x}{\eps} \right)\cdot\nabla \varphi(x)\,dx+\int Q(x)\psi\left( \frac{x}{\eps} \right)\Delta\varphi(x)\,dx.
\end{multline}
We have
\begin{multline*}
 \left| \int Q(x)\frac{1}{\eps^2}\Delta\psi\left( \frac{x}{\eps} \right)\varphi(x)\,dx\right|\leq \frac{1}{\eps^2}\|Q\|_{L^{\frac{2N}{N-2}}}\left( \int  \left|\Delta \psi \left( \frac{x}{\eps} \right)\right|^{\frac{2N}{N+2}} |\varphi(x)|^{\frac{2N}{N+2}}\,dx\right)^{\frac{N+2}{2N}}\\
 \leq \|Q\|_{L^{\frac{2N}{N-2}}}\eps^{\frac{N}{2}-1}\left\|\Delta\psi\right\|_{L^{\frac{2N}{N+2}}}\|\varphi\|_{L^{\infty}}\underset{\eps\to 0}{\longrightarrow}0.
\end{multline*}
Bounding similarly the other terms in \eqref{Leibniz}, we get
$$\lim_{\eps\to 0}\int Q(x)\Delta\left( \psi\left( \frac{x}{\eps}\right)  \varphi(x)\right)\,dx=0,$$
and thus
$$ -\int Q\Delta\varphi=\int |Q|^{\frac{4}{N-2}}Q\varphi,$$
which shows as announced that $Q$ satisfies \eqref{eq02} in the sense of distributions on $\RR^N$. 
\end{proof}

Let us prove point \eqref{I:Kelvin} of Proposition \ref{P:A1}.

 We first note that the Kelvin transformation
 $$ \TTT: f\mapsto \frac{1}{|x|^{N-2}}f\left(\frac{x}{|x|^2}\right)$$
 is an isometry of $L^{\frac{2N}{N-2}}$ that satisfies, for any smooth function $f$,
 \begin{equation}
 \label{eq_Kelvin}
 \Delta (\TTT f)=  \frac{1}{|x|^{N+2}}(\Delta f) \left(\frac{x}{|x|^2}\right),\quad x\neq 0.
  \end{equation} 
  If $f\in C^{\infty}_0\left(\RR^N\setminus \{0\}\right)$, then $\TTT f\in  C^{\infty}_0\left(\RR^N\setminus \{0\}\right)$ and by integration by parts, 
$$\|\TTT f\|_{\hdot}^2=-\int \Delta(\TTT f) \TTT f=-\int \frac{1}{|x|^{2N}}(\Delta f) \left(\frac{x}{|x|^2}\right) f\left( \frac{x}{|x|^2}\right)\,dx =-\int \Delta f f=\|f\|^2_{\hdot},$$
where we have used that the Jacobian determinant of $x\mapsto \frac{x}{|x|^2}$ is $\frac{1}{|x|^{2N}}$.
Using the density of $C^{\infty}_0(\RR^N\setminus\{0\})$ in $\hdot$, we deduce that $\TTT$ is also an isometry of $\hdot$. 

Combining the preceding argument with Lemma \ref{L:A3}, we get that if $Q$ is a $\hdot$ solution of \eqref{eq02} on $\RR^N$, then $\tQ=\TTT Q$ is also a $\hdot$ solution of \eqref{eq02} on $\RR^N$. The equality $\int Q^{\frac{2N}{N-2}}=\int |\nabla Q|^2$ follows from a simple integration by parts, which concludes the proof of \eqref{I:Kelvin}.

 \subsubsection{Asymptotic behaviour}
 Let us prove point \eqref{I:harmonic} of Proposition \ref{P:A1}. 
%
Let $Q$ and $\tQ$ be as in the proposition. By \eqref{I:Qsmooth}  and \eqref{I:Kelvin}, $\tQ$ can be extended to a $C^{4}$ solution of \eqref{eq02}. As a consequence,
$$ |Q(x)|=\frac{1}{|x|^{N-2}}\left|\tQ\left(\frac{x}{|x|^2}\right)\right|\leq \frac{C}{|x|^{N-2}}.$$
More generally, writing for $|\alpha|\leq 4$
$$\partial_x^{\alpha}Q(x)=\sum_{\gamma+\beta=\alpha} \binom{\alpha}{\gamma}\partial_x^{\beta}\left( \frac{1}{|x|^{N-2}} \right)\partial_x^{\gamma} \left( \tQ\left(\frac{x}{|x|^2}\right) \right),$$
and using that $\partial_x^{\alpha}\tQ$ is locally bounded, we obtain the desired estimate.
\qed

\subsubsection{Transforms of stationary solutions}
It remains to prove point \eqref{I:transformations} of Proposition \ref{P:A1}.
Let $M$ be the group of one-to-one maps of $\RR^N\cup \{\infty\}$ generated by 
\begin{itemize}
 \item the translations $T_a:x\mapsto x+a$, where $a\in \RR^N$;
 \item the dilations $D_{\lambda}:x\mapsto \lambda x$, where $\lambda>0$;
 \item the linear isometries $P\in \OOO_N(\RR)$;
 \item the inversion $J: x\mapsto \frac{x}{|x|^2}$.
\end{itemize}
We adopt the conventions $T_a(\infty)=D_{\lambda}(\infty)=P(\infty)=J(0)=\infty$, $J(\infty)=0$.
If $\varphi\in M$ and $f\in \hdot$, we denote by 
$$\Theta_{\varphi}(f)=|\det \varphi'(x)|^{\frac{N-2}{2N}}f(\varphi(x)).$$
We note that $\Theta_{\varphi\circ\psi}=\Theta_{\psi}\circ\Theta_{\varphi}$ and
$$\Theta_{T_a}(f)(x)=f(x+a),\quad \Theta_{D_{\lambda}}(f)(x)=\lambda^{N/2-1}f(\lambda x),\quad \Theta_P(f)(x)=f(Px),$$
and that $\Theta_J(f)$ is the Kelvin transform of $f$. We deduce that $\left\{\Theta_{\varphi},\; \varphi\in M\right\}$ is exactly the group $\MMM$ of isometries of $\hdot$ generated by space translations, scaling, linear isometries and the Kelvin transform mentioned in the introduction. In view of point \eqref{I:Kelvin} of Proposition \ref{P:A1}, 
$$f\in \Sigma \Longrightarrow \Theta_\varphi(f)\in \Sigma.$$

We next prove that the transformations $\theta_A$ defined by \eqref{def_thetaA} are in $\MMM$. Letting
\begin{equation}
\label{def_vA}
\varphi_{A}(x)=b+\frac{e^sP_c(x-a|x|^2)}{1-2\langle a,x\rangle+|a|^2|x|^2},
\end{equation} 
we see that 
$
\Big|\det (\varphi_A'(x))\Big|=e^{Ns} \left|\frac{x}{|x|} -a |x|\right|^{-2N}.
$
As a consequence, for any $f\in \hdot$,
\begin{equation}
\label{theta_Theta}
 \theta_A(f)= |\det \varphi'_A(x)|^{\frac{N-2}{2N}}f(\varphi_A(x)) =\Theta_{\varphi_A}(f)
\end{equation} 
and thus that it is sufficient to show that $\varphi_A\in M$.  For this we notice that the function $\psi_a$ defined by 
$$\psi_a(x)=J \circ T_{-a}\circ J(x)=\frac{x-a|x|^2}{1-2\langle a,x\rangle+|a|^2|x|^2}$$
 is in $M$. Since 
\begin{equation}
\label{formula_phiA}
\varphi_A=T_b\circ P_c\circ D_{e^s}\circ \psi_a
\end{equation} 
 we obtain that $\varphi_A$ is an element of $M$, which  concludes the proof.
\qed
\subsubsection{Composition and inverse of the transformations}
It remains to prove point \eqref{I:compo} of Proposition \ref{P:A1}. We use the notations $T_a$, $D_{\lambda}$, $\psi_a$ of the preceding subsection. By direct computations, if $a,b\in \RR^N$, $P\in \OOO_N$, $\lambda>0$,
\begin{gather}
 \label{formula1}
 T_b\circ D_{\lambda}=D_{\lambda}\circ T_{\lambda^{-1}b},\quad T_b\circ P=P\circ T_{P^{-1}(b)},\quad P\circ D_{\lambda}=D_{\lambda}\circ P\\
 \label{formula2}
 \psi_a\circ D_{\lambda}=D_{\lambda}\circ \psi_{\lambda a},\quad \psi_a\circ P=P\circ \psi_{P^{-1}(a)}\\
 \label{formula3}
\psi_a\circ T_b(x)=T_{\beta}\circ M\circ D_{\mu}\circ \psi_{\alpha}, \text{ where } \\
\label{formula4}
\mu^{-1}=1+|a|^2|b|^2-2\langle a, b\rangle,\quad \alpha=\mu(a-|a|^2b),\\
\label{formula5}
\beta=\mu(b-|b|^2a),\quad M(x)=\mu^{-1}2 \langle \alpha,x\rangle \beta-2\langle b,x\rangle a+x.
 \end{gather}
 Note that $\mu$ is well-defined if $a\neq b/|b|^2$, which is the case if $|a|<1$ and $|b|<1$, and that $M\in \SSS\OOO_N$, as can be checked directly by computing $M^*M$. Moreover, it is easy to see that $(a,b)\mapsto (\alpha,\beta,\mu,M)$ is $C^{\infty}$ in a neighborhood of the origin of $\RR^{2N}$.
 
Let $A_j=(a_j,b_j,c_j,s_j)\in B^{N'}(\eps)$ ($j=1,2$), $A=(a,b,c,s)\in B^{N'}(\eps)$.
Then by \eqref{formula_phiA}, 
$$\varphi_{A_1}\circ\varphi_{A_2}=T_{b_1}\circ P_{c_1}\circ D_{e^{s_1}}\circ \psi_{a_1} \circ T_{b_2}\circ P_{c_2}\circ D_{e^{s_2}}\circ \psi_{a_2}$$
and
$$\varphi_{A}^{-1}= \psi_{-a}\circ D_{e^{-s}}\circ P_{-c}\circ T_{-b}.$$

Point \eqref{I:compo} then follows from formulas \eqref{formula1},\ldots, \eqref{formula5} and the fact that $c\mapsto P_c$ is a local diffeomorphism, in a neighborhood of the origin from $\RR^{\frac{N(N-1)}{2}}$ to $\OOO_N$.
\subsection{Properties of the linearized operator}
\label{SS:linearized}
This subsection concerns the linearized operator $L_Q$ around a non-zero stationary solution $Q$, and the quadratic form associated  to $L_Q$. In \ref{SS:coercivity}, we prove a coercivity property of this quadratic form and give some consequences. We then consider, in \ref{SSS:Null}, the vector space $\tZZZ_Q$ defined in the introduction. We finally give, in \ref{SSS:eigenfunctions} the precise asymptotics of an eigenfunction associated to a negative eigenvalue of $L_Q$. 
\subsubsection{Preliminaries and notations}
Let $Q\in \Sigma$. We denote by 
\begin{equation}
\label{defLQ}
L_Q=-\Delta -\frac{N+2}{N-2}|Q|^{\frac{4}{N-2}}
\end{equation} 
the linearized operator at $Q$, and
\begin{equation}
 \label{A2}
 \Phi_Q(f)=\frac{1}{2}\int |\nabla f|^2-\frac{N+2}{2(N-2)}\int |Q|^{\frac{4}{N-2}}f^2=\frac{1}{2}\int L_Qf\,f,
\end{equation} 
 the corresponding quadratic form, defined for $f\in \hdot\left( \RR^N \right)$.
 \begin{claim}
 \label{C:A6}
  Let $V$ be a subspace of $\hdot\left(\RR^N\right)$ such that 
  \begin{equation}
   \label{A3}
   \forall f\in V,\quad \Phi_Q(f)\leq 0.
  \end{equation} 
  Then $\dim V$ is finite.
 \end{claim}
\begin{proof}
 Indeed, by Proposition \ref{P:A1}, if $f\in V$, then $\|f\|_{\hdot}^2\leq C\int \frac{1}{1+|x|^4}|f(x)|^2\,dx$. 
By Hardy's inequality and Rellich-Kondrachov Theorem, the injection $$\hdot \longrightarrow L^2\left(\RR^N, \frac{1}{1+|x|^4}\,dx\right)$$
 is compact. Thus the unit ball of $V$ is compact, which proves the result.
\end{proof}
Since by Proposition \ref{P:A1} $|Q|^{\frac{4}{N-2}}(x)\leq \frac{C}{1+|x|^4}$, it is classical (see \cite[Section 8]{Davies95BO}) that $L_Q$ is a self-adjoint operator with domain $H^2\left( \RR^N \right)$. By \cite[Theorem 8.5.1]{Davies95BO} and Claim \ref{C:A6}, the essential spectrum of $L_Q$ is $[0,+\infty)$, and $L_Q$ has no positive eigenvalue and a finite number of negative eigenvalues. We will denote this eigenvalues by $-\omega_1^2,\ldots,-\omega_p^2$, where
$$ 0<\omega_1\leq \ldots\leq \omega_p,$$
and the eigenvalues are counted with their order of multiplicity. Note that $p\geq 1$ because $L_QQ=-\frac{4}{N-2}|Q|^{\frac{4}{N-2}}Q$. The spectrum of $L_Q$ is exactly $[0,+\infty)\cup \{-\omega_j^2\}_{j=1\ldots p}$.
Let us consider an orthonormal family $(Y_j)_{j=1\ldots p}$ of eigenvectors of $L_Q$ corresponding to the eigenvalues $-\omega_j^2$:
\begin{equation}
 \label{A4}
L_QY_j=-\omega_j^2Y_j,\quad 
 \int_{\RR^N} Y_jY_k=\delta_{jk}=
 \begin{cases}
 0 &\text{ if }j\neq k\\
 1 &\text{ if }j=k.
 \end{cases}
 \qquad
\end{equation} 
By elliptic regularity, these functions are $C^3$ ($C^{\infty}$ if $N=3$ or $N=4$). It is well-known that they are  exponentially decreasing at infinity (see Proposition \ref{P:eigenfunction} below for their precise asymptotics).

The min-max principle implies
\begin{equation}
 \label{A5}
 \forall f\in H^1(\RR^N),\quad \int Y_1f=\int Y_2f=\ldots=\int Y_pf=0\Longrightarrow \Phi_Q(f)\geq 0.                                                                                                                                                                                                                    \end{equation} 
 Let 
 \begin{equation}
  \label{defZQ}
  \ZZZ_Q=\left\{ f\in \hdot(\RR^N),\text{ s.t. } L_Qf=0\right\}.
 \end{equation} 
Note that the elements of $\ZZZ_Q$ are not assumed to be in $L^2$. 
By Claim \ref{C:A6}, $\ZZZ_Q$ is finite dimensional. Let $(Z_j)_{j=1\ldots m}$ be a basis of $\ZZZ_Q$. We have
\begin{equation}
 \label{A6}
 \forall j=1\ldots m,\; \forall k=1\ldots p,\quad \int Z_jY_k=0.
\end{equation} 
Since the functions $Z_1,\ldots,Z_m,Y_1,\ldots,Y_p$ are linearly independent, one can find, by an elementary linear algebra argument, $E_1,\ldots,E_m\in C_0^{\infty}(\RR^N)$ such that
\begin{equation}
 \label{A7}
 \forall j=1\ldots m,\; \forall k=1\ldots p, \;\int E_jY_k=0,\qquad \forall j,k=1\ldots m,\; \int E_jZ_k=\delta_{jk}.
\end{equation} 
\subsubsection{A coercivity property}
\label{SS:coercivity}
In this part we prove  the following positivity property of $L_Q$:
\begin{prop}
\label{P:A7}
Let $(Y_k)_{k=1\ldots p}$, $(E_j)_{j=1\ldots m}$  be as above. There exists a constant $\tc>0$ with the following property. If $f\in \hdot(\RR^N)$ and 
\begin{equation}
 \label{A8}
 \forall k=1\ldots p,\;\int Y_kf=0\quad \text{and}\quad \forall j=1\ldots m,\; \int E_j f=0
\end{equation} 
then
\begin{equation}
 \label{A9}
 \Phi_Q(f)\geq \tc\|f\|_{\hdot}^2.
\end{equation} 
\end{prop}
We will also prove the following consequences of Proposition \ref{P:A7}:
\begin{corol}
 \label{C:A16}
 There are constants $\eps_0,C>0$ with the following property. Let $S\in \Sigma$ such that
 \begin{equation}
  \label{A33}
  \|S-Q\|_{\hdot}<\eps_0.
 \end{equation}
Then
\begin{equation}
 \label{A34}
  \|S-Q\|_{\hdot}\leq C \sum_{i=1}^m \left|\int (S-Q)E_i\right|.
\end{equation} 
Furthermore, if $A\in \RR^{N'}$ is small,
\begin{equation}
\label{thetaA_Q}
\|\theta_A(Q)-Q\|_{\hdot}\leq C|A|,
 \end{equation} 
 where the transformation $\theta_A$ is defined in \eqref{def_thetaA}.
\end{corol}

\begin{proof}[Proof of Proposition \ref{P:A7}]
 The proof is quite standard, we give it for the sake of completeness.
 
 \EMPH{Step 1} We show that for all $f\in \hdot(\RR^N)$,
 \begin{equation}
  \label{A10}
 \int Y_1f=\ldots=\int Y_pf=0\Longrightarrow \Phi_Q(f)\geq 0.
 \end{equation} 
 Indeed, by \eqref{A5}, \eqref{A10} holds if $f\in H^1(\RR^N)$. Assume that $f$ is in $\hdot$ but not in $L^2$, and that the orthogonality conditions in the left-hand side of \eqref{A10} hold. Let $\chi \in C_0^{\infty}(\RR^N)$ such that $\chi(x)=1$ if $|x|\leq 1$ and $\chi(x)=0$ if $|x|\geq 2$. Let 
 $$f_{\eps}(x)=\chi(\eps x)f(x),$$
 so that $f_{\eps}\in H^1(\RR^N)$. Then
 $$ f_{\eps}=g_{\eps}+\sum_{k=1}^p\alpha_{k\eps}Y_k,\quad\text{where for all }k, \; \alpha_{k\eps}=\int f_{\eps}Y_k,\quad \int g_{\eps}Y_k=0.$$
 We have
 $$ |\alpha_{k\eps}|=\left|\int f_{\eps}Y_k\right|=\left|\int \big(\chi(\eps x)-1\big)f(x)Y_k(x)\,dx\right|\leq C\int_{|x|\geq 1/\eps}  \left|f(x)Y_k(x)\right|\,dx\underset{\eps\to 0}{\longrightarrow} 0.$$
By the definition of $g_{\eps}$ and the fact that \eqref{A10} holds in $H^1$, we have $\Phi_Q(g_{\eps})\geq 0$. Thus
\begin{equation*}
 \Phi_Q(f_{\eps})=\Phi_Q(g_{\eps})+\sum_{k=1}^p \alpha_{k\eps}^2 \Phi_Q(Y_k)\geq \sum_{k=1}^p \alpha_{k\eps}^2 \Phi_Q(Y_k)\underset{\eps\to 0}{\longrightarrow} 0.
\end{equation*} 
Since 
$$ \Phi_Q(f_{\eps})=\frac{1}{2}\int \left|\nabla\left( \chi(\eps x)f(x) \right) \right|^2\,dx-\frac{N+2}{2(N-2)} \int \left( \chi(\eps x) \right)^2f^2(x)|Q|^{\frac{4}{N-2}}(x)\,dx \underset{\eps\to 0}{\longrightarrow} \Phi_Q(f),$$
we obtain as announced $\Phi_Q(f)\geq 0$.

\EMPH{Step 2} We show that for all $f\in \hdot(\RR^N)$,
\begin{equation}
 \label{A11}
 \int fY_1=\ldots=\int fY_p=\int fE_1=\ldots=\int fE_m=0\Longrightarrow \left( f=0\text{ or }\Phi_Q(f)>0\right).
\end{equation} 
Indeed, let 
$$ H=\Big\{g\in \hdot\text{ s.t. } \int gY_1=\ldots=\int gY_p=0\Big\}.$$
We first prove
\begin{equation}
\label{pourA11}
\Big(f\in H\text{ and } \Phi_Q(f)=0\Big)\Longrightarrow f\in \ZZZ_Q.
\end{equation} 

Let $f\in H$ such that $\Phi_Q(f)=0$.
Denoting also by $\Phi_Q$ the bilinear form 
$$\Phi_Q(f,g)=\frac 12 \int \nabla f \nabla g -\frac {N+2}{2(N-2)}\int |Q|^{\frac{4}{N-2}} fg,$$
we get by Cauchy-Schwarz for $\Phi_Q$ (using that by Step 1, $\Phi_Q$ is nonnegative on $H$),
\begin{equation}
 \label{A13}
 \forall h\in H,\quad \Phi_Q(f,h)=0.
\end{equation} 
Let $g\in \hdot$, and write $g=h+\sum_{k=1}^p\beta_k Y_k$, with $h\in H$ and $\beta_k=\int gY_k$. Then
$$\Phi_Q(f,g)=\underbrace{\Phi_Q(f,h)}_{0\text{ by }\eqref{A13}}+\sum_{k=1}^p \beta_k\underbrace{\Phi_Q(f,Y_k)}_{0\text{ since }f\in H}.$$
In particular,
\begin{equation}
 \label{A14}
 \forall g\in C_0^{\infty}(\RR^N),\quad \int fL_Qg=0,
\end{equation} 
i.e. $L_Qf=0$ in the sense of distribution. Thus $f\in \ZZZ_Q$. Hence \eqref{pourA11}.

Combining \eqref{pourA11} with the definition of $E_1$,\ldots,$E_m$, we obtain
\begin{equation*}
\left(f\in H,\;\Phi_Q(f)=0\text{ and } \int fE_1=\ldots=\int fE_m=0 \right)\Longrightarrow f=0
\end{equation*} 
and \eqref{A11} follows (using again that $\Phi_Q$ is nonnegative on $H$).

\EMPH{Step 3} We conclude the proof of Proposition \ref{P:A7}, arguing by contradiction and using a standard compactness argument. If the conclusion of the proposition does not hold, there exists a sequence $\{f_n\}_n$ in $\hdot$ such that
\begin{equation}
\label{prop_fn}
\left\{
\begin{gathered}
 \forall n,\; \forall k\in \{1,\ldots,p\},\; \forall j\in \{1,\ldots,m\},\quad \int f_nY_k=\int f_nE_j=0\\
 \forall n,\quad 0<\Phi_Q(f_n)\leq \frac{1}{n}\text{ and }\|f_n\|_{\hdot}=1.
\end{gathered}\right.
\end{equation}
Extracting a subsequence, we can assume
\begin{equation}
 \label{A19}
 f_n\xrightharpoonup[n\to \infty] f\text{  weakly in }\hdot.
\end{equation} 
In particular, $\int |\nabla f|^2\leq\limsup_{n\to\infty} \int |\nabla f_n|^2$. Furthermore, using that by Proposition \ref{P:A1}, $\lim_{|x|\to\infty}|x|^2|Q|^{\frac{4}{N-2}}(x)=0$ we get by Hardy's inequality and Rellich-Kondrachov Theorem: 
$$\lim_{n\to\infty}\int |Q|^{\frac{4}{N-2}}f_n^2=\int |Q|^{\frac{4}{N-2}} f^2.$$
Combining with \eqref{prop_fn}, we obtain
\begin{equation}
 \label{A20}
 \Phi_Q(f)\leq 0.
\end{equation} 
Since by \eqref{prop_fn} and \eqref{A19} 
$$\int fY_1=\ldots=\int fY_p=\int f E_1=\ldots =\int f E_m=0,$$
we deduce by Step 2 that $f=0$. As a consequence, $\lim_{n\to\infty}\int |Q|^{\frac{4}{N-2}}f_n^2=0$. Since $0<\Phi_Q(f_n)\leq 1/n$, we obtain $\lim_{n\to \infty}\int |\nabla f_n|^2=0$ which contradicts the equality $\|f_n\|_{\hdot}=1$ in \eqref{prop_fn}. The proof is complete.
\end{proof}

\begin{proof}[Proof of Corollary \ref{C:A16}]
 In all the proof, $C>0$ is a large, positive constant, depending only on $Q$ and the choice of $Z_1,\ldots,Z_m,E_1,\ldots,E_m$ and that may change from line to line. Let
 \begin{gather}
  \label{A35}
  g=S-Q,\quad \alpha_i=\int g\,Y_i,\quad  \beta_j =\int g\,E_j,\quad i=1,\ldots, p,\; j=1,\ldots, m\\
 \label{A36}
 h=g-\sum_{i=1}^p\alpha_j Y_j-\sum_{j=1}^m \beta_jZ_j.
 \end{gather}
Note that
\begin{gather}
 \label{A37}
 \int hY_i=0,\quad \int h E_j=0,\quad i=1,\ldots,p,\; j=1,\ldots,m.\\
\label{A38}
 \|g\|_{\hdot}\leq \|h\|_{\hdot}+C\sum_{i=1}^p |\alpha_i|+C\sum_{j=1}^m |\beta_j|.
 \end{gather} 
 Furthermore,
 \begin{equation}
  \label{A38'}
  -\Delta g=|S|^{\frac{4}{N-2}}S-|Q|^{\frac{4}{N-2}}Q=|Q+g|^{\frac{4}{N-2}}(Q+g)-|Q|^{\frac{4}{N-2}}Q=\frac{N+2}{N-2}|Q|^{\frac{4}{N-2}}g+R_{Q}(g),
 \end{equation} 
 where 
 \begin{equation}
  \label{defRQ0}
 R_{Q}(g)=|Q+g|^{\frac{4}{N-2}}(Q+g)-|Q|^{\frac{4}{N-2}}Q-\frac{N+2}{N-2}|Q|^{\frac{4}{N-2}}g
 \end{equation} 
satisfies the pointwise bound
\begin{equation}
 \label{pointwise_RQ}
\left|R_Q(g)\right|\leq C\left(|Q|^{\frac{6-N}{N-2}}|g|^2+|g|^{\frac{N+2}{N-2}}\right).
\end{equation} 
By \eqref{pointwise_RQ}, if $\|g\|_{L^{\frac{2N}{N-2}}}\leq 1$ (which holds by \eqref{A33} if $\eps_0$ is small enough),
\begin{equation}
  \label{A39}
  \|R_{Q}(g)\|_{L^{\frac{2N}{N+2}}}\leq C\|g\|_{L^{\frac{2N}{N-2}}}^2\leq C\|g\|^2_{\hdot}.
 \end{equation} 
 By \eqref{A38'}, $L_{Q}(g)=R_{Q}(g)$, and thus, by the definition \eqref{A36} of $h$,
 \begin{equation}
  \label{A40}
  L_{Q}h+\sum_{j=1}^p \alpha_j\omega_j^2Y_j=R_{Q}g.
 \end{equation} 
 Multiplying \eqref{A40} by $Y_j$ and integrating over $\RR^N$, we get, using also \eqref{A37} and \eqref{A39},
 \begin{equation}
  \label{A41}
  \forall j=1,\ldots,p,\quad |\alpha_j|\leq C\|g\|^2_{\hdot}.
 \end{equation} 
 Multiplying \eqref{A40} by $h$ and integrating over $\RR^N$, we obtain, using \eqref{A37}, \eqref{A39} and Proposition \ref{P:A7}, $\|h\|^2_{\hdot}\leq C\|g\|^2_{\hdot}\|h\|_{L^{\frac{2N}{N-2}}}$ and thus
 \begin{equation}
  \label{A42}
  \|h\|_{\hdot}\leq C\|g\|_{\hdot}^2.
 \end{equation} 
 Combining \eqref{A41} and \eqref{A42} with \eqref{A38}, we deduce
 $$ \|g\|_{\hdot}\leq C\|g\|_{\hdot}^2+C\sum_{i=1}^m |\beta_i|\leq C\eps_0\|g\|_{\hdot}+C\sum_{i=1}^m |\beta_i|,$$
 and thus, if $\eps_0>0$ is chosen small enough, the conclusion \eqref{A34} of the corollary. 
 
 It remains to prove \eqref{thetaA_Q}.  By \eqref{A34},
 $$\|\theta_{\vA}(Q)-Q\|_{\hdot}\leq C\sum_{i=1}^m \left|\int \left(\theta_{\vA}(Q)-Q\right)E_i\right|\leq C\sum_{i=1}^m \left|\int Q\left( \left(\theta_{\vA}\right)^*(E_i)-E_i\right)\right|,$$
 and the conclusion follows from Lemma \ref{L:C2} in the appendix.
\end{proof}

\subsubsection{Null directions}
\label{SSS:Null}
We next check that the vector space $\tZZZ_Q$ defined in the introduction is included in $Z_Q$.
\begin{lemma}
 \label{L:A12}
 Let $Q\in \Sigma$. Then the following functions $g$ are in $\hdot\cap C^{\infty}(\RR^N)$ and satisfy $L_Qg=0$:
 \begin{gather}
  \label{A26}
  \frac{N-2}{2}Q+x\cdot\nabla Q\\
 \label{A28}
 (2-N)x_jQ+|x|^2\partial_{x_j}Q-2x_jx\cdot\nabla Q,\; k\in \{1,\ldots ,N\}\\
 \label{A25}
  \partial_{x_j}Q,\; j=1,\ldots ,N,\;\\
  \label{A27}
 (x_j\partial_{x_k}-x_k\partial_{x_j})Q,\; 1\leq j<k\leq N.
 \end{gather}
\end{lemma}
\begin{proof}
 The fact that the functions defined in \eqref{A26}\ldots \eqref{A27} are smooth follows immediately from the fact that $Q$ is smooth (see Proposition \ref{P:A1}). Furthermore, by Proposition \ref{P:A1} again, all the functions \eqref{A26}, \eqref{A25} and \eqref{A27} are in $\hdot$.
We have 
\begin{gather*}
 \forall s\in \RR,\quad -\Delta \left( e^{\frac{(N-2)}{2}s}Q\left( e^sx \right) \right)=\left|e^{\frac{N-2}{2}s}Q\left(e^sx\right)\right|^{\frac{4}{N-2}}e^{\frac{N-2}{2}s}Q\left( e^sx \right)\\
 \forall b\in \RR^N,\quad -\Delta Q(x+b)=|Q(x+b)|^{\frac{4}{N-2}}Q(x+b)\\
 \forall c\in  \RR^N,\quad -\Delta Q(P_c x)=|Q(P_c x)|^{\frac{4}{N-2}}Q(P_c x).
\end{gather*}
Differentiating these equalities with respect to $s$, $b$ or $c$, and taking the resulting equality at $0$, we get \eqref{A26}, \eqref{A25} and \eqref{A27}.

To get \eqref{A28}, let $h=\frac{\partial}{\partial y_j} \left( \frac{1}{|y|^{N-2}}Q\left( \frac{y}{|y|^2} \right)\right)$, and observe that by points \eqref{I:Qsmooth} and \eqref{I:Kelvin} of Proposition \ref{P:A1} and \eqref{A25}, $h$ is in $\hdot\cap C^{4}$ and satisfies:
\begin{equation*}
 \left( \Delta+\frac{N+2}{(N-2)|y|^4}|Q|^{\frac{4}{N-2}}\left( \frac{y}{|y|^2} \right) \right)h=0,
\end{equation*} 
at least away from the origin.
Let $g=\frac{1}{|x|^{N-2}}h\left( \frac{x}{|x|^2} \right)$. Since the Kelvin transformation is an isometry of $\hdot$, we get that $g$ is in $\hdot$. Using that $\Delta g=\frac{1}{|x|^{N+2}}(\Delta h)\left( \frac{x}{|x|^2} \right)$, we obtain
\begin{equation}
\label{A29}
\Delta g+\frac{N+2}{N-2}|Q|^{\frac{4}{N-2}} g=0
\end{equation} 
outside $x=0$. An explicit computation gives $g(x)=-(N-2)x_jQ+|x|^2\partial_{x_j}Q-2x_jx\cdot\nabla Q$. Thus $g$ is smooth and must satisfy \eqref{A29} also at $x=0$, which concludes the proof.
\end{proof}
\subsubsection{Estimates on the eigenfunctions}
\label{SSS:eigenfunctions}
Consider the radial coordinates:
\begin{equation*}
 r=|x|,\quad \theta=\frac{x}{|x|}\in S^{N-1}.
\end{equation*}
In this section we recall the following result of V.~Z.~Meshkov \cite{Meshkov89}:
\begin{prop}
\label{P:eigenfunction}
 Let $Q\in \Sigma$ and $Y\in L^2(\RR^N)$ such $Y\neq 0$ and
 \begin{equation}
  \label{G1}
  L_QY=-\omega^2Y,
 \end{equation} 
 with $\omega>0$. Then
 $$Y(x)=\frac{e^{-\omega |x|}}{|x|^{\frac{N-1}{2}}}\left(V\left(\frac{x}{|x|}\right)+\Phi(x)\right),$$
 where $V\in L^2(S^{N-1})$ is not identically $0$, and
 \begin{equation}
  \label{G2}
\int_{S^{N-1}} |\Phi(r,\theta)|^2\,d\theta\leq Cr^{-1/2}.
\end{equation} 
\end{prop}
As an immediate consequence of Proposition \ref{P:eigenfunction}, we obtain:
\begin{corol}
\label{C:lower_estimates}
Let $Y$ be as in Proposition \ref{P:eigenfunction}. Then there exists a constant $C>0$ such that for large $r$,
$$\int_{S^{N-1}}|Y(r,\theta)|^2\,d\theta\geq \frac{e^{-2\omega r}}{C\,r^{N-1}}.$$
\end{corol}

Proposition \ref{P:eigenfunction} is Theorem 4.3 of \cite{Meshkov89}. The proof of this result uses the following bound:
\begin{equation}
 \label{bound_above}
\forall r\geq 1,\quad \int_{S^{N-1}} |Y(r,\theta)|^2\,d\theta\leq C\frac{e^{-2\omega r}}{r^{N-1}},
\end{equation} 
which follows from estimates of S.~Agmon \cite{Agmon82BO}. We give a proof of \eqref{bound_above} for the sake of completeness, refering to \cite{Meshkov89} for the rest of the proof of Proposition \ref{P:eigenfunction}.
\begin{proof}[Proof of \eqref{bound_above}]
By scaling we can assume $\omega=1$.
By elliptic regularity (and since $|Q|^{\frac{4}{N-2}}\in C^{1}(\RR^N)$), we have $Y\in C^{3}(\RR^N)\cap H^2(\RR^N)$.
 Let 
 \begin{equation}
 \label{def_G}
 G(R)=\int_{R}^{+\infty} \int_{S^{N-1}} |\nabla Y(r,\theta)|^2+|Y(r,\theta)|^2\,d\theta r^{N-1}\,dr. 
 \end{equation} 
 \EMPH{Step 1. Bound on $G$} We show that there exists $C>0$ such that
 \begin{equation}
  \label{G6}
\forall R>0,\quad  G(R)\leq Ce^{-2R}.
 \end{equation} 
 By \eqref{G1},
 $$ \int_{R}^{+\infty} \int_{S^{N-1}} (\Delta Y\,Y+\frac{N+2}{N-2}|Q|^{\frac{4}{N-2}}Y^2-Y^2)\,d\theta\,r^{N-1}dr=0.$$
 Integrating by parts and using that $Y\in \hdot(\RR^N)$ to prove that the ``boundary term'' at infinity is zero, we obtain
 \begin{multline*}
 -\int_R^{+\infty} \int_{S^{N-1}} \left(|\nabla Y|^2+Y^2\right)\,d\theta\,r^{N-1}dr+\frac{N+2}{N-2} \int_R^{+\infty} \int_{S^{N-1}}|Q|^{\frac{4}{N-2}}Y^2\,d\theta\,r^{N-1}dr\\
-R^{N-1}\int_{S^{N-1}} (Y\partial_rY)(R,\theta)\,d\theta=0,
 \end{multline*} 
 and thus
 \begin{equation}
  \label{G7}
  G(R)=-R^{N-1}\int_{S^{N-1}} (Y\partial_r Y)(R,\theta)\,d\theta+\frac{N+2}{N-2}\int_R^{+\infty} \int_{S^{N-1}} |Q|^{\frac{4}{N-2}}Y^2\,d\theta\,r^{N-1}dr.
 \end{equation} 
 Furthermore, differentiating the definition \eqref{def_G} of $G$, we obtain 
$$G'(R)=-R^{N-1}\int_{S^{N-1}}|\nabla Y(R,\theta)|^2+(Y(R,\theta))^2\,d\theta.$$
 Therefore
 \begin{multline*}
  2G(R)+G'(R)\\
  =-R^{N-1}\int_{S^{N-1}} \left(|\nabla Y|^2+Y^2+2Y\partial_rY\right)(R,\theta)\,d\theta+\frac{2(N+2)}{N-2}\int_R^{+\infty} \int_{S^{N-1}}|Q|^{\frac{4}{N-2}}Y^2\,d\theta r^{N-1}\,dr.
 \end{multline*} 
By Proposition \ref{P:A1} \eqref{I:harmonic},
\begin{equation}
 \label{G9}
 2G(R)+G'(R) \leq \frac{C}{R^4} \int_{R}^{+\infty} \int_{S^{N-1}}Y^2(r,\theta)\,d\theta \,r^{N-1}dr\leq \frac{C}{R^4}G(R).
\end{equation} 
Thus
\begin{equation}
 \label{G10}
 \frac{d}{dR} \Big[\log G(R)+2 R\Big]\leq \frac{C}{R^4}.
\end{equation} 
Integrating, we obtain that $\log\left(e^{2R}G(R)\right)$ is bounded from above, which yields \eqref{G6}.

\EMPH{Step 2: end of the proof} 
Let for $r>0$. 
$$ b(r)=r^{N-1}\int_{S^{N-1}}|Y(r,\theta)|^2\,d\theta.$$
By Step 1, $b(r)$ and $b'(r)$ are integrable on $(1,+\infty)$. Thus $b(r)$ converges to $0$ as $r\to+\infty$. Hence, for $R\geq 1$,
\begin{multline*}
 |b(R)|=\left|\int_R^{\infty} b'(r)\,dr\right|\\
\leq \int_R^{\infty} \left((N-1)r^{N-2}\int_{S^{N-1}} |Y(r,\theta)|^2\,d\theta+ r^{N-1}\int_{S^{N-1}} |\partial_r Y(r,\theta)\,Y(r,\theta)|\,d\theta\right)\,dr\\
\leq CG(R)\leq Ce^{-2R},
\end{multline*}
which gives \eqref{bound_above}.
\end{proof}
We will also need the following estimate on the $L^{\frac{2(N+2)}{N-2}}$ norm of $Y$, which is essentially a corollary of Proposition \ref{P:eigenfunction} and its proof:
\begin{lemma}
\label{L:estim_Y}
There exists $C>0$ such that 
\begin{equation*}
 \forall R\geq 1,\quad \int_{|x|\geq R} |Y(x)|^{\frac{2(N+2)}{N-2}}\,dx\leq C\,\frac{e^{-\frac{2(N+2)\omega}{N-2}R}}{R^{q_N}},
\end{equation*} 
where $q_N=\frac{4(N-1)}{N-2}$ if $N=3,4$ and $q_5=\frac{32}{9}$.
\end{lemma}
\begin{proof}
We assume as before $\omega=1$.
Let for $J=(j,k)$, $1\leq j<k\leq N$,
$\partial_{\theta_J}=x_j\partial_{x_k}-x_{k}\partial_{x_j}$.
We notice that the derivatives $\partial_{\theta_J}$ are tangential to the spheres $rS^{N-1}$, and that the tangential component of the gradient, $\nabla_Tv$, satisfies $|\nabla_T v|\leq \frac{C}{r}\sum_{J}|\nabla_{\theta_J}v|$. Furthermore, each $\partial_{\theta_J}$ commutes with $\Delta$. 

\EMPH{Step 1. Estimate on $G_J$}

Fix $J=(j,k)$ with $1\leq j<k\leq N$. In this step we prove
\begin{equation}
 \label{est_tangential}
\forall R\geq 1,\quad \int_{S^{N-1}}\left(|Y(R,\theta)|^2+|\partial_{\theta_J}Y(R,\theta)|^2\right)\,d\theta\leq C\frac{e^{-2R}}{R^{N-1}}.
\end{equation} 
Let
$$ G_J(R)= \int_{R}^{+\infty} \int_{S^{N-1}} |\nabla \partial_{\theta_J}Y(r,\theta)|^2+|\partial_{\theta_J}Y(r,\theta)|^2\,d\theta r^{N-1}\,dr. $$
Using the argument of Step 2 of the proof of \eqref{bound_above}, we see that \eqref{est_tangential} will follow from:
 \begin{equation}
  \label{est_GJ}
\forall R\geq 1,\quad  G_J(R)\leq Ce^{-2R}.
 \end{equation} 
We next prove \eqref{est_GJ}. We have 
\begin{equation}
 \label{eq_dJ}
\Delta(\partial_{\theta_J} Y)-\partial_{\theta_J} Y+\frac{N+2}{N-2}|Q|^{\frac{4}{N-2}}\partial_{\theta_J}Y=-\partial_{\theta_J}\left(\frac{N+2}{N-2}|Q|^{\frac{4}{N-2}}\right)Y.
\end{equation} 
Proceding as in Step 1 of the proof of \eqref{bound_above}, we obtain
\begin{multline}
\label{eqGJ}
2G_J(R)+G_J'(R)\leq C\int_{R}^{+\infty} \int_{S^{N-1}}|Q|^{\frac{4}{N-2}}(\partial_{\theta_J}Y)^2d\theta\,r^{N-1}dr\\
+
C\int_{R}^{+\infty} \int_{S^{N-1}}\left|\partial_{\theta_J}\left(|Q|^{\frac{4}{N-2}}\right)Y\partial_{\theta_J}Y\right|d\theta\,r^{N-1}dr. 
\end{multline}
By Proposition \ref{P:A1} \eqref{I:harmonic}, $|Q|^{\frac{4}{N-2}}+\left|\partial_{\theta_J}\left(|Q|^{\frac{4}{N-2}}\right)\right|\leq C/R^{4}$. In view of \eqref{G6}, we obtain
\begin{equation*}
 \int_{R}^{+\infty} \int_{S^{N-1}}|Q|^{\frac{4}{N-2}}(\partial_{\theta_J}Y)^2d\theta\,r^{N-1}dr\leq \frac{C}{R^2}\int_{R}^{+\infty} \int_{S^{N-1}}|\nabla Y|^2d\theta\,r^{N-1}dr\leq \frac{Ce^{-2R}}{R^2}.
\end{equation*}
and
\begin{equation*}
 \int_{R}^{+\infty} \int_{S^{N-1}}\left|\partial_{\theta_J}\left(|Q|^{\frac{4}{N-2}}\right)Y\partial_{\theta_J}Y\right|d\theta\,r^{N-1}dr\leq \frac{C}{R^3}\int_{R}^{+\infty} \int_{S^{N-1}}|\nabla Y|^2+|Y|^2d\theta\,r^{N-1}dr\leq \frac{Ce^{-2R}}{R^3}.
\end{equation*} 
By \eqref{eqGJ}, we deduce
\begin{equation*}
 2G_J(r)+G_J'(R)\leq \frac{Ce^{-2R}}{R^2},
\end{equation*} 
and thus
\begin{equation*}
 \frac{d}{dR}\left(e^{2R}G_J(R)\right) \leq \frac{C}{R^2}.
\end{equation*} 
Integrating between $1$ and $R>1$, we obtain \eqref{est_GJ}.

\EMPH{Step 2} We prove the conclusion of the lemma in the case $N\in \{3,4\}$. By Sobolev embedding on the sphere $S^{N-1}$,
$$\|f\|_{L^{\frac{2(N+2)}{N-2}}(S^{N-1})}\leq C\|f\|_{H^{\frac{2(N-1)}{N+2}}(S^{N-1})}.  $$ 
If $N=3$ or $N=4$, $\frac{2(N-1)}{N+2}\leq 1$. By \eqref{est_tangential},
\begin{multline*}
 \int_{S^{N-1}} |Y(r,\theta)|^{\frac{2(N+2)}{N-2}}d\theta\leq C\left(\int_{S^{N-1}} |Y(r,\theta)|^2+|\partial_{\theta}Y(r,\theta)|^2\,d\theta\right)^{\frac{N+2}{N-2}}\leq \frac{C}{r^{\frac{(N-1)(N+2)}{N-2}}}e^{-\frac{2(N+2)}{N-2}r}
\end{multline*}
where we have denoted $|\partial_{\theta}Y|^2=\sum_J |\partial_{\theta_J}Y|^2$.
Multiplying by $r^{N-1}$ and integrating between $R$ and $\infty$, we obtain the desired estimate when $N=3$ or $N=4$.

\EMPH{Step 3} We next treat the case $N=5$. Note that $\frac{2(N+2)}{N-2}=\frac{14}{3}$.
Since $\frac{2(N-1)}{N+2}=\frac{8}{7}>1$, it is tempting to differentiate a second time the equation \eqref{G1} on $Y$ to obtain a $L^2$ estimates on $\partial_{\theta}^2Y$ and use a Sobolev inequality  on the sphere $S^4$ to bound the $L^{14/3}$ norm. This is not possible because of the low regularity of $|Q|^{\frac{4}{3}}$, and we will rather use directly the equation \eqref{G1} to get a bound on $\Delta Y$, then the $H^2$-critical Sobolev inequality on $\RR^5$.

Using the Sobolev inequality $\|f\|_{L^4(S^4)}\leq C\|f\|_{H^1(S^4)}$ we obtain, by Step 1, and the same proof as in Step 2,
\begin{equation}
\label{Step3a}
 \int_{R}^{+\infty} \int_{S^4} |Y(r,\theta)|^4d\theta r^4dr\leq \int_{R}^{+\infty} \left(\int_{S^4} |Y(r,\theta)|^2+|\partial_{\theta} Y(r,\theta)|^2d\theta \right)^2r^4dr\leq \frac{Ce^{-4R}}{R^4}.
\end{equation}
Let $\varphi \in C^{\infty}(\RR)$ such that $\varphi(r)=0$ if $r\leq 0$ and $\varphi(r)=1$ if $r\geq 1$. Let 
$$Y_R(x)=\varphi(|x|-R)Y(x).$$
By \eqref{G6}, and equation \eqref{G1} (noting that all derivatives of $x\mapsto \varphi(|x|-R)$ are uniformly bounded for $R\geq 1$),
\begin{equation*}
 \forall R\geq 1,\quad \int_{\RR^5} \left(Y_R^2+|\nabla Y_R|^2+|\Delta Y_R|^2\right)dx\leq Ce^{-2R}.
\end{equation*} 
Using the $H^2$-critical Sobolev inequality in $\RR^5$, we deduce 
\begin{equation}
\label{Step3b}
 \int_{|x|\geq R} Y^{10}(x)dx\leq \int_{\RR^5} Y_R^{10}(x)dx\leq Ce^{-10R}.
\end{equation}
The conclusion of the lemma follows from \eqref{Step3a}, \eqref{Step3b} and the interpolation inequality
$$\|f\|_{L^{\frac{14}{3}}}\leq \|f\|_{L^4}^{\frac{16}{21}}\|f\|_{L^{10}}^{\frac{5}{21}}.$$
\end{proof}

\subsection{Choice of the modulation parameters}
\label{SS:modulation}

Recall from the introduction the definition of $\tZZZ_Q$. By Lemma \ref{L:A12}, $\tZZZ_Q\subset \ZZZ_Q$. The nondegeneracy assumption \eqref{ND} means that these two vector spaces are identical.

As before, we denote by $Z_1$,\ldots,$Z_m$ a basis of $\ZZZ_Q$ and $E_1$,\ldots,$E_m$ elements of $C^{\infty}_0$ such that \eqref{A7} holds. 

 Let $A\in \RR^{N'}$. Recall from \eqref{def_thetaA} the definition of $\theta_A$. We will denote by $\theta^{-1}_A$ the inverse of $\theta_A$, and $\left(\theta_A^{-1}\right)^*$ its adjoint. An explicit computation shows that 
 \begin{multline}
  \label{A31}
  \left(\theta^{-1}_{\vA}\right)^*(g)(x)=\left|\det \varphi_{\vA}'(x)\right|^{\frac{N+2}{2N}} g\left(\varphi_{\vA}(x)\right)\\
=e^{\frac{(N+2)s}{2}} \left|\frac{x}{|x|} -|x|a\right|^{-(N+2)}g\left(b+\frac{e^sP_c(x-|x|^2a)}{1-2\langle a,x\rangle +|a|^2|x|^2}\right).
 \end{multline} 
 In this subsection we show the following:
\begin{lemma}
 \label{L:A13}
 Assume \eqref{ND}. There exists a neighborhood $\UUU$ of $0$ in $\RR^{N'}$, a neighborhood $\VVV$ of $Q$ in $H^{-1}$, and a $C^1$, Lipschitz continuous map $\Psi:\VVV\to \UUU$ such that, 
 \begin{equation}
  \label{A30}
  \forall f\in \VVV,\; \forall i=1\ldots m,\quad \left\langle f,\left(\theta_{\Psi(f)}^{-1}\right)^{*}(E_i)\right\rangle_{H^{-1},H^1}-\int Q E_i=0.
 \end{equation} 
\end{lemma}
\begin{remark}
 If $f\in \hdot$, then \eqref{A30} is equivalent to
 \begin{equation}
 \label{A30'}
  \forall i=1\ldots m,\quad \int \left(\theta^{-1}_{\Psi(f)}(f)-Q\right)E_i=0.
 \end{equation}
 We will often use \eqref{A30'} instead of \eqref{A30}, but will also need \eqref{A30} which has the advantage of making sense for any $f\in H^{-1}$.
 \end{remark}
\begin{proof}[Proof of Lemma \ref{L:A13}]
 By Corollary \ref{C:C1function} (with $\psi=E_j$ which is an element of $\SSS$) and \eqref{A31},
 \begin{equation*}
\Phi: 
(\vA,f)\in B^{N'}(\eps)\times H^{-1}(\RR^N) \longmapsto \bigg(\la f, \left(\theta_{\vA}^{-1}\right)^* (E_j)\ra_{H^{-1},H^1}-\int QE_j\bigg)_{j=1\ldots m}\in \RR^m
\end{equation*}
is well defined and $C^1$. Using Corollary \ref{C:C1function} again, we can differentiate 
$$ \Phi_j(\vA,Q)=\int Q\left(\theta_{\vA}^{-1}\right)^*(E_j)-\int QE_j $$
below the integral sign, which yields:
\begin{align*}
 \frac{\partial\Phi_j}{\partial s}(0,Q)&=-\int \left(x\cdot \nabla Q+\frac{N-2}{2} Q\right) E_j\\
 \frac{\partial \Phi_j}{\partial a_i} (0,Q)&=-\int \left( (N-2)x_iQ-|x|^2\partial_{x_i}Q+2x_i x\cdot\nabla Q \right)E_j\\
 \frac{\partial \Phi_j}{\partial b_i}(0,Q)&=-\int \partial_{x_i}QE_j\\
 \frac{\partial \Phi_j}{\partial c_i}(0,Q)&=\int (x_{\ell}\partial_{x_{k}} -x_{k}\partial_{x_{\ell}}) Q E_j,
\end{align*}
with $\zeta(k,\ell)=i$, where $\zeta$ is the map appearing in the definition \eqref{Pc} of $P_c$.
By the nondegeneracy assumption \eqref{ND}, we deduce that the differential map $d_{1}\Phi(0,Q)$ with respect to the first variable $A$ is onto from $\RR^{N'}$ to $\RR^m$. If $m=N'$, this map is an isomporhism of $\RR^{N'}$ and we can directly apply the implicit function theorem. In the general case, Let $L:\RR^m\to \RR^{N'}$ be a right inverse of $d_{1}\Phi(0,Q)$ (so that $ d_{1}\Phi(0,Q) L$ is the identity of $\RR^m$), and $\tPhi(B,f)=\Phi(L(B),f)$. Then (taking a smaller $\eps>0$ if necessary),
$$ \tPhi: B^m(\eps)\times H^{-1}\to \RR^m$$
satisfies the assumptions of the implicit function theorem. We deduce that there exists a neighborhood $\UUU$ of $0$ in $\RR^m$, a neighborhood $\VVV$ of $Q$ in $H^{-1}$, and a $C^1$ map $\tPsi:\VVV\to \UUU$ such that
$$\forall f\in \VVV,\quad \tPhi\left(\tPsi(f),f\right)= \Phi\left(L(\tPsi(f)),f\right)=0.$$
Letting $\Psi(f)=L\circ \tPsi$, we obtain a $C^1$ map that satisfies \eqref{A30}. It remains to prove that $\Psi$ is Lipschitz-continuous on $\VVV$. For this it is sufficient to prove that the differential of $\tPsi$ is bounded on $\VVV$. Differentiating the relation $\tPhi\left(\tPsi(f),f\right)=0$ with respect to $f$, we obtain:
$$(\partial_1\tPhi)\left(\tPsi(f),f\right)\circ \left( d\tPsi(f) \right)+(\partial_2\tPhi)\left(\tPsi(f),f\right)=0.$$
Next, note that $(\partial_2\tPhi)\left(\tPsi(f),f\right)$ is the map $g\mapsto \bigg(\la g, \left(\theta_{\tPsi(f)}^{-1}\right)^* (E_j)\ra_{H^{-1},H^1}\bigg)_{j=1\ldots m}$, and that $(\partial_1\tPhi)(\tPsi(f),f)$ is an isomorphism of $\RR^m$, uniformly bounded (as well as its inverse) on $\VVV$. Finally, we obtain that $d\tPsi(f)$, as a bounded linear operator from $H^{-1}$ to $\RR^m$, is uniformly bounded for $f\in \VVV$, which yields the result.
\end{proof}
We conclude this section by a technical estimate, which says that under the nondegeneracy assumption \eqref{ND}, in a neighborhood of a a non-zero stationary solution $Q$, the distance of a $\hdot$ function $f$ to $\Sigma$ is well estimated by the distance of $f$ to $\theta_{\Psi(f)}(Q)$, where $\Psi$ is the map constructed in Lemma \ref{L:A13}.
\begin{lemma}
 \label{L:A18}
 Let $Q\in \Sigma$ such that the nondegeneracy assumption \eqref{ND} holds. There exists a small constant $\eps_2>0$ with the following property. Let $f\in \hdot$ and $S\in \Sigma$ such that
 \begin{equation}
 \label{bounds_f}
  \|f-Q\|_{\hdot}+\|f-S\|_{\hdot}<\eps_2.
 \end{equation}
 Then, if $\vA=\Psi(f)$ is given by Lemma \ref{L:A13},
 \begin{equation}
  \label{A49}
  \left\|f-\theta_{\vA}(Q)\right\|_{\hdot}\leq C\|f-S\|_{\hdot}.
 \end{equation}
\end{lemma}
\begin{proof}
 We will prove
 \begin{equation}
  \label{A50}
  \left\|S-\theta_{\vA}(Q)\right\|_{\hdot}\leq C\|f-S\|_{\hdot},
 \end{equation} 
 which obviously implies \eqref{A49} (with a larger constant $C$).
 
 Since the map $\Psi$ of Lemma \ref{L:A13} is Lipschitz-continuous, we have $|\vA|\leq C\eps_2$. Taking $\eps_2>0$ small, we can use Corollary \ref{C:A16}. By \eqref{thetaA_Q} and assumption \eqref{bounds_f},
 \begin{multline}
  \label{A51}
  \left\|\theta^{-1}_{\vA}(S)-Q\right\|_{\hdot}=\|S-\theta_{\vA}(Q)\|_{\hdot}\\
  \leq \|f-S\|_{\hdot}+\|Q-\theta_A(Q)\|_{\hdot}+\|f-Q\|_{\hdot}\leq C\eps_2.
 \end{multline}
By \eqref{A34} in Corollary \ref{C:A16}
\begin{equation}
 \label{A52}
\left\|\theta_{\vA}^{-1}(S)-Q\right\|_{\hdot}\leq C\sum_{i=1}^m \left|\int (\theta_{\vA}^{-1}(S)-Q)E_i\right|.
\end{equation} 
By the definition of $\vA=\Psi(f)$, we have $\int \left(\theta^{-1}_{\vA}(f)-Q\right)E_i=0$ for all $i$, and \eqref{A52} implies
\begin{multline*}
 \left\|\theta^{-1}_{\vA}(S)-Q\right\|_{\hdot}\leq C\sum_{i=1}^m \left|\int \left( \theta_{\vA}^{-1}(S)-\theta_{\vA}^{-1}(f) \right)E_i\right|\\ \leq C\|\theta_{\vA}^{-1}(S)-\theta_{\vA}^{-1}(f)\|_{\hdot}=C\|S-f\|_{\hdot},
\end{multline*}
which concludes the proof.
\end{proof}

\section{Proof of the main result}
\label{S:proof}
In this section we prove Theorem \ref{T:maintheo2}.
Let, for $f\in \hdot$
\begin{equation}
 \label{eq08}
d(\Sigma,f):=\inf\left\{ \|f-Q\|_{\hdot},\quad Q\in \Sigma\right\}.
\end{equation}
We will use the following proposition, proved in Section \ref{S:modulation} below:
\begin{prop}
\label{P:00}
 Let $Q\in \Sigma$ satisfying the nondegeneracy property \eqref{ND}. Then there exists $\delta_0=\delta_0(Q)$ with the following property.
If $u$ is a solution of \eqref{CP} with maximal time of existence $T_+$ and such that
\begin{gather}
 \label{eq09} 
E(u,\partial_tu)=E(Q,0)\\
\label{eq010}
\|u_0-Q\|_{\hdot}+\|u_1\|_{L^2}<\delta_0\\
\label{eq011}
\sup_{t\in [0,T_+)}d(\Sigma, u(t))+\|\partial_tu(t)\|_{L^2}<\delta_0.
\end{gather}
Then $T_+=+\infty$. Furthermore, there exists $S\in \Sigma$, of the form $S=\theta_A(Q)$ with $A\in \RR^{N'}$ close to $0$, such that one of the following holds:
$$u\equiv S$$
\textbf{ or }there exists a (non-zero) eigenfunction $Y$ of $L_S$, with eigenvalue $-\omega^2$ such that, for some $\omega^+>\omega$,
\begin{equation}
 \label{eq012}
\left\|(u(t),\partial_t u(t))-\left(S+e^{-\omega t}Y,-\omega e^{-\omega t}Y\right)\right\|_{\hdot\times L^2}\leq Ce^{-\omega^+ t}.
\end{equation} 
\end{prop}

\begin{remark}
 Proposition \ref{P:00} remains valid if $Q=0$. In this case, $S=0$ and the linearized operator $L_S=-\Delta$ has no eigenvalue, so that \eqref{eq012} is impossible. The conclusion of the proposition means that the only small solution of  \eqref{CP} such that $E[u]=0$ is $0$, which is an immediate consequence of the critical Sobolev inequality. 
\end{remark}
We will also need the fact, proved in Section \ref{S:Lorentz} that the Lorentz transform of a solution with the compactness property is also a solution with the compactness property (see Proposition \ref{P:lorentz} for a precise statement).

We let $u$ be a non-zero solution with the compactness property, so that there exists $\lambda(t)$, $x(t)$ such that
$$K=\left\{ \left( \lambda^{\frac{N}{2}-1}(t)u\left(t,\lambda(t)\cdot+x(t)\right), \lambda^{\frac N2}(t)\partial_t u\left(t,\lambda(t)\cdot+x(t)\right)\right),\; t\in (T_-(u),T_+(u))\right\}$$
has compact closure in $\hdot\times L^2$

\subsection{Reduction to the case of zero momentum}
\label{SS:reduction}
By Proposition \ref{P:maintheo1}, there exists $Q\in \Sigma$, $\vell\in B^N(1)$ and $s_n\to T_+(u)$ such that
\begin{multline}
\lim_{n\to\infty}\left\|\lambda^{N/2-1}\left(s_n\right)u\left(s_n,\lambda\left(s_n\right)\cdot +x\left(s_n\right)\right)-Q_{\vell}\left(0\right)\right\|_{\hdot}\\
+\left\|\lambda^{N/2}\left(s_n\right)\partial_t u\left(s_n,\lambda\left(s_n\right)\cdot +x\left(s_n\right)\right)-\partial_tQ_{\vell}\left(0\right)\right\|_{\hdot}=0.
\end{multline}
By Proposition \ref{P:lorentz}, we get that $u_{-\vell}$ is well defined and has the compactness property. Furthermore, in view of the proof of Lemma \ref{L:Lorentz_compactness} (see Remark \ref{R:Lorentz}), there exists a sequence of times $\{t_n\}$ such that $(u_{-\vell}(t_n),\partial_tu_{-\vell}(t_n))$ converges to $(Q,0)$ (up to scaling and space translation). We deduce $P[u_{-\vell}]=0$. As a consequence, it is sufficient to prove Theorem \ref{T:maintheo2} assuming
$$ P[u]=0,$$
which we will do in the sequel.

\subsection{Existence of the stationary profile}
\label{SS:stat}
By Proposition \ref{P:maintheo1}, \eqref{I:theo_virial}, and since $P[u]=0$, there exists a sequence $t_n^+\to T_+(u)$  and $Q\in \Sigma$ such that
\begin{equation}
 \label{F1}
\lim_{n\to \infty}\left\|\lambda^{\frac{N-2}{2}}(t_n^+)u\left(t_n^+,\lambda(t_n^+)\cdot +x(t_n^+)\right)-Q\right\|_{\hdot}+\left\| \partial_tu(t_n^+)\right\|_{L^2}=0, \text{ where }Q\text{ verifies }\eqref{ND}.
\end{equation}  
Let $\MMM(Q)=\left\{\theta(Q),\; \theta\in \MMM\right\}$. If $(f,g)\in \hdot\times L^2$, we let $d_Q$ be defined by
\begin{equation}
 \label{F2}
d_Q(f,g)=\inf\left\{\|f-\theta(Q)\|_{\hdot}+\|g\|_{L^2},\quad \theta \in \MMM\right\}
\end{equation} 
We claim:
\begin{claim}
 \label{C:MdQ}
 \begin{enumerate}
  \item \label{I:dQ_cont} $d_Q$ is continuous on $\hdot\times L^2$.
  \item \label{I:MQclosed} $\MMM_Q$ is closed in $\hdot$.
  \item \label{I:dQ0} $\ds \forall (f,g)\in \hdot\times L^2,\quad d_Q(f,g)=0\iff (f,g)\in \MMM_Q\times \{0\}$.
 \end{enumerate}
\end{claim}
\begin{proof}
Point \eqref{I:dQ_cont} is elementary, and \eqref{I:dQ0} follows immediately from \eqref{I:MQclosed}. Let us prove \eqref{I:MQclosed}.

Let $\{Q_n\}_n$ be a sequence in $\MMM_Q$, $S\in \hdot$ such that $\lim_n\|Q_n-S\|_{\hdot}=0$. We fix a small constant $\eps>0$ and choose $n_0$ such that
$$\forall n,p\geq n_0,\quad \|Q_n-Q_p\|_{\hdot}<\eps.$$
For $n\geq n_0$, we let $A_n=\Psi(Q_n)\in \RR^{N'}$, where $\Psi$ is given by Lemma \ref{L:A13} with $Q=Q_{n_0}$. Note that the sequence $\{A_n\}_{n\geq n_0}$ is bounded in $\RR^{N'}$ (by $C\eps$). Extracting, we can assume
$$\lim_{n\to\infty} A_n=A\in \RR^{N'}\text{ with }|A|\leq C\eps.$$
By Lemma \ref{L:A18} with $f=S=Q_n$, $Q=Q_{n_0}$, and taking $\eps$ smaller than the constant $\eps_2$ of Lemma \ref{L:A18}, we have
$$Q_n=\theta_{A_n}(Q_{n_0})\underset{n\to\infty}{\longrightarrow} \theta_{A}(Q_{n_0}).$$
Thus $S=\theta_A(Q_{n_0})\in \MMM_{Q}$ since $Q_{n_0}\in \MMM_{Q}$. 
\end{proof}
\subsection{Existence and properties of an asymptotic compact solution}
\label{SS:compact}
We first prove two lemmas. We must show that $(u_0,u_1)\in \MMM(Q)\times\{0\}$. We argue by contradiction.
\begin{lemma}
 \label{L:limit_el}
Let $u$ be as above, and assume that $(u_0,u_1)\notin \MMM(Q)\times\{0\}$. Let $\delta_1>0$ be a small parameter. Then there exists a solution $w$ of \eqref{CP} such that $w$ has the compactness property and
\begin{gather}
\label{F7}
d_Q(w_0,w_1)=\delta_1\\
 \label{F8}
\forall t\in [0,T_+(w)),\quad d_Q(w(t),\partial_tw(t))\leq \delta_1.
\end{gather} 
\end{lemma}
\begin{proof}
 \EMPH{Step 1: construction of  $w$}
By Claim \ref{C:MdQ}, \eqref{I:dQ0}, and since $(u_0,u_1)\notin \MMM(Q)\times\{0\}$ we have $d_Q(u_0,u_1)>0$. Choose $\delta_1$ small, so that $\delta_1<d_Q(u_0,u_1)$. We have
\begin{equation}
\label{dQ0}
\lim_{n\to \infty}d_Q(u(t_n^+),\partial_tu(t_n^+))=0. 
\end{equation} 
By continuity of $d_Q$, if $n$ is large, there exists $t_n$ such that
\begin{gather}
 \label{F3}
0<t_n<t_n^+\\
\label{F4}
\forall t\in (t_n,t_n^+], \quad d_Q\big(u(t),\partial_tu(t)\big)<\delta_1\\
\label{F5}
d_Q(u(t_n),\partial_tu(t_n))=\delta_1.
\end{gather} 
We let
$$ w_n(s,y)=\lambda^{\frac{N-2}{2}}(t_n)u\left(t_n+\lambda(t_n)s,x(t_n)+\lambda(t_n)y\right),$$
where $\lambda(t)$, $x(t)$ are given by Definition \ref{D:compactness}. By the compactness property, extracting subsequences in $n$ if necessary, there exist $(w_0,w_1)\in \hdot\times L^2$ such that
\begin{equation}
 \label{F6}
\lim_{n\to\infty} \left\|(w_n(0)-w_0,\partial_sw_n(0)-w_1)\right\|_{\hdot\times L^2}=0.
\end{equation} 
By Lemma \ref{L:Lorentz_continuity} \eqref{I:compact}, $w$ has the compactness property.

\EMPH{Step 2. Proof of \eqref{F7} and \eqref{F8}}

The equality \eqref{F7} follows directly from \eqref{F5} and \eqref{F6} by the continuity of $d_Q$.

To prove \eqref{F8}, it is sufficient to prove that for all $s$ in $[0,T_+(w))$, there exists $N(s)$ such that
\begin{equation}
\label{F11}
\forall n\geq N(s),\quad 0\leq t_n+\lambda(t_n)s<t_n^+.
\end{equation}
The desired property \eqref{F8} will then follow from \eqref{F4}, \eqref{F6} and the continuity of $d_Q$. 

We prove \eqref{F11} by contradiction. Assume that there exists $s\in [0,T_+(w))$ such that \eqref{F11} does not hold. Then there exists a sequence $s_n$ such that
\begin{equation}
 \label{F12}
\forall n,\; s_n\in [0,s]\text{ and } t_n+\lambda(t_n)s_n=t_n^+.
\end{equation} 
Extracting subsequences, we assume
$$\lim_{n\to\infty}s_n=s_{\infty}\in [0,s].$$
By long time perturbation theory, 
\begin{equation*}
\lim_{n\to \infty} \left\|(w_n(s_n),\partial_sw_n(s_n))-(w(s_{\infty}),\partial_sw(s_{\infty}))\right\|_{\hdot\times L^2}=0,
\end{equation*} 
that is
\begin{equation*}
 \lim_{n\to\infty}\Big(\lambda^{N/2-1}(t_n)u\left(t_n^+,\lambda(t_n)\cdot+x(t_n)\right),\lambda^{N/2}(t_n)\partial_t u\left(t_n^+,\lambda(t_n)\cdot+x(t_n)\right)\Big)=(w(s_{\infty}),\partial_sw(s_{\infty}))
\end{equation*}
in $\hdot\times L^2$. By \eqref{dQ0} and the continuity of $d_Q$, we deduce $d_Q(w(s_{\infty}),\partial_sw(s_{\infty}))=0$. By Claim \ref{C:MdQ} \eqref{I:dQ0}, $w(s_{\infty})\in \MMM_Q$ and $\partial_sw(s_{\infty})=0$, which contradicts \eqref{F7}. The proof of the Lemma is complete.
\end{proof}
We can choose $\delta_1$ in Lemma \ref{L:limit_el} so small that $d_Q(w_0,w_1)=\delta_1<\delta_0$, where $\delta_0$ is given by Proposition \ref{P:00}. As a consequence, there exists $\tQ\in \MMM(Q)$ such that
$$ \|(w_0,w_1)-(\tQ,0)\|_{\hdot\times L^2}< \delta_0.$$
Since $Q$ satisfies the non-degeneracy assumption \eqref{ND}, it is also the case of $\tQ$. Combining with \eqref{F8}, we see that $w$ satisfies the assumptions of Proposition \ref{P:00}. By the conclusion of this proposition, $T_+(w)=+\infty$ and, since $(w_0,w_1)\notin \MMM(Q)$ (by \eqref{F7}), there exists $S\in \MMM(Q)$ and $Y$ such that
\begin{gather}
 \label{F15}
L_SY=-\omega^2Y,\quad Y\not \equiv 0,\quad \omega>0\\
\label{F16}
 \forall t\geq 0,\quad \|(\eps(t),\partial_t\eps(t))\|_{\hdot\times L^2}\leq Ce^{-\omega^+t},
\end{gather}
where $\omega_+>\omega$ is close to $\omega$ and
\begin{equation}
 \label{def_eps_t}
\eps(t)=w(t)-S-e^{-\omega t}Y.
\end{equation} 
We next prove:
\begin{lemma}
 \label{L:crucial}
There exists $R,T>0$ and a constant $C>0$ such that for all $t_0\geq T$, $r_0\geq R$, 
\begin{equation}
 \label{F17} 
\sup_{T_-(w)<t\leq t_0}\left\|(\nabla \eps(t),\partial_t\eps(t))\right\|_{L^2\left(\left\{|x|\geq r_0+|t_0-t|\right\}\right)}\leq Ce^{-\omega^+t_0}.
\end{equation} 
\end{lemma}
\begin{remark}
 In the supremum in \eqref{F17}, $t$ can be negative, and, thus if $|T_-(w)|$ is large, $e^{-\omega t}$ can be very large. However, in the region $\left\{|x|\geq r_0+|t_0-t|\right\}$, $e^{-\omega t}Y(x)$ is small (see Claim \ref{C:estimates} below).
\end{remark}
\begin{proof}
 We notice that $f(t)=w(t)-S$ satisfies $\partial_t^2 f+L_Sf=R_S(f)$, where $R_S(f)$ is defined in \eqref{defRQ0}.
Thus $\eps(t)=f(t)-e^{-\omega t}Y$ satisfies
\begin{equation}
 \label{F18}
\partial_t^2\eps-\Delta \eps=\frac{N+2}{N-2}|S|^\frac{4}{N-2}\eps+R_S\left(\eps+e^{-\omega t} Y\right).
\end{equation} 
In the sequel, we denote by $\chi_{r_0,t_0}$ the characteristic function of the set $\Big\{(t,x)\in \RR^{N+1},\;s.t.\;|x|\geq r_0+|t-t_0|\Big\}$. We will need the following bounds, proved in Appendix \ref{A:space-time}, which are consequences of the estimates obtained in \S \ref{SSS:eigenfunctions}.  
\begin{claim}
 \label{C:estimates}
\begin{gather*}
\chi_{r_0,t_0} S \in  L^{\frac{N+2}{N-2}}L^{\frac{2(N+2)}{N-2}}:=L^{\frac{N+2}{N-2}}\left(\RR,L^{\frac{2(N+2)}{N-2}}(\RR^N)\right),\quad \chi_{r_0,t_0} e^{-\omega t}Y \in  L^{\frac{N+2}{N-2}}L^{\frac{2(N+2)}{N-2}}\\
\left\|\chi_{r_0,t_0} S\right\|_{L^{\frac{N+2}{N-2}}L^{\frac{2(N+2)}{N-2}}}\leq \frac{C}{r_0^{N/2-1}},\quad
\left\|\chi_{r_0,t_0} e^{-\omega t}Y\right\|_{L^{\frac{N+2}{N-2}}L^{\frac{2(N+2)}{N-2}}} \leq Ce^{-\omega(t_0+r_0)}.
\end{gather*}
\end{claim}

By Strichartz estimates and the local well-posedness theory for \eqref{CP}, 
$$w\in L^{\frac{N+2}{N-2}}_{\loc}\left(\left(T_-(w),+\infty\right),L^{\frac{2(N+2)}{N-2}}\right).$$
Thus $\eps \in L^{\frac{N+2}{N-2}}_{\loc}\left(\left(T_-(w),+\infty\right),L^{\frac{2(N+2)}{N-2}}\right)$. Using Claim \ref{C:estimates}, we deduce
$$ \left(R_S\left(\eps+e^{-\omega t} Y\right)+\frac{N+2}{N-2}|S|^{\frac{4}{N-2}}\eps\right) \chi_{r_0,t_0}\in L^{1}_{\loc}\left((T_-(w),+\infty),L^2\right).$$
We define $\overline{\eps}$ as the solution, in the integral sense, of the following equation:
\begin{equation}
\label{F21} 
\left\{
\begin{aligned}
 \partial_t^2\epsb-\Delta \epsb&=\left(R_S\left(\eps+e^{-\omega t} Y\right)+\frac{N+2}{N-2}|S|^{\frac{4}{N-2}}\eps\right) \chi_{r_0,t_0}\\
\left(\epsb,\partial_t\epsb\right)_{\restriction t=t_0}&=(\eps,\partial_t\eps)_{\restriction t=t_0}.
\end{aligned}\right.
\end{equation}
By Strichartz estimates,
\begin{equation}
\label{epsb_loc}
(\epsb,\partial_t\epsb)\in C^{0}\left((T_-(w),+\infty),\hdot\times L^2\right)\text{ and }\epsb\in L^{\frac{N+2}{N-2}}_{\loc}\left((T_-(w),+\infty),L^{\frac{2(N+2)}{N-2}}\right).
\end{equation}
By finite speed of propagation and equations \eqref{F18} and \eqref{F21}, $\eps=\epsb$ almost everywhere for $(x,t)$ such that $|x|\geq r_0+|t-t_0|$, and we can rewrite \eqref{F21} as
\begin{equation}
\label{F21'}
\tag{\ref{F21}'} 
\left\{
\begin{aligned}
 \partial_t^2\epsb-\Delta \epsb&=\left(R_S\left(\epsb+e^{-\omega t} Y\right)+\frac{N+2}{N-2}|S|^{\frac{4}{N-2}}\epsb\right) \chi_{r_0,t_0}\\
\left(\epsb,\partial_t\epsb\right)_{\restriction t=t_0}&=(\eps,\partial_t\eps)_{\restriction t=t_0}.
\end{aligned}\right.
\end{equation}
We shall prove that there is a large constant $C>0$ such that if $r_0$ and $t_0$ are large, 
\begin{equation}
 \label{F22} \left\|\epsb\right\|_{L^{\frac{N+2}{N-2}}\big((T_-(w),+\infty);L^{\frac{2(N+2)}{N-2}}\big)}+\sup_{t>T_-(w)}\left\|(\epsb(t),\partial_t\epsb(t))\right\|_{\hdot\times L^2} \leq Ce^{-\omega^+t_0}.
\end{equation} 
We will use a bootstrap argument. Let $I\subset (T_-(w),+\infty)$ be an interval such that $t_0\in I$ and
\begin{equation}
 \label{F23}
\left\|\epsb\right\|_{L^{\frac{N+2}{N-2}}\big(I;L^{\frac{2(N+2)}{N-2}}\big)}\leq Me^{-\omega^+t_0}
\end{equation} 
($M$ will be specified later). We will write $\chi=\chi_{r_0,t_0}$, $L(I)= L^{\frac{N+2}{N-2}}\big(I,L^{\frac{2(N+2)}{N-2}}(\RR^N)\big)$ to lighten the notation. By Strichartz estimates, \eqref{F16} and equation \eqref{F21'},
\begin{multline}
\label{Jeudi1}
\left\|\epsb\right\|_{L(I)}+\sup_{t\in I}\|(\epsb(t),\partial_t\epsb(t))\|_{\hdot\times L^2}\leq Ce^{-\omega_+t_0}+C\left\|R_S(\epsb+e^{-\omega t})\chi\right\|_{L^1(I,L^2)}\\
+C\left\| |S|^{\frac{4}{N-2}}\epsb\chi\right\|_{L^1(I,L^2)}
\end{multline}
(here and until the end of the proof, $C$ denotes a large positive constant, that may change from line to line and \emph{is independent of $M$}).
By the pointwise bound \eqref{pointwise_RQ} on $R_S$  and H\"older's inequality
\begin{equation*}
 \left\|R_S(\epsb+e^{-\omega t})\chi\right\|_{L^1(I,L^2)}\leq C\bigg(\left\|S \chi\right\|_{L(I)}^{\frac{6-N}{N-2}}\left\| (\epsb+e^{-\omega t}Y)\chi\right\|_{L(I)}^2
+\left\| (\epsb+e^{-\omega t}Y)\chi\right\|_{L(I)}^{\frac{N+2}{N-2}}  \bigg).
\end{equation*} 
Combining with  Claim \ref{C:estimates} and the bootstrap assumption \eqref{F23}, we obtain
\begin{multline}
\label{Jeudi2}
 \left\|R_S(\epsb+e^{-\omega t})\chi\right\|_{L^1(I,L^2)}\\
\leq C\left(\frac{1}{r_0^{\frac{6-N}{2}}}\left(M^2e^{-2\omega^+t_0}+e^{-2\omega(t_0+r_0)}\right)+M^{\frac{N+2}{N-2}} e^{-\frac{N+2}{N-2}\omega^+t_0}+e^{-\frac{N+2}{N-2}\omega(t_0+r_0)}\right).
\end{multline} 
On the other hand, using again H\"older's inequality, Claim \ref{C:estimates} and the bootstrap assumption \eqref{F23},
\begin{equation}
\label{Jeudi3}
 \left\| |S|^{\frac{4}{N-2}}\epsb\chi\right\|_{L^1(I,L^2)}\leq C \|S\chi\|^{\frac{4}{N-2}}_{L(I)}\|\epsb\|_{L(I)}\leq \frac{C}{r_0^2}Me^{-\omega^+t_0}.
\end{equation}
Combining \eqref{Jeudi1}, \eqref{Jeudi2} and \eqref{Jeudi3}, we obtain that there exists a constant $C_0>0$ (independent of the choice of $M$) such that
\begin{multline}
\label{Jeudi4}
 \left\|\epsb\right\|_{L(I)}+\sup_{t\in I}\|(\epsb(t),\partial_t\epsb(t))\|_{\hdot\times L^2}
\\
\leq C_0\bigg(e^{-\omega^+t_0}+\frac{M}{r_0^2}e^{-\omega^+t_0}+\frac{M^2}{r_0^{\frac{6-N}{2}}}e^{-2\omega^+t_0}+M^{\frac{N+2}{N-2}} e^{-\frac{N+2}{N-2}\omega^+t_0}+\frac{1}{r_0^{\frac{6-N}{2}}}e^{-2\omega(t_0+r_0)}
+e^{-\frac{N+2}{N-2}\omega(t_0+r_0)}\bigg).
\end{multline} 
We let $M=C_0+3$, choose $R>0$ large so that 
$$\frac{C_0M}{R^2}\leq 1,$$
then choose $T>0$ large, so that
$$C_0\bigg(\frac{M^2}{R^{\frac{6-N}{2}}}e^{-2\omega^+T}+M^{\frac{N+2}{N-2}} e^{-\frac{N+2}{N-2}\omega^+T}+\frac{1}{R^{\frac{6-N}{2}}}e^{-2\omega(T+R)}+e^{-\frac{N+2}{N-2}\omega(T+R)}\bigg)e^{\omega_+T}\leq 1$$
Then, if $r_0\geq R$ and $t_0\geq T$, 
\begin{equation}
 \label{big_bootstrap} 
\eqref{F23}\Longrightarrow\sup_{t\in I}\left\|(\epsb(t),\partial_t\epsb(t))\right\|_{\hdot\times L^2}+\|\epsb\|_{L(I)}\leq (M-1)e^{-\omega^+{t_0}}.
 \end{equation}

By \eqref{epsb_loc}, 
$ \left\|\epsb\right\|_{L\left((t_0-\eta,t_0+\eta)\right)}\leq Me^{-\omega^+t_0}$
for small positive $\eta$. Letting 
$$\sigma=\sup\left\{t>t_0,\; \|\epsb\|_{L\left((t_0,t)\right)}\leq  Me^{-\omega^+t_0}\right\},$$
we obtain by \eqref{big_bootstrap} that $\|\epsb\|_{L\left((t_0,\sigma)\right)}\leq (M-1)e^{-\omega^+t_0}$, and thus $\sigma=+\infty$ and, by \eqref{big_bootstrap}, 
$$\sup_{t_0<t}    \left\|(\epsb(t),\partial_t\epsb(t))\right\|_{\hdot\times L^2} \leq (M-1)e^{-\omega^+{t_0}}.$$
 Using a similar argument for times $t\leq t_0$, we deduce \eqref{F22}.
Since $\eps=\epsb$ in the region $|x|\geq r_0+|t-t_0|$, we obtain the conclusion \eqref{F17} of the Lemma.
\end{proof}

\subsection{End of the proof}
\label{SS:end_of_proof}
We next close the proof of Theorem \ref{T:maintheo2} by an energy channel argument. Let $w$ be the asymptotic solution, satisfying the compactness property, defined in the preceding subsection. We will prove that 
$$\liminf_{t\to T_-(w)} \|\partial_tw(t)\|_{L^2}>0,$$
contradicting Proposition \ref{P:maintheo1}. Let $t\in (T_-(w),0]$. We have
$$\partial_t w(t)=-\omega e^{-\omega t}Y+\partial_t \eps(t).$$
Hence
\begin{multline*}
  \left( \int_{r_0+|t-t_0|\leq |x|\leq r_0+|t-t_0|+1} |\partial_t w(t,x)|^2\,dx \right)^{\frac 12}\geq \\ \frac{1}{C}\left(\int_{r_0+|t-t_0|}^{r_0+|t-t_0|+1} \int_{S^{N-1}} Y^2(r,\theta)e^{-2\omega t}\,d\theta\, r^{N-1}\,dr\right)^{\frac 12}
  -\left( \int_{r_0+|t-t_0|\leq |x|\leq r_0+|t-t_0|+1} |\partial_t \eps(t,x)|^2\,dx \right)^{\frac 12}
\end{multline*} 
By Corollary \ref{C:lower_estimates} and Lemma \ref{L:crucial} we deduce that if $t_0\geq T$, $r_0\geq R$,
\begin{equation*}
\|\partial_tw(t)\|_{L^2}\geq \frac{1}{C}\left(\int_{r_0+|t-t_0|}^{r_0+|t-t_0|+1} e^{-2\omega t}e^{-2\omega r}\,dr\right)^{1/2}-Ce^{-\omega^+t_0}\geq \frac{1}{C_1}e^{-\omega(t_0+r_0)}-C_2e^{-\omega^+t_0},
\end{equation*} 
for some constants $C_1,C_2$. We have used that $t$ is negative, so that $|t-t_0|=t_0-t$. We fix $r_0=R$ and $t_0\geq T$ such that $\frac{1}{C_1}e^{-\omega r_0}\geq 2C_2e^{-(\omega^+-\omega)t_0}$, and obtain
$$\forall t\in (T_-(w),t_0], \quad \|\partial_t w\|_{L^2}\geq \frac{1}{2C_1} e^{-\omega t_0-\omega r_0}.$$
Since $P[w]=0$, we must have $\vell=0$ in Proposition \ref{P:maintheo1}, which shows that
$$\lim_{n\to \infty} \|\partial_tw(t_n^-)\|_{L^2}=0,$$
a contradiction. The proof is complete.

\subsection{Rigidity result with an additional bound on the solution}
We prove here the following consequence of Theorem \ref{T:maintheo2}, which is a corrected version of \cite[Theorem 2]{DuKeMe12}. See also the corrected arXiv version \textsf{arXiv:1003.0625v5}, where a proof independent of Theorem \ref{T:maintheo2} is given.
\begin{corol}
 \label{C:W}
 Let $u$ be a solution of \eqref{CP} with the compactness property. Let $\vell=-P[u]/E[u]$. Assume that one of the following holds:
 \begin{equation}
  \label{a_priori_bound_1}
\limsup_{t\to T^+(u)} \|\nabla u(t)\|^2_{L^2}<\frac{2N-2(N-1)|\vell|^2}{N\sqrt{1-|\vell|^2}}\|\nabla W\|_{L^2}^2
 \end{equation}
or
\begin{equation}
  \label{a_priori_bound_2}
\limsup_{t\to T^+(u)} \|\nabla u(t)\|^2_{L^2}+(N-1)\|\partial_t u(t)\|_{L^2}^2<\frac{2}{\sqrt{1-|\vell|^2}}\|\nabla W\|^2_{L^2}.
 \end{equation}
Then there exists $x_0\in \RR^N$, $\iota_0\in \{-1,+1\}$ and $\lambda_0>0$ such that
$$u(t,x)=\iota_0 \lambda_0^{\frac{N}{2}-1} W_{\vell}(\lambda_0 t,\lambda_0 x),$$
where $\vell=-P[u]/E[u]$ is an element of   $B^N(1)$ by Proposition \ref{P:maintheo1}.
\end{corol}
\begin{remark}
 Note that  
 $$\inf_{0\leq \ell<1} \frac{2N-2(N-1)\ell^2}{N\sqrt{1-\ell^2}}=\frac{4\sqrt{N-1}}{N},$$
 which shows that Corollary \ref{C:W} implies Theorem 2 of the arXiv version of \cite{DuKeMe12}.
 \end{remark}
 \begin{proof}[Proof of Corollary \ref{C:W}]
  Let $u$ be as in Corollary \ref{C:W}, and $Q^+\in \Sigma$ be given by Proposition \ref{P:maintheo1}. According to Theorem \ref{T:maintheo2}, and since $W$ satisfies the nondegeneracy assumption \eqref{ND}, it is sufficient to prove that $Q^+$ is equal to $W$ up to sign change, space translation and scaling. We recall (see \cite[Proof of Lemma 2.6]{DuKeMe12}):
  \begin{multline}
   \label{uniqueness_W}
   Q\in \Sigma \text{ and }\|\nabla Q\|_{L^2}^2<2\|\nabla W\|^2_{L^2}\\
   \Longrightarrow \exists \iota_0\in \{-1,+1\},\; \lambda_0>0,\; x_0\in \RR^N\text{ s.t. }Q(x)=\iota_0 \lambda_0^{\frac{N}{2}-1} W(\lambda_0 x).
  \end{multline} 
  We are thus reduced to prove
  \begin{equation}
   \label{Q+small}
   \|\nabla Q^+\|_{L^2}^2<2\|\nabla W\|^2_{L^2}.
  \end{equation} 
  Recall that 
  $$\|\nabla Q^+\|_{L^2}^2=\|Q^+\|_{L^{\frac{2N}{N-2}}}^{\frac{2N}{N-2}}.$$
  Let $j=1\ldots N$. Multiplying the equation $-\Delta Q^+=|Q^+|^{\frac{4}{N-2}}Q^+$ by $x_j\partial_{x_j}Q^+$ and integrating by parts, we obtain
  $$\|\partial_{x_j}Q^+\|_{L^2}^2=\frac{1}{N}\|\nabla Q^+\|_{L^2}^2.$$
  We deduce by direct computations
\begin{align*}
  \|\nabla Q_{\vell}^+(0)\|_{L^2}^2&=\frac{N-(N-1)|\vell|^2}{N\sqrt{1-|\vell|^2}}\|\nabla Q^+\|^2_{L^2}\\
\|\partial_tQ_{\vell}^+(0)\|_{L^2}^2&=\frac{|\vell|^2}{N\sqrt{1-|\vell|^2}}\|\nabla Q^+\|_{L^2}^2.
  \end{align*} 
  Thus we see that \eqref{a_priori_bound_1} or \eqref{a_priori_bound_2} implies, together with the conclusion \eqref{theo_CV} of Proposition \ref{P:maintheo1}, that \eqref{Q+small} holds, which concludes the proof of Corollary \ref{C:W}.
 \end{proof}

\section{Convergence to a stationary solution by modulation theory}
\label{S:modulation}
This section is devoted to the proof of the Proposition \ref{P:00}.

We divide the proof into two steps: in \ref{SS:exp}, we prove that $u$ converges exponentially to a stationary solution; in \ref{SS:expansion},  we conclude the proof. The proof relies on modulation theory and precise asymptotics on approximate linear differential systems.
\subsection{Exponential convergence to the stationary solution}
\label{SS:exp}
In this subsection, we prove the following proposition, which is the first of two steps of the proof of Proposition \ref{P:00}.
\begin{prop}
 \label{P:expo}
 Let $u$ satisfy the assumptions of Proposition \ref{P:00}. Then $T_+=+\infty$. Furthermore, there exist $S\in \Sigma$, of the form $S=\theta_A(Q)$ with $A\in \RR^{N'}$ close to $0$, and $\omega,C>0$ such that 
$$ \forall t\geq 0,\quad  \left\|(u(t)-S,\partial_t u(t))\right\|_{\hdot\times L^2}\leq Ce^{-\omega t}.$$
\end{prop}

In all \S \ref{SS:exp}, we consider a solution $u$ as in Propositions \ref{P:00} and \ref{P:expo}, and two small parameters, $\delta_0,r_0>0$ such that
$$ 0<\delta_0\ll r_0\ll 1.$$
\subsubsection{Modulation of the solution}
By our assumptions on $u$, 
\begin{equation}
 \label{A53}
 \|u_0-Q\|_{\hdot} <\delta_0.
\end{equation} 
Let $r_0>0$ be such that $B_{\hdot}(Q,r_0)\subset \VVV$, where $\VVV$ is the neighborhood of $Q$ in $H^{-1}$ given by Lemma \ref{L:A13}. Let
\begin{equation}
\label{A54}
 T_0=\inf\Big\{ t\in [0,T_+)\text{ s.t. } \|u(t)-Q\|_{\hdot}\geq r_0\Big\}.
\end{equation} 
If $\|u(t)-Q\|_{\hdot}<r_0$ for all $t\in [0,T_+)$, we let $T_0=T_+$. We can choose $\delta_0$ such that $0<\delta_0<r_0/2$, which implies by \eqref{eq010} that $T_0>0$. 

If $t\in [0,T_0)$, we let $\vA(t)=\Psi(u(t))$ ($\Psi$ given by Lemma \ref{L:A13}), so that
\begin{equation}
 \label{A55}
 \forall j=1,\ldots,m,\quad \int h(t,x)E_j(x)\,dx=0,
\end{equation} 
where
\begin{equation}
 \label{def_h}
 h=\theta_{\vA}^{-1}(u)-Q.
\end{equation} 
By Lemma \ref{L:A18} and  \eqref{eq011}
\begin{equation}
\label{A59}
\forall t\in  [0,T_0),\quad 
\|h(t)\|_{\hdot}=\left\|u(t)-\theta_{\vA(t)} (Q)\right\|_{\hdot}\leq C\delta_0. 
\end{equation} 
Since $u\in C^{1}\left([0,T_+), H^{-1}(\RR^N)\right)$, we know by Lemma \ref{L:A13} that $\vA\in C^{1}([0,T_0),\RR^{N'})$. Furthermore (using the Lipschitz continuity of the function $\Psi$ of Lemma \ref{L:A13}), $\|\vA(t)\|\leq Cr_0$.
Let
\begin{gather}
\label{A56}
 \alpha_j(t)=\int h(t)Y_j,\quad \beta_j(t)=\int \partial_t u(t) (\theta_{\vA(t)}^{-1})^{*} Y_j,\quad j=1,\ldots,p\\
 \label{A57}
 \delta(t)=\sqrt{\sum_{j=1}^p \alpha_j^2(t)}.
\end{gather}
\begin{lemma}
 \label{L:A19}
 There exists $C>0$ such that
 \begin{equation}
  \label{A60}
  \forall t\in [0,T_0),\quad \frac{1}{C}\left( \|\partial_t u(t)\|_{L^2}+\|h(t)\|_{\hdot} \right)\leq \delta(t)\leq C\|h(t)\|_{\hdot}.
 \end{equation} 
\end{lemma}
\begin{proof}
 The inequality at the right-hand side of \eqref{A60} follows immediately from the definition of $\delta(t)$. Let us show the other inequality.
 
 By conservation of the energy, we have, for $t\in [0,T_0)$, 
 $$E(Q,0)=E(u(t),\partial_tu(t))=E\left(\theta_{\vA(t)}^{-1}(u(t)),\partial_tu(t)\right)=E(Q+h(t),\partial_tu(t))$$
 and thus
 \begin{equation}
  \label{A60'}
  E(Q,0)=E(Q,0)+\frac{1}{2}\int (\partial_t u(t))^2+\Phi_{Q}(h(t))+\OOO\left(\|h(t)\|_{L^{\frac{2N}{N-2}}}^3\right).
 \end{equation} 
 Hence:
 \begin{equation}
  \label{A61}
  \int (\partial_tu(t))^2+\Phi_{Q}\left( h(t)-\sum_{j=1}^p\alpha_j(t)Y_j \right)\leq C\|h(t)\|_{\hdot}^3+C\delta^2(t).
 \end{equation} 
 By the definitions of $\vA$, $h$ and $\alpha_j$, we have
 \begin{align*}
  \int \Big( h(t)-\sum_{j=1}^p \alpha_j(t)Y_j \Big)E_k&=0, \qquad k=1,\ldots,m,\\
\int \Big( h(t)-\sum_{j=1}^p \alpha_j(t)Y_j \Big)Y_{\ell}&=0,\qquad \ell=1,\ldots,p.
 \end{align*} 
 By Proposition \ref{P:A7},
 $$\left\|h-\sum_{j=1}^p\alpha_jY_j\right\|_{\hdot}^2\leq C\Phi_{Q}\left(h-\sum_{j=1}^p \alpha_jY_j\right).$$
 Hence, by \eqref{A61},
 $$\int (\partial_tu)^2+\left\|h-\sum_{j=1}^p\alpha_jY_j\right\|_{\hdot}^2\leq C\|h\|^3_{\hdot}+C\delta^2(t).$$
Noting
\begin{multline*}
 \|h(t)\|^2_{\hdot}=\left\|\sum_{j=1}^p\alpha_jY_j+h-\sum_{j=1}^p\alpha_jY_j\right\|^2_{\hdot}
 \\
 \leq 2\left\| h-\sum_{j=1}^p \alpha_j Y_j\right\|^2_{\hdot}+2\left(\sum_{j=1}^p |\alpha_j|\left\|Y_j\right\|_{\hdot}\right)^2
 \leq 2\left\| h-\sum_{j=1}^p \alpha_j Y_j\right\|^2_{\hdot}+C\delta^2,
\end{multline*}
we obtain the left-hand inequality in \eqref{A60}.
\end{proof}
\begin{lemma}
 \label{L:A20}
 There exists $C>0$ such that 
 $$\forall t\in [0,T_0),\quad \|\vA'(t)\|\leq C\delta(t).$$
\end{lemma}
\begin{proof}
 Note that $\vA(t)=\Psi(u(t))$, where $\Psi$ is a $C^1$ map from $H^{-1}$ to $\RR^{N'}$. Differentiating, we obtain
 $$\vA'(t)=(d\Psi)(u(t)) \frac{du}{dt},$$
 and thus, using the uniform bound of $d\Psi$ in the proof of Lemma \ref{L:A13},
 $$\|\vA'(t)\|\leq C\left\|\frac{du}{dt}\right\|_{H^{-1}}\leq C\left\|\frac{du}{dt}\right\|_{L^2}\leq C\delta(t),$$
 where the last inequality follows from Lemma \ref{L:A19}.
\end{proof}
\subsubsection{Reduction to an approximate finite-dimensional linear differential system}
\begin{lemma}
 \label{L:A21}
 Under the assumptions of Propostion \ref{P:00}, let $\alpha_j$, $\beta_j$ and $\delta$ be defined by \eqref{A56},\eqref{A57}. Then $\alpha_j,\beta_j\in C^1([0,T_0),\RR)$ and
 \begin{align}
  \label{A63}
 \left| \alpha'_j(t)-\beta_j(t)\right|&\leq C\delta^2(t)\\
  \label{A64}
  \left|\beta_j'(t)-\omega_j^2\alpha_j(t)\right|&\leq C\left( |\vA(t)|\delta(t)+\delta^2(t) \right).
 \end{align}
\end{lemma}
\begin{proof}
 We have
 \begin{equation*}
  \alpha_j(t)=-\int QY_j+\int u(t)\left(\theta_{\vA(t)}^{-1}\right)^*(Y_j)\text{ and }
  \beta_j(t)=\int \partial_tu(t)\left(\theta_{\vA(t)}^{-1}\right)^*(Y_j),
 \end{equation*}
and the fact that $\alpha_j$ and $\beta_j$ are $C^1$ follows from the fact that $u\in C^2([0,T_0),H^{-1}(\RR^N)$, $\vA\in C^1([0,T_0),\RR^{N'})$ and Corollary \ref{C:C1function} with $\psi=Y_j$. Differentiating under the integral defining $\alpha_j$, we obtain
\begin{equation}
 \label{A65}
 \alpha_j'(t)=\underbrace{\int \partial_tu(t)\left(\theta_{\vA(t)}^{-1}\right)^*(Y_j)}_{\beta_j(t)}+\int u(t)\frac{\partial}{\partial t}\left( \left(\theta_{\vA(t)}^{-1}\right)^*(Y_j) \right).
\end{equation}
We note that the second integral is equal to 
\begin{equation*}
 \int \theta^{-1}_{\vA(t)}\left(u(t)\right)\left(\theta_{\vA(t)}\right)^*\left[\frac{\partial}{\partial t}\left( \left(\theta_{\vA(t)}^{-1}\right)^*(Y_j) \right)\right].
\end{equation*}
We have
$$\left(\theta_{\vA(t)}\right)^*\left(\frac{\partial}{\partial t}\left( \left(\theta_{\vA(t)}^{-1}\right)^*(Y_j) \right)\right)=\frac{\partial}{\partial\tau}\left(\left(\theta_{\vA(t)}\right)^*\left(\theta^{-1}_{\vA(\tau)}\right)^*(Y_j)\right)_{\restriction \tau=t}=\frac{d}{d\tau}\left(\left(\theta_{B(\tau)}^{-1}\right)^*(Y_j)\right)_{\restriction \tau =t},$$
where, in view of point \eqref{I:compo} of Proposition \ref{P:A1}, $\tau\mapsto B(\tau)$ is a $C^1$ function such that $B(0)=0$. 

Using Corollary \ref{C:C1function}, we get that 
$$\left(\theta_{\vA(t)}\right)^*\left(\frac{\partial}{\partial t}\left( \left(\theta_{\vA(t)}^{-1}\right)^*(Y_j) \right)\right)$$
is a linear combination of terms of the form $T^* Y_j$, where $T$ is one of the transformations defining $\tZZZ$: $\partial_{x_j}$, $x_j\partial_{x_k}-x_k\partial_{x_j}$, $(2-N)x_j+|x|^2\partial_{x_j}-2x_jx\cdot\nabla$ and $\frac{N-2}{2}+x\cdot \nabla$, we deduce:
\begin{equation*}
 \int Q\left(\theta_{\vA(t)}\right)^*\left[\frac{\partial}{\partial t}\left( \left(\theta_{\vA(t)}^{-1}\right)^*(Y_j) \right)\right]=\sum_{k=1}^m \int \gamma_k(t)Z_kY_j=0,
\end{equation*} 
where for $k=1\ldots m$, $\gamma_k(t)\in \RR$. Using the definition of $h$, we get $\theta_{\vA(t)}^{-1}\left( u(t) \right)=h(t)+Q$ and thus
\begin{multline*}
\left|\int  u(t)\frac{\partial}{\partial t}\left( \left(\theta^{-1}_{\vA(t)}\right)^*(Y_j) \right)\right|=\left|\int h(t)\left(\theta_{\vA(t)}\right)^*\frac{\partial}{\partial t}\left( \left(\theta^{-1}_{\vA(t)}\right)^*(Y_j) \right)\right|\\
\leq C\|h(t)\|_{L^{\frac{2N}{N-2}}} \left\| \left(\theta_{\vA(t)}\right)^*\frac{\partial}{\partial t}\left( \left(\theta^{-1}_{\vA(t)}\right)^*(Y_j) \right)\right\|_{L^{\frac{2N}{N-2}}}\\
\leq C\delta(t) \left\|\frac{\partial}{\partial t} \left(\left(\theta^{-1}_{\vA(t)}\right)^*Y_j\right)\right\|_{L^{\frac{2N}{N-2}}}\leq C|A'(t)|\delta(t)\leq C\delta^2(t),
\end{multline*}
by Lemmas \ref{L:A19}, \ref{L:A20}, and Corollary \ref{C:C1function}. Hence \eqref{A63}.

We next prove \eqref{A64}. We have
\begin{equation}
 \label{A66} \beta_j'(t)=\int \partial_t^2u(t) \left(\theta^{-1}_{\vA(t)}\right)^*(Y_j)+\int \partial_t u(t) \frac{\partial}{\partial t}\left( \left(\theta^{-1}_{\vA(t)}\right)^*Y_j \right),
\end{equation} 
and 
\begin{equation}
 \label{A66'}
 \left|\int \partial_t u \frac{\partial}{\partial t}\left( \left(\theta^{-1}_{\vA(t)}\right)^*Y_j \right)\right|\leq C\|\partial_t u(t)\|_{L^2}\left\|\frac{\partial}{\partial t}\left( \left(\theta^{-1}_{\vA(t)}\right)^*Y_j \right)\right\|_{L^2}\leq C\|\partial_t u(t)\|_{L^2}|\vA'(t)|,
\end{equation}
by Corollary \ref{C:C1function}. By Lemma \ref{L:A19} and \ref{L:A20}, the right-hand term of \eqref{A66'} is bounded by $C\delta^2(t)$ for a constant $C>0$. Let us consider the first term in the right-hand side of \eqref{A66}:
\begin{equation*}
 \int \partial_t^2u \left(\theta^{-1}_{\vA}\right)^*(Y_j)=\int (\Delta u+|u|^{\frac{4}{N-2}}u)\left(\theta^{-1}_{\vA}\right)^*(Y_j).
\end{equation*} 
We fix $t$ and denote, to simplify notation  
$$\varphi(x)=\varphi_{A(t)}(x)=b(t)+\frac{e^{s(t)}P_{c(t)} (x-a(t)|x|^2)}{1-2\langle a(t),x\rangle+|a(t)|^2|x|^2}.$$
Recall that 
\begin{equation}
\label{det}
|\det \varphi'(x)|= e^{Ns}\left|\frac{x}{|x|} -|x|a\right|^{-2N}
\end{equation} 
and (see \eqref{theta_Theta})
$$\theta_{\vA}(f)(x)=|\det \varphi'(x)|^{\frac{N-2}{2N}} f(\varphi(x))=e^{\frac{(N-2)s}{2}} \left|\frac{x}{|x|}-a|x|\right|^{-(N-2)}f(\varphi(x)).$$
Using that, by the definition \eqref{def_h} of $h$,
$$u=\theta_{\vA}(Q+h)=|\det(\varphi'(x))|^{\frac{N-2}{2N}}(Q+h)(\varphi(x)),\quad \Delta u=|\det(\varphi'(x))|^{\frac{N+2}{2N}}(\Delta (Q+h))(\varphi(x)),$$
we obtain
\begin{multline*}
 \int \partial_t^2u\left(\theta_{\vA}^{-1}\right)^*Y_j=\int |\det\varphi'(x)|^{\frac{N+2}{2N}}(\Delta u+|u|^{\frac{4}{N-2}}u)(x)Y_j(\varphi(x))\,dx\\
 =\int |\det\varphi'(x)|^{\frac{N+2}{N}}\left(\Delta (Q+h)+|Q+h|^{\frac{4}{N-2}}(Q+h)\right)(\varphi(x))Y_j(\varphi(x))\,dx\\
=\int \left|\det\varphi'\left(\varphi^{-1}(y)\right)\right|^{\frac 2N}(-L_{Q}h+R_{Q}(h))(y)Y_j(y)\,dy,
 \end{multline*}
where $R_{Q}$ is defined in \eqref{defRQ0}. Hence
\begin{multline}
 \label{A67}
 \int \partial_t^2u\left(\theta_{\vA}^{-1}\right)^*Y_j\\
=-\int L_{Q}hY_j+\int \left(1-\left|\det\varphi'(\varphi^{-1}(y))\right|^{\frac 2N}\right)L_Q(h) Y_j+\int \left|\det\varphi'(\varphi^{-1}(y))\right|^{\frac 2N}R_Q(h)Y_j.
\end{multline} 
By \eqref{A39} and Lemma \ref{L:A19}, $\left\|R_{Q}(h)\right\|_{L^{\frac{2N}{N+2}}}\leq C\|h\|_{L^{\frac{2N}{N-2}}}^2\leq C\delta^2$.  
Note that $x=\varphi^{-1}(y)\iff \frac{x}{|x|^2}=a+e^s P_{-c}\, \frac{y-b}{|y-b|^2}$. Thus:
$$\frac{1}{|\varphi^{-1}(y)|}=\left|a+\frac{e^sP_{-c}(y-b)}{|y-b|^2}\right|$$
and 
$$ \left|\frac{\varphi^{-1}(y)}{|\varphi^{-1}(y)|^2} -a\right|=\left|\frac{e^sP_{-c}(y-b)}{|y-b|^2}\right|=\frac{e^s}{|y-b|}.$$
As a consequence (see \eqref{det})
\begin{multline*}
 \left|\det(\varphi'(\varphi^{-1}(y)))\right|^{\frac 2N}=e^{2s}\left|\frac{\varphi^{-1}(y)}{|\varphi^{-1}(y)|} -a\left|\varphi^{-1}(y)\right|\right|^{-4}=e^{2s}\frac{1}{|\varphi^{-1}(y)|^4}\left|\frac{\varphi^{-1}(y)}{\left|\varphi^{-1}(y)\right|^2}-a\right|^{-4}\\
 =e^{-2s} \left( |a|^2|y-b|^2+2\la a,e^sP_{-c}(y-b)\ra +e^{2s} \right)^2.
\end{multline*}
Let us denote by $g(\vA,y)$ the expression on the last line. Note that $g(0,y)=1$ for all $y$, and that $\|\nabla_{\vA} g\|\leq C_K(1+|y|^4)$ if $\vA$ stays in a bounded set $K$ of $\RR^{N'}$. Hence, if $|\vA|\leq 1$,
$$\left| \left|\det(\varphi'(\varphi^{-1}(y)))\right|^{\frac 2N}-1\right|\leq C|\vA|\left(1+|y|^4\right).$$
Similarly,
$$\nabla_y \left(\left|\det\left(\varphi'(\varphi^{-1}(y))\right)\right|^{\frac 2N}\right)\leq C|\vA|(1+|y|^3).$$
Going back to \eqref{A67}, we get
\begin{multline*}
 \left|\int \partial_t^2 u (\theta_{\vA}^{-1})^*(Y_j)-\omega_j^2\int hY_j\right|\\
\leq C|\vA| \int_{\RR^N} \left(|Q|^{\frac{4}{N-2}} |h \,Y_j|+|\nabla h|\left(|\nabla Y_j|+|Y_j|\right)\right) (1+|y|)^4
+C\left\| (1+|y|)^4Y_j\right\|_{L^{\frac{2N}{N-2}}}\|R_{Q}(h)\|_{L^{\frac{2N}{N+2}}}\\
\leq C\left(\delta(t)|\vA|+\delta^2(t)\right),
\end{multline*}
by Lemma \ref{L:A19} and the decay properties of $Y_j$. Combining with \eqref{A66} and \eqref{A66'}, we get \eqref{A64}.
\end{proof}

\subsubsection{Exponential decay for a linear differential system}
In this subsection, we consider approximate ordinary differential systems of the form \eqref{A63}, \eqref{A64}. 
\begin{lemma}
\label{L:A22}
Let $0<\omega_1<\ldots<\omega_p$ be real numbers. There exists $\eps_3>0$, $C_3>0$ (depending only on the $\omega_j$s) such that the following holds.
Let $T_0\in [0,+\infty]$ and, consider, for $j=1\ldots p$, $\alpha_j$, $\beta_j\in C^1([0,T_0),\RR)$. Let 
$$\gamma(t)=\sqrt{\sum_{j=1}^p \omega_j^2|\alpha_j(t)|^2+|\beta_j(t)|^2}$$ 
and assume
\begin{gather}
\label{A68}
 \|\gamma\|_{\infty}=\sup_{t\in [0,T_0)} \gamma(t)<\infty\\
\label{A69}
\forall j\in 1\ldots p,\quad \forall t\in[0,T_0), \quad |\alpha_j'(t)-\beta_j(t)|\leq \eps_3\gamma(t)\\
\label{A70}
\forall j\in 1\ldots p,\quad \forall t\in[0,T_0), \quad |\beta_j'(t)-\omega_j^2\alpha_j(t)|\leq \eps_3\gamma(t).
\end{gather}
Then
\begin{equation}
\label{A71}
 \int_0^{T_0}\gamma(t)\,dt\leq C_3\|\gamma\|_{\infty}.
\end{equation}
If moreover $T_0=+\infty$ then
\begin{equation}
 \label{A72}
\lim_{t\to +\infty}e^{\frac{\omega_1}{2}t}\gamma(t)=0.
\end{equation}
\end{lemma}
\begin{proof}
\EMPH{Step 1}
 We let
$$ E_{\pm}(t):=\sum_{j=1}^p \left(\beta_j(t)\pm \omega_j \alpha_j(t)\right)^2.$$
Note that
\begin{equation}
 \label{A73}
\gamma^2(t)=\frac 12\left( E_+(t)+E_-(t)\right).
\end{equation}
If $t\in [0,T_0)$,
$$E_{\pm}'(t)=2\sum_{j=1}^p (\beta_j'(t)\pm \omega_j\alpha_j'(t))(\beta_j(t)\pm \omega_j \alpha_j(t)).$$
By \eqref{A69} and \eqref{A70}:
\begin{multline*}
 \left|E_+'(t)-2\sum_{j=1}^p \omega_j (\beta_j(t)+\omega_j\alpha_j(t))^2\right|\leq 2\sum_{j=1}^p \left(\omega_j\left|\alpha_j'-\beta_j\right|+\omega_j|\beta_j'-\omega_j^2\alpha_j|\right)\,\left|\beta_j+\omega_j\alpha_j\right|\\
\leq 2\sqrt{\sum_{j=1}^p (\beta_j+\omega_j\alpha_j)^2\, \sum_{j=1}^p \left(\omega_j|\alpha_j'-\beta_j|+\omega_j|\beta_j'-\omega_j^2\alpha_j|\right)^2}\leq C\eps_3\sqrt{E_+(t)}\gamma(t),
\end{multline*}
 Chosing  $\eps_3$ small enough, we get
\begin{equation}
 \label{A74}
E_+'(t)\geq 2\omega_1 E_+(t)-\frac{\omega_1}{2}\sqrt{E_+(t)}\gamma(t)
\end{equation} 
and similarly
\begin{equation}
 \label{A75}
E_-'(t)\leq -2\omega_1 E_-(t)+\frac{\omega_1}{2} \sqrt{E_-(t)} \gamma(t).
\end{equation}
\EMPH{Step 2} We show that for all $\tau\in [0,T_0)$,
\begin{equation}
 \label{A76} E_+(\tau)>E_-(\tau)\Longrightarrow \forall t\in [\tau,T_0),\; E_+(t)>E_-(t).
\end{equation} 
We argue by contradiction, assuming that there exists $\tau\in [0,T_0)$ and $t\in (\tau,T_0)$ such that $E_+(\tau)>E_-(\tau)$ and $E_+(t)\leq E_-(t)$. Let
\begin{equation}
 \label{A77}
\sigma:=\inf\left\{ t\in [\tau,T_0)\text{ s.t. } E_+(t)\leq E_-(t)\right\}>\tau.
\end{equation} 
If $t\in [\tau,\sigma]$, then $E_+(t)\geq E_-(t)$, and thus, by \eqref{A73}, $E_-(t)\leq \gamma^2(t)\leq E_+(t)$. Combining with \eqref{A74} (respectively \eqref{A75}) we get 
\begin{equation}
 \label{A78}
E_+'(t)\geq \frac{3}{2}\omega_1 E_+(t)
\end{equation} 
and 
\begin{equation}
 \label{A79}
E_-'(t)\leq \frac{1}{2}\omega_1 E_+(t)
\end{equation} 
Thus, if $t\in [\tau,\sigma]$, $E_+'(t)\geq E_-'(t)$, which contradicts the facts that $E_+(\tau)>E_-(\tau)$ and $E_+(\sigma)\leq E_-(\sigma)$. Step 2 is complete.

\EMPH{Step 3} We show \eqref{A71}. 
By Step 1, there exists $\tau\in [0, T_0)$ such that:
\begin{equation}
\label{A80} \forall t\in (0,\tau), \; E_-(t)\geq E_+(t)\quad\text{and} \quad \forall t\in (\tau,T_0),\; E_-(t)<E_+(t).
\end{equation} 
We take $\tau=0$ if $\forall t\in [0,T_0)\; E_-(t)<E_+(t)$ and $\tau=T_0$ if $\forall t\in [0,T_0), E_-(t)\geq E_+(t)$.

\EMPH{Estimate on $[0,\tau)$} Assume $\tau>0$. Then 
\begin{equation}
 \label{A81}
\forall t\in [0,\tau),\quad E_+(t)\leq \gamma^2(t)\leq E_-(t).
\end{equation} 
Hence by \eqref{A75},
$E_-'(t)\leq -\frac{3\omega_1}{2} E_-(t)$ which yields, (using \eqref{A73} again to get the last inequality) 
\begin{equation}
 \label{A82}
\forall t\in [0,\tau), \quad \gamma^2(t)\leq E_-(t)\leq e^{-\frac{3}{2}\omega_1 t}E_-(0)\leq 2 e^{-\frac{3}{2}\omega_1 t}\|\gamma\|^2_{\infty}.
\end{equation} 
As a consequence
\begin{equation}
 \label{A83}
\int_0^{\tau} \gamma(t)dt\leq \sqrt{2}\|\gamma\|_{\infty} \int_0^{\tau} e^{-\frac{3\omega_1}{4}t} \,dt\leq \frac{4\sqrt{2}}{3\omega_1}\|\gamma\|_{\infty}.
\end{equation} 

\EMPH{Estimate on $[\tau,T_0)$}
Assume $\tau <T_0$. Then
\begin{equation}
 \label{A84} \forall t\in [\tau,T_0),\quad E_-(t)\leq \gamma^2(t)\leq E_+(t).
\end{equation} 
By \eqref{A74}, $E_+'(t)\geq \frac{3\omega_1}{2} E_+(t)$, which gives, fixing $T\in (\tau,T_0)$, 
\begin{equation}
 \label{A85}
\forall t\in [\tau,T],\quad E_+(T)\geq e^{\frac{3}{2}\omega_1 (T-t)}E_+(t)\geq e^{\frac 32\omega_1(T-t)}\gamma^2(t).
\end{equation} 
As a consequence 
\begin{equation*}
 \int_{\tau}^T \gamma(t)dt\leq \int_{\tau}^T \sqrt{E_+(t)} e^{\frac{3}{4}\omega_1(t-T)} dt
\leq \sqrt{2}\|\gamma\|_{\infty} \int_{-\infty}^{T} e^{\frac{3}{4}\omega_1 (t-T)}\,dt
\leq \frac{4\sqrt{2}}{3\omega_1}\|\gamma\|_{\infty}.
\end{equation*}
Letting $T\to T_0$ and combining with \eqref{A83}, we get \eqref{A71} with $C_3=\frac{8\sqrt{2}}{3\omega_1}$.

\EMPH{Step 4}  In this step, we assume $T_0=+\infty$ and prove \eqref{A72}. We first note:
\begin{equation}
 \label{A86} \forall t>0,\; E_-(t)\geq E_+(t).
\end{equation} 
(in other words, the parameter $\tau$ of Step 3 is equal to $+\infty$).
If not, by Step 2, there exists $\tau>0$ such that for all $t>\tau$, $E_+(t)>E_-(t)$. Then by \eqref{A74} and \eqref{A84}, 
$E_+'\leq \frac{3\omega_1}{2}E_+$ on $[\tau,+\infty)$ which implies
$$ \forall t>\tau,\quad E_+(t)\geq e^{\frac{3}{2}\omega_1 (t-\tau)}E_+(\tau).$$
Since $E_+(\tau)>0$, this is a contradiction with the fact that $\gamma$ is bounded. Hence \eqref{A86}. As a consequence of \eqref{A86}, the estimate \eqref{A82} is valid on $[0,+\infty)$ which concludes the proof.
\end{proof}
\subsubsection{End of the proof of the exponential convergence}
We are now ready to conclude the proof of Proposition \ref{P:expo}. Let $u$ be as in this proposition. We proceed in several steps.

\EMPH{Step 1. Closeness to the stationary solution}
Recall from \eqref{A54} the definition of  $T_0$. We show (assuming that $r_0$ and $\delta_0/r_0$ are small enough) that $T_0=T_+$, where by definition $T_+=T_+(u)$. By the definition of $\vA(t)$, and since the function $\Psi$ of Lemma \ref{L:A13} is Lipschitz-continuous, we have
 \begin{equation}
  \label{A105}
  \forall t\in [0,T_0), \quad \|\vA(t)\|\leq Cr_0.
 \end{equation} 
 Chosing $r_0$ small, we see that \eqref{A63} and \eqref{A64} imply that $\alpha_j$ and $\beta_j$ satisfy the assumptions of Lemma \ref{L:A22}. Thus, by \eqref{A71},
 \begin{equation}
  \label{A106}
  \int_0^{T_0}\delta(t)\,dt\leq C_3\sup_{t\in [0,T_0]} \left(\delta(t)+\|\partial_t u(t)\|_{L^2}\right)\leq C\delta_0
\end{equation} 
(the second bound follows from \eqref{A59} and \eqref{A60}). By Lemma \ref{L:A20}, since by the assumptions of Proposition \ref{P:00}, $\|u(0)-Q\|\leq \delta_0$ ,
\begin{equation}
 \label{A107}
 \forall t\in [0,T_0), \quad |\vA(t)|\leq |\vA(0)|+C\delta_0\leq C'\delta_0,
\end{equation} 
for some constant $C'>0$. Recalling \eqref{A59}:
\begin{equation}
 \label{A108} \forall t\in [0,T_0), \quad \left\|u(t)-\theta_{\vA(t)}Q\right\|_{\hdot}\leq C\delta_0,
\end{equation} 
and combining with \eqref{A107}, and \eqref{thetaA_Q} in Corollary \ref{C:A16}, we obtain 
\begin{equation}
 \label{A109}
 \forall t\in [0,T_0),\quad \|u(t)-Q\|_{\hdot}\leq C\delta_0.
\end{equation} 
Taking $\delta_0/r_0$ small enough, we deduce:
\begin{equation}
 \label{A110}
 \forall t\in  [0,T_0),\quad \|u(t)-Q\|_{\hdot} \leq r_0/2,
\end{equation} 
which, by the definition of $T_0$, implies $T_0=T_+$, concluding Step 1.

\EMPH{Step 2. Global existence} We next show $T_+=+\infty$.
 
 Assume by contradiction that $T_+$ is finite. Since by Step 1, $T_+=T_0$, we have $\|u(t)-Q\|_{\hdot}<r_0$ for all $t\in [0,T_+)$. This gives a contradiction by a standard local well-posedness/stability result around $Q$ if $r_0$ is small enough.
 
 \EMPH{Step 3. Convergence to a stationary solution}
 We conclude the proof of Proposition \ref{P:expo}, proving that 
 there exists $\vA_0\in \RR^{N'}$, close to $0$, and $\omega,C>0$ such that $S=\theta_{\vA_0}(Q)$ satisfies
 $$\forall t\geq 0,\quad \left\|u(t)-S\right\|_{\hdot}+\|\partial_tu(t)\|_{L^2}\leq Ce^{-\omega t}.$$
 Indeed, by Lemma \ref{L:A22},
 \begin{equation}
  \label{A111} 
  \lim_{t\to +\infty} e^{\omega_1 t/2}\delta(t)=0.
 \end{equation} 
 By Lemma \ref{L:A20}, $\vA(t)$ has a limit $\vA_0$ as $t\to+\infty$, and
 \begin{equation}
  \label{A112}
  \|\vA(t)-\vA_0\|\leq Ce^{-\omega_1 t/2}.
 \end{equation} 
 Let $S=\theta_{\vA_0}(Q)$. Then
 \begin{multline*}
\|\partial_tu(t)\|_{L^2}+  \|u(t)-S\|_{\hdot}\leq \|\partial_tu(t)\|_{L^2}+\|u(t)-\theta_{\vA(t)}(Q)\|_{\hdot}+\|\theta_{\vA(t)}(Q)-\theta_{\vA_0}(Q)\|_{\hdot}
  \\ \leq C\delta(t)+\left\|Q-\theta_{\vA_0}^{-1}\theta_{\vA(t)}Q\right\|_{\hdot}\leq Ce^{-\omega_1t/2},
 \end{multline*}
 where the bound by $C\delta(t)$ at the second line follows from Lemma \ref{L:A19}.\qed

\subsection{Expansion of the solution}
\label{SS:expansion}
We next conclude the proof of Proposition \ref{P:00}, showing:
\begin{prop}
\label{P:expansion}
 Let $u$ be a solution of \eqref{CP} such that $T_+(u)=+\infty$, and there exists $S\in \Sigma$, $C,\eps>0$ such that
 \begin{equation}
  \label{prox_exp_Q}
  \forall t\geq 0,\quad \|u(t)-S\|_{\hdot}+\|\partial_tu(t)\|_{L^2}\leq Ce^{-\eps t}.
 \end{equation} 
 Then $u$ satisfies the conclusion of Proposition \ref{P:00}.
\end{prop}
\begin{remark}
In Proposition \ref{P:expansion}, we do not need to assume that $S$ satisfies the nondegeneracy assumption \eqref{ND}.
 \end{remark}
\begin{proof}
 \EMPH{Step 1}
 
 Let $Y_1,\ldots,Y_p$, $\omega_1,\ldots,\omega_p$, $Z_1,\ldots,Z_m$, $E_1,\ldots,E_m$ be defined in Section \S \ref{SS:coercivity} (with $Q=S$). Let
 \begin{equation*}
  \sigma_i(t)=\int Y_i(u-S),\quad \rho_j(t)=\int E_j (u-S),\quad i=1\ldots p,\; j=1,\ldots m.
 \end{equation*}
Define:
\begin{equation}
\label{exp_u}
 h=u-S,\quad g=u-S-\sum_{i=1}^p \sigma_i Y_i-\sum_{j=1}^m \rho_j Z_j,
\end{equation} 
and note that
\begin{equation}
 \label{ortho_g}
 \forall t\geq 0,\quad \int g Y_i=\int g E_j=0,\quad i=1\ldots p,\; j=1,\ldots m.
\end{equation} 
By energy conservation, $E(u,\partial_t u)=E(S,0)$. In view of expansion \eqref{exp_u}, we obtain
$$ \left|\frac{1}{2}\|\partial_t u\|^2_{L^2}+\Phi_S(g)-\frac{1}{2}\sum_{i=1}^p \omega_i^2\sigma_i^2\right|\leq C\|h\|^3_{L^{\frac{2N}{N-2}}},$$
(see \eqref{A60'}, \eqref{A61} for a similar argument) 
and thus, by \eqref{ortho_g} and Proposition \ref{P:A7},
\begin{equation}
\label{E1}
\|\partial_t u\|^2_{L^2} -\sum_{i=1}^p \omega_i^2\sigma_i^2+\frac{1}{C}\|g\|_{\hdot}^2\leq C\|h\|^3_{L^{\frac{2N}{N-2}}},
\end{equation} 
We let 
$$\omega=\sup\left\{a>0,\; \lim_{t\to +\infty} e^{ta}\left(\|\partial_tu(t)\|_{L^2}+\|h(t)\|_{\hdot}\right)=0\right\}\in [0,\infty].$$
By assumption \eqref{prox_exp_Q}, $\omega >0$. 
We note that
\begin{equation}
 \label{E1'}
 \omega=\sup\left\{a>0,\; \lim_{t\to +\infty} e^{ta}\sum_{j=1}^p|\sigma_j(t)|=0\right\}.
\end{equation} 
Indeed, let us 	temporarily denote by $\tilde{\omega}$ the right-hand side of \eqref{E1'}. Clearly, $\omega\leq \tilde{\omega}$. Let $a<\tilde{\omega}$. Then 
$$ \sum_{j=1}^p |\sigma_j(t)|\leq Ce^{-at}.$$
By \eqref{E1} and the definition of $\omega$,
$$ \|\partial_t u(t)\|_{L^2}+\|g(t)\|_{L^2}\leq Ce^{-at}+Ce^{-\frac{5}{4}\omega t}.$$
As a consequence, $|\rho'_j(t)|=\left|\int E_j \partial_tu(t)\right|\leq Ce^{-at}+Ce^{-\frac{5}{4}\omega t}$. Integrating, we get
$|\rho_j(t)|\leq C\left(e^{-at}+e^{-\frac{5}{4}\omega t}\right)$. Combining these estimates with the expansions \eqref{exp_u} we deduce $a<\omega$. Since $a$ is arbitrarily close to $\tilde{\omega}$, we deduce $\tilde{\omega}\leq \omega$, concluding the proof that $\omega=\tilde{\omega}$.

In the sequel, if $\omega$ is finite, we will denote by $\omega^-$ a positive number such that $\omega^-<\omega$, arbitrarily close to $\omega$ and that may change from line to line. 
If $\omega=\infty$, $\omega^-$ is a large positive constant that may change from line to line.

\EMPH{Step 2}

Let $j\in \{1,\ldots,p\}$. We prove that there exists $C>0$ such that
\begin{equation}
 \label{E2}
 \forall t\geq 0,\quad |\sigma_j(t)|+|\sigma_j'(t)|\leq C\max\left( e^{-\omega_jt},e^{-2\omega^-t}\right).
\end{equation} 
Furthermore, if $2\omega>\omega_j$, there exists $S_j\in \RR$, $C>0$ such that
\begin{equation}
 \label{E3}
 \forall t\geq 0,\quad \left|\sigma_j(t)-S_je^{-\omega_jt}\right|+\left|\sigma_j'(t)+\omega_j S_je^{-\omega_jt}\right|\leq Ce^{-2\omega^-t}.
\end{equation} 
Indeed, let $\sigma_{j,\pm}(t)=\sigma_j'(t)\pm \omega_j\sigma_j(t)$. Using 
$$\sigma_j''(t)=\int Y_j\partial_t^2u=-\int Y_jL_S(h)-\int Y_jR_S(h),$$
we get
\begin{equation*}
\left|\sigma_j''(t)-\omega_j^2\sigma_j(t)\right|\leq C\|h\|^2_{L^{\frac{2N}{N-2}}}\leq Ce^{-2\omega^-t}.
\end{equation*} 
Thus 
\begin{equation}
 \label{E3'}
 \left|\sigma_{j,\pm}'(t)\mp \omega_j \sigma_{j,\pm}(t)\right|\leq Ce^{-2\omega^-t}.
\end{equation} 
Let us prove \eqref{E2}. We have
$$\left|\frac{d}{dt}\left(e^{-\omega_j t}\sigma_{j,+}(t)\right)\right|\leq Ce^{-(2\omega^-+\omega_j)t}.$$
Integrating between $t$ and $+\infty$, we get
\begin{equation}
 \label{E4}
 \left|\sigma_{j,+}(t)\right|\leq Ce^{-2\omega^-t}.
\end{equation} 
Similarly
\begin{equation}
 \label{E5}
 \left|\frac{d}{dt}\left( e^{\omega_jt}\sigma_{j,-}(t) \right)\right|\leq Ce^{(\omega_j-2\omega^-)t}
\end{equation} 
Integrating between $0$ and $t$, we obtain
\begin{equation}
 \label{E5'}
 \left|\sigma_{j,-}(t)\right|\leq Ce^{-\omega_jt}+Ce^{-2\omega^-t}.
\end{equation}
Combining \eqref{E4} and \eqref{E5'}, we obtain \eqref{E2}.

Next, we assume $2\omega>\omega_j$ and prove \eqref{E3}. We can take $\omega^-<\omega$ so that $2\omega^->\omega_j$. By \eqref{E5}, $e^{\omega_jt}\sigma_{j,-}(t)$ has a limit $\ell_j$ as $t\to +\infty$, and
\begin{equation}
\label{E6}
\left|\sigma_{j,-}(t)-e^{-\omega_jt}\ell_j\right|\leq Ce^{-2\omega^-t}.
\end{equation} 
Combining \eqref{E4} and \eqref{E6}, we get \eqref{E3} with $S_j=-\frac{1}{2\omega_j}\ell_j$. Step 2 is complete.

\EMPH{Step 3} In this step, we assume that 
\begin{equation}
\label{Assu_Step2}
\forall j\in\{1,\ldots,p\},\quad\omega_j\neq \omega,  
\end{equation} 
and we prove that $u\equiv S$. We will need the following Claim, whose proof is postponed to the appendix.
\begin{claim}
 \label{C:uniq_expo}
 Let $S\in \Sigma$. There exists $\nu=\nu(S)>0$ such that for all solutions $u$ of \eqref{CP} such that $T_+(u)=+\infty$ and 
 \begin{equation}
 \label{uniq_expo}
 \lim_{t\to +\infty} e^{\nu t}\left(\|u(t)-S\|_{\hdot}+\|\partial_tu(t)\|_{L^2}\right)=0,  
 \end{equation} 
 we have $u\equiv S$.
\end{claim}
By Claim \ref{C:uniq_expo}, it is sufficient to show $\omega=\infty$. We prove this by contradiction, assuming that $\omega<\infty$. We will prove that there exists $\omega^+>\omega$ such that, for all $j\in \{1,\ldots,p\}$,
\begin{equation}
 \label{E7}
 \quad |\sigma_j(t)|+|\sigma_j'(t)|\leq Ce^{-\omega^+t},
\end{equation} 
contradicting the definition \eqref{E1'} of $\omega$. 

Let $j\in \{1,\ldots,p\}$. We distinguish two cases.

If $\omega<\omega_j$, then \eqref{E7} follows from \eqref{E2}, taking $\omega<\omega^+<\max(\omega_j,2\omega^-)$.

Assume now $\omega\geq \omega_j$. In this case, \eqref{E3} holds. Furthermore, by assumption \eqref{Assu_Step2}, $\omega>\omega_j$, and thus
$$\lim_{t\to\infty}e^{\omega_jt}\left(|\sigma_j(t)|+|\sigma_j'(t)|\right)=0.$$
As a consequence, the real number $S_j$ in \eqref{E3} must be $0$, wich proves again \eqref{E7}, concluding Step 2.

\EMPH{Step 4} In this step, we assume that there exists $k\in \{1,\ldots,p\}$ such that $\omega=\omega_k$. We define the following subsets $J_0,J_+$ and $J_-$ of $\{1,\ldots,p\}$:
$$J_0=\Big\{j,\; \omega_j=\omega\Big\},\quad J_-=\Big\{j,\; \omega_j<\omega\Big\},\quad J_+=\Big\{j,\; \omega_j>\omega\Big\}.$$
We first prove that there exists $\omega^+>\omega$ such that
\begin{equation}
 \label{E9}
 j\in J_+\cup J_-\Longrightarrow \exists C>0,\; \forall t\geq 0,\; \left|\sigma_j(t)\right|+\left|\sigma_j'(t)\right|\leq Ce^{-\omega_+t}.
\end{equation} 
If $j\in J_+$, then \eqref{E9} follows from \eqref{E2}. If $j\in J_-$, then \eqref{E3} holds. By the definitions of $J_-$ and $\omega$, we must have $S_j=0$ in \eqref{E3}, and \eqref{E9} follows again.

We next notice that if $j\in J_0$, then \eqref{E3} holds for some $S_j\in \RR$. Furthermore, in view of \eqref{E9} and the definition \eqref{E1'} of $\omega$, there exists $j\in J_0$ such that $S_j\neq 0$.  We let
$$ Y(x)=\sum_{j\in J_0}S_jY_j,$$
and note that $Y\neq 0$ and $L_QY=-\omega^2Y$. By \eqref{E3}, \eqref{E9} and the considerations above,
\begin{equation}
 \label{E9'}
\sum_{j=1}^p \left|{\sigma'_j(t)}^2-\omega_j^2\sigma_j^2(t)\right|\leq Ce^{-2\omega^+t}
\end{equation} 
and
\begin{equation}
 \label{E10}
 \forall t\geq 0,\quad \left\|\sum_{j=1}^p \sigma_j(t)Y_j-e^{-\omega t}Y\right\|_{\hdot}+\left\|\sum_{j=1}^p \sigma_j'(t)Y_j +\omega e^{-\omega t}Y\right\|_{L^2} \leq Ce^{-\omega^+t}.
\end{equation} 
It remains to prove:
\begin{equation}
 \label{E10'}
 \forall t\geq 0,\quad \|g\|_{\hdot}+\sum_{j=1}^m |\rho_j(t)|+\left\|\partial_t u(t)-\sum_{j=1}^p \sigma_j'(t)Y_j\right\|_{L^2}\leq Ce^{-\omega^+t}.
\end{equation}
Note that $\partial_tu=\partial_th=\sum_{i=1}^p \sigma_i'Y_i+\sum_{j=1}^m\rho'_jZ_j+\partial_tg$, and, from \eqref{ortho_g} and the orthogonality properties of the functions $Y_i$ and $Z_j$,
$$\forall i\in \{1,\ldots, p\}, \quad \int Y_i\left(\sum_{j=1}^m\rho'_jZ_j+\partial_tg\right)=0.$$
As a consequence,
\begin{equation}
 \label{E11}
 \left\|\partial_t u\right\|_{L^2}^2 =\sum_{i=1}^p {\sigma_i'}^2+\left\|\sum_{j=1}^m\rho_j'Z_j+\partial_t g\right\|^2_{L^2}.
\end{equation} 
By \eqref{E1}, 
$$\left\|\sum_{j=1}^m \rho'_jZ_j +\partial_t g\right\|_{L^2}^2+\sum_{i=1}^p{\sigma_i'}^2-\sum_{i=1}^p \omega_i^2\sigma_i^2+\frac{1}{C}\|g\|^2_{\hdot}\leq Ce^{-3\omega^-t}.$$
Combining with \eqref{E9'}, we deduce
\begin{equation}
\label{E12}
\|g\|^2_{\hdot}+\left\|\sum_{j=1}^m \rho'_jZ_j+\partial_tg\right\|_{L^2}^2\leq e^{-2\omega^+t}.
\end{equation} 
Since for $k=1\ldots m$, $\rho'_k=\int E_k(\sum_{j=1}^m\rho'_jZ_j+\partial_tg)$, we obtain $|\rho'_k|\leq Ce^{-\omega^+t}$, and thus
\begin{equation}
 \label{E13}
 \left|\rho_k(t)\right|\leq Ce^{-\omega^+t}.
\end{equation} 
Combining \eqref{E11}, \eqref{E12} and \eqref{E13} we get \eqref{E10'}, which concludes the proof of Proposition \ref{P:expansion}.
\end{proof}
\section{Lorentz transformation}
\label{S:Lorentz}
In this section we prove that the Lorentz transform of a solution of \eqref{CP} with the compactness property is well-defined and is a solution of \eqref{CP} with this property. We consider without loss of generality a Lorentz transformation with a parameter $\vell$ which is parallel to $\ve_1=(1,0,\ldots,0)$. We first need some notations. Let $\ell\in (-1,+1)$. If $(s,y)\in \RR\times \RR^N$ we define $(t,x)$ by
\begin{equation}
 \label{def_phil}
(t,x)=\phi_{\ell}(s,y)=\left(\frac{s+\ell y_1}{\sqrt{1-\ell^2}},\frac{y_1+\ell s}{\sqrt{1-\ell^2}},y'\right),
 \end{equation} 
 where $x'=(x_2,\ldots,x_N)$, $y'=(y_2,\ldots,y_N)$. Thus $(s,y)=\phi_{\ell}^{-1}(t,x)=\left(\frac{t-\ell x_1}{\sqrt{1-\ell^2}},\frac{x_1-\ell t}{\sqrt{1-\ell^2}},x'\right)$.
 We recall that the Lorentz transform of a global solution $u$ is defined by
\begin{equation}
 \label{def_ul}
 u_{\ell}(t,x)=u\left(\frac{t-\ell x_1}{\sqrt{1-\ell^2}},\frac{x_1-\ell t}{\sqrt{1-\ell^2}},x'\right)=u\left(\phi_{\ell}^{-1}(t,x)\right).
\end{equation} 
 When $u$ is global in time, $u_{\ell}$ is well-defined (say, as an element $L_{\loc}^{\frac{2(N+1)}{N-2}}(\RR^{N+1})$) and, at least in the distributional sense, a solution of \eqref{CP}. In this section we prove that $u_{\ell}$ is indeed a solution of \eqref{CP} in the usual sense:
\begin{lemma}
\label{L:Lorentz_global}
 Let $u$ be a global, finite-energy solution of \eqref{CP}, and $\ell\in (-1,+1)$. Then $u_{\ell}$ is a global, finite-energy solution of \eqref{CP}.
\end{lemma} 
The conclusion of Lemma \ref{L:Lorentz_global} seems to be folklore knowledge and has been used in the literature before without any proof that we were aware of. We provide a proof here to close this apparent gap.

 Next consider a solution $u=u(s,y)$ of \eqref{CP} with the compactness property,
 and $(s^-,s^+)$ its maximal interval of existence. Recall that $s^+=+\infty$ or $s^-=-\infty$. If $s^+<\infty$, then $s^-=-\infty$ and by \cite[Lemma 4.8]{KeMe08}, there exists $y^+\in \RR^N$ such that $u$ is supported in the cone
 $$ \left\{(s,y)\in (-\infty,s^+)\times \RR^N \text{ s.t. } \left|y-y^+\right|\leq \left|s-s^+\right|\right|\}.$$
 In this case we call $y^+$ the \emph{blow-up point} of $u$ for positive time, and let
 $$ (t^+,x^+)=\phi_{\ell}(s^+,y^+)$$
 If $t^-$ is finite, we define similarly $y^-$, the blow-up point of $u$ for negative times, and let
 $$ (t^-,x^-)=\phi_{\ell}(s^-,y^-).$$
 If $s^{\pm}=\pm \infty$ we let $t^{\pm}=\pm \infty$. If the solution $u$ is not global in time, we extend it to $\RR$ as a function $\unu$ letting:
 \begin{equation*}
 \begin{cases}
 \unu(s,y)=0  &\text{ if } s\notin \Imax(u)\\
 \unu(s,y)=u(s,y)  &\text{ if } s\in \Imax(u).
 \end{cases}
\end{equation*} 
The main result of this section is the following proposition:
 \begin{prop}
\label{P:lorentz}
Let $u$ be a solution of \eqref{CP} with the compactness property, $s^{\pm}$, $y^{\pm}$, $t^{\pm}$, $x^{\pm}$ and $\unu$ be as above. Let 
 \begin{equation}
 \label{def_ul_bis}
 u_{\ell}(t,x)=\unu\left(\frac{t-\ell x_1}{\sqrt{1-\ell^2}},\frac{x_1-\ell t}{\sqrt{1-\ell^2}},x'\right),\quad t\in (t^-,t^+),\; x=(x_1,x')\in \RR\times \RR^{N-1}. 
 \end{equation} 
 Then $u_{\ell}$ is a solution of \eqref{CP} with the compactness property, with maximal interval of existence $(t^-,t^+)$.
\end{prop}

\begin{remark}
 If $t^{+}$ is finite, then $x^+$ is the blow-up point of $u_{\ell}$ for positive time. A similar statement holds for negative times.
\end{remark}

\begin{remark}
In this paper, we will only need to apply the Lorentz transformation to solutions of \eqref{CP} with the compactness property. Let us mention that is always possible, adapting the argument of this section, to define the Lorentz transform of a solution which is global in at least one time direction.
\end{remark}
\begin{remark}
In Lemma \ref{L:Lorentz_global} and in Proposition \ref{P:lorentz}, as well as in all this paper, a \emph{solution} of \eqref{CP} is a solution of \eqref{CP} in the sense of Definition \ref{D:solution}.
\end{remark}

\subsection{Lorentz transform of global solutions}
In this subsection we prove Lemma \ref{L:Lorentz_global}. We start with the easier case of scattering solutions.
\begin{lemma}
\label{L:scattering_Lorentz}
 Let $u$ be a global solution of \eqref{CP} scattering in both time directions. Let $u_{\ell}$ be defined by \eqref{def_ul}. Then $u_{\ell}$ is a global solution of \eqref{CP}, scattering in both times directions.
\end{lemma}
We first recall from \cite[Lemma 2.2 and Remark 2.3]{KeMe08} the following claim:
\begin{claim}
\label{C:linear_Lorentz}
Let  $h\in L^1(\RR,L^2(\RR^N))$, $(w_0,w_1)\in \hdot\times L^2$, $\ell\in (-1,+1)$ and 
\begin{equation}
\label{LCP}
w(t)=\cos(t\sqrt{-\Delta})w_0+\frac{\sin(t\sqrt{-\Delta})}{\sqrt{-\Delta}}w_1+\int_0^t\frac{\sin\left( (t-s)\sqrt{-\Delta} \right)}{\sqrt{-\Delta}}h(s)\,ds,\quad t\in \RR.
\end{equation}
Then $(w_{\ell},\partial_tw_{\ell})\in C^0\left(\RR,\hdot\times L^2\right)$ and there is a constant $C_{\ell}$ (depending only on $\ell$) such that
\begin{equation*}
 \sup_{t} \left\|(w_{\ell}(t),\partial_tw_{\ell}(t)\right\|_{\hdot\times L^2}\leq C_{\ell}\left( \|(w_0,w_1)\|_{\hdot\times L^2}+\|h\|_{L^1(\RR,L^2)} \right).
\end{equation*} 
\end{claim}

\begin{proof}[Proof of Lemma \ref{L:scattering_Lorentz}]
 Since $u$ scatters in both time directions, we get 
 $$ \left\||u|^{\frac{4}{N-2}}u\right\|_{L^1_tL^2_x}=\|u\|^{\frac{N+2}{N-2}}_{L^{\frac{N+2}{N-2}}_tL^{\frac{2(N+2)}{N-2}}_x}<\infty.$$
 By Claim \ref{C:linear_Lorentz}, $(u_{\ell},\partial_tu_{\ell}) \in C^0\left(\RR,\hdot\times L^2\right)$. Furthermore, since $u\in L^{\frac{2(N+1)}{N-2}}(\RR^{N+1})$, $u_{\ell}\in L^{\frac{2(N+1)}{N-2}}(\RR^{N+1})$. If $(u_0,u_1)\in (C^{\infty}_0(\RR^N))^2$, it is easy to see that $(u_{\ell}(0),\partial_tu_{\ell}(0))\in(C^{\infty}_0(\RR^N))^2$ and that $u_{\ell}$ satisfies \eqref{CP} in the classical sense, so that $u_{\ell}$ is a solution of \eqref{CP} (in the sense of Definition \ref{D:solution}). In the general case, we will use Claim \ref{C:solution_sequence} to prove that $u_{\ell}$ is a solution.
 
 Let $(u_{0}^k,u_{1}^k)\in (C^{\infty}_0(\RR^N))^2$ such that
 \begin{equation}
  \label{approx_u}
  \lim_{k\to\infty} \left\|\left( u_0^k,u_1^k \right)-\left( u_0,u_1 \right)\right\|_{\hdot\times L^2}=0.
 \end{equation} 
 Let $u^k$ be the solution of \eqref{CP} with initial data $\left( u_0^k,u_1^k \right)$. Then by long-time perturbation theory (see \cite{KeMe08}), $u^k$ is global for large $k$ and 
 \begin{equation}
  \label{lim_Strichartz}
  \lim_{k\to\infty} \left\|u^k-u\right\|_{L^{\frac{2(N+1)}{N-2}}(\RR^{N+1})\cap L^{\frac{N+2}{N-2}}\big(\RR,L^{\frac{2(N+2)}{N-2}}(\RR^N)\big)}=0.
 \end{equation} 
 Since 
 \begin{multline*}
u^k-u=\cos\left( \sqrt{-\Delta}t \right)(u^k_0-u_0)+\frac{\sin\left( \sqrt{-\Delta}t \right)}{\sqrt{-\Delta}}(u^k_1-u_1)\\
+\int_0^t \frac{\sin\left( (s-t)\sqrt{-\Delta} \right)}{\sqrt{-\Delta}}\left(|u^k|^{\frac{4}{N-2}}u^k(s)-|u|^{\frac{4}{N-2}}u(s)\right)\,ds,  
 \end{multline*}
we deduce from \eqref{lim_Strichartz} and Claim \ref{C:linear_Lorentz}
 \begin{equation}
  \label{lim_energy}
  \sup_{t\in \RR} \left\|\left( u_{\ell}^k-u_{\ell},\partial_t u_{\ell}^k-\partial_tu_{\ell}\right)(t) \right\|_{\hdot\times L^2}\underset{k\to\infty}{\longrightarrow}0.
 \end{equation} 
 Since $u^k_{\ell}$ is a solution of \eqref{CP} in the sense of Definition \ref{D:solution}, and 
 $$\left\|u^k_{\ell}\right\|_{L^{\frac{2(N+1)}{N-2}}(\RR^{N+1})}=\left\|u^k\right\|_{L^{\frac{2(N+1)}{N-2}}(\RR^{N+1})}$$
 is, by \eqref{lim_Strichartz}, uniformly bounded, we get by Claim \ref{C:solution_sequence} that $u_{\ell}$ is a solution of \eqref{CP}, concluding the proof.
\end{proof}

We next prove Lemma \ref{L:Lorentz_global}.

Note that $u_{\ell}$ is well-defined, as an element of $L^{\frac{2(N+1)}{N-2}}_{\loc}(\RR^{N+1})$. 
 
 We denote by $(s,y)$ the time and space variables for $u$ and $(t,x)$ the time and spaces variable for $u_{\ell}$. This variables are related by \eqref{def_phil}. We note that
 \begin{equation}
  \label{B0}
  |x|^2-t^2=|y|^2-s^2
 \end{equation} 
 and
 \begin{equation}
  \label{B1}
 |s|+|y|\leq c_{\ell} (|t|+|x|),\quad |t|+|x|\leq c_{\ell} (|s|+|y|),\text{ where }c_{\ell}=\sqrt{\frac{1+|\ell|}{1-|\ell|}}.
 \end{equation} 
 
 \EMPH{Step 1. Estimate at infinity}
We prove that there exists $B>0$ and a scattering solution $v$ of \eqref{CP} such that 
 \begin{equation}
 \label{B1'}
 |x|\geq |t|+B\Longrightarrow u_{\ell}(t,x)=v(t,x)\text{ a.e.}
 \end{equation} 
 Let $A>0$ be a large constant. Denote by $\chi_A(y)=\chi\left( \frac{y}{A} \right)$, where
 $$\chi\in C^{\infty}(\RR^N),\quad \chi(y)=1\text{ if }|y|\geq 1,\quad \chi(y)=0\text{ if }|y|\leq \frac{1}{2}.$$
 Let 
 \begin{equation*}
  (\tu_0,\tu_1)=(\chi_Au_0,\chi_Au_1).
 \end{equation*} 
 Let $\delta_0>0$ be the small constant given by the small data theory. We choose $A$ large, so that $\|(\tu_0,\tu_1)\|_{\hdot\times L^2}<\delta_0$. Let $\tu$ be the solution of \eqref{CP} with initial data $(\tu_0,\tu_1)$ at $s=0$. By the small data theory, $\tu$ is a scattering solution of \eqref{CP}. By Lemma \ref{L:scattering_Lorentz}, the Lorentz transform $\tu_{\ell}$ of $\tu$ is a scattering solution of \eqref{CP}. Furthermore, by finite speed of propagation,
 \begin{equation}
  \label{B2} \tu(s,y)=u(s,y)\text{ i.e. }\tilde{u}_{\ell}(t,x)=u_{\ell}(t,x)\text{ for almost all }(s,y)\text{ s.t. }|y|\geq A+|s|.
 \end{equation} 
 We claim
 \begin{equation}
  \label{B3} |x|\geq c_{\ell}A+|t|\Longrightarrow |y|\geq A+|s|.
 \end{equation} 
 Indeed, by \eqref{B0}, \eqref{B1},
 $$|y|-|s|=\frac{|y|^2-s^2}{|y|+|s|}=\frac{|x|^2-t^2}{|y|+|s|}\geq \frac{|x|^2-t^2}{c_{\ell}(|x|+|t|)}=\frac{|x|-|t|}{c_{\ell}},$$
 and \eqref{B3} follows. The desired conclusion \eqref{B1'}, with $v=\tilde{u}_{\ell}$ and $B=c_{\ell}A$, follows from \eqref{B2} and \eqref{B3}. Note that \eqref{B1'} implies
 \begin{equation}
  \label{step1_L5L10}
  \int_{\RR}\left(\int_{|x|\geq |t|+B} \left|u_{\ell}(t,x)\right|^{\frac{2(N+2)}{N-2}}\,dx\right)^{\frac{1}{2}}\,dt<\infty.
 \end{equation} 
 
 \EMPH{Step 2. Local estimate} Let $(T,X)\in \RR\times \RR^N$. We show that there exists $\eps>0$ and a scattering solution $v$ of \eqref{CP} such that
 \begin{equation}
  \label{goal_step2} 
  |x-X|\leq \eps-|t-T|\Longrightarrow u_{\ell}(t,x)=v(t,x).
 \end{equation} 
 Indeed, let  $(S,Y)=\Phi_{\ell}^{-1}(T,X)$. Let $\psi\in C_0^{\infty}(\RR^N)$ such that $\psi(y)=1$ if $|y|\leq 1$ and $\psi(y)=0$ if $|y|\geq 2$. Let $(\tu_0,\tu_1)=\psi\left( \frac{y-Y}{\eta} \right)\left(u(S,y),\partial_tu(S,y)\right)$. Choose $\eta>0$ small, so that $\left\|(\tu_0,\tu_1)\right\|_{\hdot\times L^2}\leq \delta_0$ ($\delta_0$ is again given by the small data theory). Let $\tu$ be the solution of \eqref{CP} with data:
 $$ (\tu(S),\partial_t\tu(S))=(\tu_0,\tu_1).$$
 Then $\tu$ is globally defined and scatters. By Lemma \ref{L:scattering_Lorentz}, its Lorentz transform $\tu_{\ell}$ is a scattering solution of \eqref{CP}. Note that by finite speed of propagation,
 \begin{equation}
  \label{B2'}
  |y-Y|\leq \eta-|s-S|\Longrightarrow u_{\ell}(t,x)=\tu_{\ell}(t,x).
 \end{equation} 
 Furthermore, by \eqref{B1}, $|y-Y|+|s-S|\leq c_{\ell}\left(|x-X|+|t-T|\right)$, and thus
 \begin{equation}
  \label{B3'} |x-X|\leq \frac{\eta}{c_{\ell}}-|t-T|\Longrightarrow |y-Y|\leq \eta-|s-S|.
 \end{equation} 
 The desired conclusion \eqref{goal_step2} follows from \eqref{B2'} and \eqref{B3'} with $v=\tu_{\ell}$ and $\eps=\eta/c_{\ell}$.
 Again, \eqref{goal_step2} implies
 \begin{equation}
  \label{step2_L5L10}
  \int_{T-\eps}^{T+\eps}\left(\int_{|x-X|\leq \eps-|t-T|} \left|u_{\ell}(t,x)\right|^{\frac{2(N+2)}{N-2}}\,dx\right)^{\frac{1}{2}}\,dt<\infty.
 \end{equation}

 \EMPH{Step 3. End of the proof} 
Combining Step 1 and 2, we get that $(u_{\ell},\partial_t u_{\ell})\in C^0(\RR,\hdot\times L^2)$. By \eqref{step1_L5L10} and \eqref{step2_L5L10}, $u_{\ell}\in L^{\frac{N+2}{N-2}}_{\loc}\left(\RR,L^{\frac{2(N+2)}{N-2}}(\RR^N)\right)$. Furthermore, by Steps 1 and 2 and Remark \ref{R:distrib}, $u_{\ell}$ satisfies $\partial_t^2u_{\ell}-\Delta u_{\ell}=u_{\ell}^5$ in the distributional sense, which (in view of Lemma \ref{L:sol_distrib}) yields the result.
\qed
\subsection{Lorentz transform of nonglobal solutions with the compactness property}
Recall that a solution with the compactness property is global in at least one time direction. We treat the case of solutions that are global backward in time, an analoguous result holds for solutions that are global forward in time.
\begin{lemma}
\label{L:Lorentz_nonglobal}
 Let $u$ be a solution with the compactness property. Assume that the maximal interval of existence of $u$ is of the form $(-\infty,s^+)$ with $s^+\in \RR$. Let $y^+,t^+,x^+,\unu$ be as in the introduction of Section \ref{S:Lorentz}. Define $u_{\ell}$ by \eqref{def_ul_bis}. Then $u_{\ell}$ is a solution of \eqref{CP}, with maximal interval of existence $(-\infty,t^+)$, and such that
 $$\supp u_{\ell}\subset \Big\{(t,x)\in \RR^{N+1}\text{ s.t. } t<t^+ \text{ and } \left|x-x^+\right|\leq \left|t-t^+\right|\Big\}.$$
\end{lemma}
The proof is very close to the proof of Lemma \ref{L:Lorentz_global} and we will only sketch it. To deal with the fact that the solution is not global in time, we will need the following technical claim:
\begin{claim}
 \label{C:FSP}
 Let $u$ be as in Lemma \ref{L:Lorentz_nonglobal}. Let $\tu$ be a globally defined, scattering solution of \eqref{CP}. Let $S\in (-\infty,s^+)$.
 \begin{enumerate}
  \item \label{I:FSPloc}Let $Y\in \RR^N$ and assume 
 $$ (\tu,\partial_t\tu)(S,y)=(u,\partial_tu)(S,y)\text{ if }|y-Y|<\eps.$$
 Let $D=D(S,Y,\eps)=\left\{(s,y)\in \RR\times \RR^N\text{ s.t. } |y-Y|< \eps-|s-S|\right\}$. Then $(s^+,y^+)\notin D$ and
 \begin{equation}
 \label{equal}
 (s,y)\in D\Longrightarrow \tu(s,y)=\unu(s,y). 
 \end{equation} 
\item \label{I:FSPinfinity} Let $A>0$ and assume
 $$ (\tu,\partial_t\tu)(S,y)=(u,\partial_tu)(S,y)\text{ if }|y|>A.$$
 Let $D'=D'(S,A)=\left\{(s,y)\in \RR\times \RR^N\text{ s.t. } |y|>A+|s-S|\right\}$. Then $(s^+,y^+)\notin D'$ and
 \begin{equation}
 \label{equal'}
 (s,y)\in D'\Longrightarrow \tu(s,y)=\unu(s,y). 
 \end{equation} 
 \end{enumerate}
 \end{claim}
\begin{proof}
We prove only \eqref{I:FSPloc}. The proof of \eqref{I:FSPinfinity} is very similar.

If $s^+>S+\eps$, then $D\subset (-\infty,s^+)\times \RR^N$ and $u=\unu$ on $D$. The conclusion \eqref{equal} follows immediately by finite speed of propagation.
 
 Assume $S<s^+\leq S+\eps$. By finite speed of propagation, 
 \begin{equation}
 \label{equal1}
 \Big((s,y)\in D\text{ and }s<s^+\Big)\Longrightarrow \tu(s,y)=u(s,y)=\unu(s,y). 
 \end{equation} 
 Since $\tu$ scatters, we have $\|\tu\|_{L^{\frac{2(N+1)}{N-2}}(D)}<\infty$ and thus by \eqref{equal1} $\|u\|_{L^{\frac{2(N+1)}{N-2}}(D\cap \{s<s^+\})}<\infty$. By the finite-time blow-up criterion, $\|u\|_{L^{\frac{2(N+1)}{N-2}}(|y-y^+|<|s-s^+|)}=+\infty$. Thus $(s^+,y^+)\notin D$. We deduce that 
 $$ \forall \varphi\in C^{\infty}_0(D),\quad (\varphi \unu,\varphi \partial_t \unu)\in C^0(\RR,\hdot\times L^2).$$
 By \eqref{equal1} and a continuity argument,
 \begin{equation}
 \label{equal2}
 \forall y\in \RR^N,\quad 
 (s^+,y)\in D\Longrightarrow \tu(s^+,y)=\unu(s^+,y)=0. 
 \end{equation} 
 By finite speed of propagation, $\tu(s,y)=0$ if $(s,y)\in D$ and $s\geq s^+$. The proof is complete.
\end{proof}

\begin{proof}[Proof of Lemma \ref{L:Lorentz_nonglobal}]
 \EMPH{Step 1}
 We first notice that
\begin{equation}
 \label{suppul}
\end{equation} 
 $$\supp u_{\ell}\subset \left\{ (t,x)\in (-\infty,t^+)\times \RR^N\text{ s.t. } |x-x^+|\leq |t-t^+|\right\}.$$
 Indeed 
 $$\supp \unu\subset \left\{ (s,y)\in \RR\times \RR^N\text{ s.t. } |y-y^+|\leq |s-s^+|\right\}.$$
 By \eqref{B0},
$$ |x-x^+|\leq |t-t^+|\iff |y-y^+|\leq |s-s^+|,$$
and the claim follows from the definition of $u_{\ell}$.

\EMPH{Step 2} Let $T<t^+$, $X\in \RR^N$. We prove that there exists $\eps>0$ with $\eps<t^+-T$, and a scattering solution $v$ of \eqref{CP} such that 
$$|x-X|<\eps-|t-T|\Longrightarrow u_{\ell}(t,x)=v(t,x).$$
As in Step 2 in the proof of Lemma \ref{L:Lorentz_global}, we let $(S,Y)=\phi_{\ell}^{-1}(T,X)$. We distinguish three cases.

\EMPH{Case 1: $S>s^+$} In this case $(S,Y)$ is not in the support of $\underline{u}$, or equivalently $(T,X)$ is not in the support of $u_{\ell}$. Thus $u_{\ell}=0$ in a neighborhood of $(T,X)$ and the conclusion of Step 2 is obvious.

\EMPH{Case 2: $S=s^+$} We cannot have $Y=y^+$ (which would imply $T=t^+$, contradicting our assumptions). Thus again $(S,Y)$ is not in the support of $\unu$, which implies that $(T,X)$ is not in the support of $u_{\ell}$. Again, the conclusion of Step 2 is obvious.

\EMPH{Case 3: $S<s^+$} Then $(\unu(S),\partial_t \unu(S))=(u(S),\partial_t u(S))\in \hdot\times L^2$. The same arguments as in Step 2 of the proof of Lemma \ref{L:Lorentz_global} (using Claim \ref{C:FSP}) yields the desired conclusion. We omit the details.

\EMPH{Step 3. Conclusion of the proof} By Step $1$, for all $t$ in $(-\infty,t^+)$, $(u_{\ell}(t),\partial_tu_{\ell}(t))$ is compactly supported. By Step 2, 
$$(u_{\ell},\partial_t u_{\ell})\in C^0((-\infty,t^+),\hdot\times L^2)\text{ and }u_{\ell}\in L^{\frac{N+2}{N-2}}_{\loc}\left((-\infty,t^+),L^{\frac{2(N+2)}{N-2}}(\RR^N)\right).$$
Again (using Remark \ref{R:distrib}), $u_{\ell}$ satisfies $\partial_t^2u_{\ell}-\Delta u_{\ell}=|u_{\ell}|^{\frac{4}{N-2}}u_{\ell}$ in the distributional sense on $(-\infty,t^+)$. Thus by Lemma \ref{L:sol_distrib}, $u_{\ell}$ is a solution of \eqref{CP} on $(-\infty,t^+)$. The fact that $(-\infty,t^+)$ is the maximal time of existence of $u_{\ell}$ follows from the inclusion \eqref{suppul} of Step 1.
\end{proof}

\subsection{Continuity of the Lorentz transformation in the energy space}
We next prove the following continuity fact:
\begin{lemma}
\label{L:Lorentz_continuity}
 Let $\{u_n\}_n$ be a sequence of non-zero solutions of \eqref{CP}. Let $I_{\max}(u_n)$ be the maximal interval of existence of $u_n$ and assume $0\in I_{\max}(u_n)$. Assume furthermore the following uniform compactness property: there exist $\mu_n(s)$ and $y_n(s)$ (defined for $n\in \NN$, $s\in I_{\max}(u_n)$), such that the set
 $$K=\left\{\left( \mu_n^{N/2-1}(s) u_n\left(s,\mu_n(s)\cdot+y_n(s)\right), \mu_n^{N/2}(s) \partial_su_n\left(s,\mu_n(s)\cdot+y_n(s)\right)\right),\; n\in \NN,\; s\in I_{\max}(u_n)\right\}$$
 has compact closure in $\hdot\times L^2$. 
 Let $(u_{0n},u_{1n})=(u_n(0),\partial_t u_n(0))$. Assume that there exists $(u_0,u_1)\in \hdot\times L^2$ such that
 \begin{equation}
 \label{CV_un}
 \lim_{n\to\infty} \left\| (u_{0n},u_{1n})-(u_0,u_1)\right\|_{\hdot\times L^2}=0. 
 \end{equation} 
 Let $u$ be the solution of \eqref{CP} such that $(u,\partial_tu)(0)=(u_0,u_1)$. Then
 \begin{enumerate}
  \item \label{I:compact} $u$ has the compactness property.
  \item \label{I:Lorentz_continuity} If $\ell\in (-1,+1)$ and $t\in I_{\max}(u_{\ell})$, then $t\in I_{\max}(u_{n\ell})$ for large $n$ and
$$\lim_{n\to\infty} \left\| (u_{n\ell},\partial_t u_{n\ell})(t)-(u_{\ell},\partial_t u_{\ell})(t)\right\|_{\hdot\times L^2}=0.$$
 \end{enumerate}
\end{lemma}

\begin{proof}[Proof of Lemma \ref{L:Lorentz_continuity}]
 
\EMPH{Proof of point \eqref{I:compact}}

It is classical. We give a proof for the sake of completeness. We fix $s\in I_{\max}(u)$. 

By \eqref{CV_un} and the continuity of the flow of \eqref{CP}, $s\in I_{\max}(u_n)$ for large $n$ and
\begin{equation}
 \label{CV_uns}
 \lim_{n\to\infty}\left\|\left(u_n(s),\partial_su_n(s))-(u(s),\partial_su(s)\right)\right\|_{\hdot\times L^2}=0.
\end{equation} 
Let 
\begin{equation}
\label{modulated_n}
v_{0n}(y)=\mu_n^{N/2-1}(s)u_n\left(s,\mu_n(s)y+y_n(s)\right),\quad v_{1n}(y)=\mu_n^{N/2}(s)\partial_su_n\left(s,\mu_n(s)y+y_n(s)\right).
\end{equation} 
Note that $(v_{0n},v_{1n})\in K$ for all $n$. Since $\overline{K}$ is compact and the $u_n$'s are not identically zero so that $(0,0)\notin \overline{K}$, we obtain in view of \eqref{CV_uns} that there exists a constant $C_0(s)$ such that
$$\forall n,\quad |y_n(s)|+\mu_n(s)+1/\mu_n(s)\leq C_0(s).$$
We can extract subsequences, so that $(y_n(s),\mu_n(s))$ converges, as $n\to\infty$ to some $(y(s),\mu(s))\in \RR^N\times (0,+\infty)$. Passing to the limit in \eqref{modulated_n}, we deduce, in view of \eqref{CV_uns} 
$$\left(\mu^{N/2-1}(s)u\left(s,\mu(s)\cdot+y(s)\right),\mu^{N/2}(s)\partial_su\left(s,\mu(s)\cdot +y(s)\right)\right)\in \overline{K},$$
concluding the proof.

\EMPH{Proof of Point \eqref{I:Lorentz_continuity}}
We first prove:
\begin{claim}
\label{C:continuity}
Let $\{u_n\}_n$, $u$, be as in Lemma \ref{L:Lorentz_continuity}. Assume furthermore
$$I_{\max}(u)=(-\infty,s^+),\quad s^+\in (0,+\infty).$$
Let $y^+$ be the blow-up point of $u$ at time $s^+$.
Let $s\in I_{\max}(u)$, and $\eps>0$. Then if $u_n$ is global for large $n$,
$$ \lim_{n\to\infty} \int_{|y-y^+|>|s-s^+|+\eps} |\nabla u_n(s)|^2+(\partial_t u_n(s))^2\,dy=0.$$
If $u_n$ is not global for large $n$ and $y_n^{\pm}$ is the blow-up point of $u_n$ at time $s_n^{\pm}$, then $(s_n^{\pm},y_n^{\pm})\notin \{|y-y^+|>|s-s^+|+\eps\}$ and
$$ \lim_{n\to\infty} \int_{|y-y^+|>|s-s^+|+\eps} |\nabla \unu_n(s)|^2+(\partial_t \unu_n(s))^2\,dy=0.$$
\end{claim}
 \begin{proof}
 Let $\eta$ such that $0<\eta\leq \eps$ and $s<s^+-\eta$.
 Let $s_0=s^+-\eta$. Let 
 $(\tu_{n0},\tu_{n1})=\chi\left(\frac{y-y^+}{2\eta}\right)(u_n(s_0),\partial_tu_n(s_0))$, where $\chi\in C^{\infty}(\RR^N)$, $\chi(y)=1$ if $|y|\geq 1$, $\chi(y)=0$ if $|y|\leq 1/2$. Since $(u(s_0,y),\partial_tu(s_0,y))=0$ if $|y-y^+|\geq \eta$, we have
 \begin{equation*}
 \lim_{n\to\infty} \|(\tu_{n0},\tu_{n1})\|_{\hdot\times L^2}=0.
 \end{equation*} 
 and thus
 \begin{equation}
 \label{lim_tu}
 \lim_{n\to\infty} \|(\tu_{n}(s),\partial_s\tu_{n}(s))\|_{\hdot\times L^2}=0.
 \end{equation} 
  By Claim \ref{C:FSP} (or simply finite speed of propagation if $u_n$ is global), 
 $$ |y-y^+|\geq 2\eta+|s-s_0|\Longrightarrow \tu_n(s,y)=\unu_n(s,y).$$
 Since $s<s_0$, we have $2\eta+|s-s_0|=2\eta+s_0-s=\eta+|s^+-s|$, and the conclusion of the claim follows from \eqref{lim_tu}.
 \end{proof}

We can assume, without loss of generality, that  we are in one of the following three cases: $u_n$ is global for large $n$; $I_{\max}(u_n)$ is of the form $(-\infty,s^+_n)$, $s_n^+\in\RR$ for large $n$; or $I_{\max}(u_n)$ is of the form $(s^-_n,+\infty)$, $s_n^-\in\RR$ for large $n$.

\EMPH{Step 1} We prove that there exists $T\in I_{\max}(u_{\ell})$ such that $T\in I_{\max}(u_{n\ell})$ for large $n$. 

If $(u_n)_n$ is global for large $n$, then $u_{n\ell}$ is global for large $n$ and the result is obvious.

Assume that $I_{\max}(u_n)$ is of the form $(-\infty,s_n^+)$ for large $n$, and denote by $y_n^+$ the blow-up point of $u_n$ at $s=s_n^+$. Then 
$$I_{\max}(u_{n\ell})=\left(-\infty,t^+_n\right),\quad t^+_n=\frac{s_n^++\ell y_{n1}^+}{\sqrt{1-\ell^2}}.$$
Since $u_n$ has the compactness property, we deduce
$$\supp \left(u_{0n},u_{1n}\right)\subset \left\{|y-y_n^+|\leq s_n^+\right\}\subset\left\{|y_n^+|-s_n^+\leq |y|\right\}.$$
Since $(u_{0n},u_{1n})$ converges in $\hdot\times L^2$ to $(u_0,u_1)\neq (0,0)$, we deduce that there exists a constant $M>0$ such that
\begin{equation}
 \label{B5}
 \forall n, \quad |y_n^+|\leq M+s_n^+.
\end{equation} 
As a consequence
\begin{equation}
 \label{B6}
 t_n^+=\frac{s_n^+ +\ell y_{n1}^+}{\sqrt{1-\ell^2}}\geq \frac{(1-|\ell|)s_n^+-|\ell|M}{\sqrt{1-\ell^2}}\geq -\frac{|\ell|M}{\sqrt{1-\ell^2}}.
\end{equation} 
If the domain of existence of $u$ is of the form $(-\infty,s^+)$ with $s^+<\infty$, then the domain of existence of $u_{\ell}$ is of the form $(-\infty,t^+)$ and any $T<\min\left(t^+,-\frac{|\ell|M}{\sqrt{1-\ell^2}}\right)$ satisfies the desired property. If $u$ is global, we can take any $T<-\frac{|\ell|M}{\sqrt{1-\ell^2}}$. Finally, if the domain of existence of $u$ is of the form $(s^-,+\infty)$, then by standard long-time perturbation $s_n^+\to +\infty$,  and thus by \eqref{B6}, $t_n^+\to +\infty$. Any $T>t^-$ satisfies the conclusion of Step 1.

If the maximal domain of existence of $u_n$ is of the form $(s_n^-,+\infty)$ for large $n$, the proof is identical and we omit it.

\EMPH{Step 2} We prove that there exists $B>0$ such that
$$ \lim_{n\to\infty} \int_{|x|\geq B}\left|\nabla u_{n\ell}(T,x)-\nabla u_{\ell}(T,x)\right|^2+\left( \partial_tu_{n\ell}(T,x)-\partial_tu_{\ell}(T,x)\right)^2 \,dx=0. $$
Let $\chi\in C_0^{\infty}(\RR^N)$, $\chi(y)=1$ for $|y|\geq 1$, $\chi(y)=0$ for $|y|\leq 1/2$. Let
\begin{equation*}
 (\tu_{0n},\tu_{1n})(y)=\chi\left( \frac{y}{A} \right)(u_{0n},u_{1n})(y),\quad (\tu_{0},\tu_{1})(y)=\chi\left( \frac{y}{A} \right)(u_{0},u_{1})(y).
\end{equation*}
We choose $A$ large, so that 
$$\|(\tu_0,\tu_1)\|_{\hdot\times L^2}\leq \frac{\delta_0}{2},$$
where again $\delta_0$ is given by the small data theory. Then for large $n$,
$$\|(\tu_{0n},\tu_{1n})\|_{\hdot\times L^2}\leq \delta_0,$$
and $\tu_n$ is global. By Claim \ref{C:FSP}, 
$$|y|\geq A+|s|\Longrightarrow \unu_n(s,y)=\tu_n(s,y),$$
and, if $u_n$ is not globally defined forward in time (respectively backward in time), $|y_n^{+}|< A+|s_n^{+}|$ (respectively $|y_n^{-}|< A+|s_n^{-}|$). Similarly, $\unu=\tu$ for $|y|\geq A+|s|$ and, if $u$ is not globally defined forward in time (respectively backward in time), $|y^{+}|< A+|s^{+}|$ (respectively $|y^{-}|< A+|s^{-}|$). As a consequence (for large $n$),
\begin{equation}
 \label{B7}
 |x|\geq c_{\ell}A+|t|\Longrightarrow \tu_{n\ell}(t,x)=u_{n\ell}(t,x),\quad \tu_{\ell}(t,x)=u_{\ell}(t,x).
 \end{equation}
 By the small data theory,
 $$\lim_{n\to\infty}\left\|\tu_n-\tu\right\|_{L^{\frac{N+2}{N-2}}L^{\frac{2(N+2)}{N-2}}}=0.$$
Since 
$$\partial_t^2(\tu_n-\tu)-\Delta(\tu_n-\tu)=|\tu_n|^{\frac{4}{N-2}}\tu_n-|\tu|^{\frac{4}{N-2}}\tu,$$
and 
$$\left\||\tu_n|^{\frac{4}{N-2}}\tu_n-|\tu|^{\frac{4}{N-2}}\tu\right\|_{L^1L^2}\leq C\|\tu_n-\tu\|_{L^{\frac{N+2}{N-2}}L^{\frac{2(N+2)}{N-2}}}\left(\left\|\tu_n\right\|^{\frac{4}{N-2}}_{L^{\frac{N+2}{N-2}}L^{\frac{2(N+2)}{N-2}}}+\left\|\tu\right\|^{\frac{4}{N-2}}_{L^{\frac{N+2}{N-2}}L^{\frac{2(N+2)}{N-2}}}\right),$$
goes to $0$ as $n\to\infty$,
we obtain by Claim \ref{C:linear_Lorentz}
\begin{equation}
 \label{B8}
 \sup_{t\in \RR}\left\|\left(\tu_{n\ell}(t)-\tu_{\ell}(t),\partial_t \tu_{n\ell}(t)-\partial_t\tu_{\ell}(t)\right)\right\|_{\hdot\times L^2}\underset{n\to\infty}{\longrightarrow}0.
\end{equation} 
The conclusion of Step 2 follows from \eqref{B7} and \eqref{B8}.
 
\EMPH{Step 3} Let $X\in \RR^N$. We show that there exists $\eta>0$ such that
\begin{equation*}
 \lim_{n\to\infty} \int_{|x-X|<\eta} \left|\nabla u_{n\ell}(T,x)-\nabla u_{\ell}(T,x)\right|^2+\left|\partial_t u_{n\ell}(T,x)-\partial_t u_{\ell}(T,x)\right|^2\,dx =0.
\end{equation*} 
We let as usual $(S,Y)=\phi_{\ell}^{-1}(T,X)$. We distinguish two cases

\EMPH{Case 1: $S\in I_{\max}(u)$} In this case, $S\in I_{\max}(u_n)$ for large $n$. We let $\psi\in C_0^{\infty}(\RR^N)$ such that $\psi(y)=1$ if $|y|\leq 1$ and $\psi(y)=0$ if $|y|\geq 2$. Define:
\begin{equation*}
 (\tu_{0n},\tu_{1n})(y)=\psi\left( \frac{y-Y}{\eps} \right)(u_{0n},u_{1n})(y),\quad (\tu_{0},\tu_{1})(y)=\psi\left( \frac{y-Y}{\eps} \right)(u_{0},u_{1})(y),
\end{equation*}
and choose $\eps>0$ so that $\|(\tu_0,\tu_1)\|_{\hdot\times L^2}\leq \delta_0/2$. As in Step 2, we get that $\tu$ is global and scattering, that $\tu_n$ is global and scattering for large $n$ and (using Claim \ref{C:FSP}) that
$$ \left|x-X\right|\leq \frac{\eps}{c_{\ell}}-|t-T|\Longrightarrow \tu_{n\ell}(t,x)=u_{n\ell}(t,x)\text{ and } \tu_{\ell}(t,x)=u_{\ell}(t,x).$$
Using as in Step 2 Claim \ref{C:linear_Lorentz}, we get 
\begin{equation*}
 \sup_{t\in \RR}\left\|\left(\tu_{n\ell}(t)-\tu_{\ell}(t),\partial_t \tu_{n\ell}(t)-\partial_t\tu_{\ell}(t)\right)\right\|_{\hdot\times L^2}\underset{n\to\infty}{\longrightarrow}0.
\end{equation*} 
and the conclusion follows.

\EMPH{Case 2: $S\notin I_{\max}(u)$} We assume to fix ideas 
$$I_{\max}(U)=(-\infty,s^+),\quad s^+\in\RR,\; S\geq s^+.$$
Using that $T<t^+$, we get $\frac{S+\ell Y_1}{\sqrt{1-\ell^2}}<\frac{s^++\ell y_1^+}{\sqrt{1-\ell^2}}$ and thus $S-s^+<\ell\left|Y_1-y_1^+\right|$. As a consequence, since $S\geq s^+$,
$$ |Y-y^+|>|S-s^+|.$$
This implies $(\unu,\partial_t\unu)=(0,0)$ close to $(S,Y)$. Since by Claim \ref{C:continuity} $(u_n,\partial_tu_n)\to 0$ close to $(S,Y)$, locally in $\hdot\times L^2$, the result follows.

\EMPH{Step 4. End of the proof} By Steps 2 and 3, 
$$\lim_{n\to \infty}\left\|(u_{n\ell}(T),\partial_tu_{n\ell}(T))-(u_{\ell}(T),\partial_tu_{\ell}(T))\right\|_{\hdot\times L^2}=0.$$
The conclusion of Lemma \ref{L:Lorentz_continuity} follows from global in time perturbation theory with initial time $t=T$. 
\end{proof}

\subsection{Preservation of the compactness property by Lorentz transformation}
In view of Lemmas \ref{L:Lorentz_global}, \ref{L:Lorentz_nonglobal}, the proof of Proposition \ref{P:lorentz} will be complete once we have proved:
\begin{lemma}
 \label{L:Lorentz_compactness}
 Let $u$ be a non-zero solution of \eqref{CP} with the compactness property. Then $u_{\ell}$ has the compactness property. 
\end{lemma}
We will need the following Claim, proved in Appendix \ref{A:modulation}.
\begin{claim}
 \label{C:choice_modulation}
 Let $u$ be a non-zero solution of \eqref{CP} with the compactness property. Let $(s^-,s^+)=I_{\max}(u)$. Then there exist $\mu(s)>0$, $y(s)\in \RR^N$ defined for $s\in (s^-,s^+)$ such that
 \begin{enumerate}
  \item \label{I:modul1} if $s^-<s<s^+$,
  $$\frac{1}{3}\int |\nabla_{s,y}u(s,y)|^2\,dy\leq \int_{y_1(s)}^{+\infty} \int_{\RR^{N-1}}|\nabla_{s,y} u(s,y)|^2\,dy'\,dy_1\leq \frac{2}{3}  \int |\nabla_{s,y}u(s,y)|^2\,dy.$$
  \item \label{I:modul2}$\ds s\mapsto y_1(s)$ is continuous on $(s^-,s^+)$.
  \item \label{I:modul3}$\ds K=\left\{\left( \mu^{N/2-1}(s)u(s,\mu(s)\cdot+y(s)),\mu^{N/2}(s)\partial_su(s,\mu(s)\cdot+y(s))\right),\; s\in (s^-,s^+) \right\}$ 
  has compact closure in $\hdot\times L^2$.
 \end{enumerate}
\end{claim}
\begin{remark}
 Of course in the setting of Claim \ref{C:choice_modulation} we could also choose $y_2(s)$, $y_3(s)$ and $\mu(s)$ continuous, but we do not need this fact in the sequel.
\end{remark}
\begin{proof}[Proof of Lemma \ref{L:Lorentz_compactness}]
 We let $y(s)$, $\mu(s)$ and $K$ be as in Claim \ref{C:choice_modulation}. Let $(t^-,t^+)$ be the maximal interval of existence of $u_{\ell}$. 
 
 \EMPH{Step 1}
 Let $t\in (t^-,t^+)$. We show that there exists $s=s(t)\in (s^-,s^+)$ such that $t=\frac{s+\ell y_1(s)}{\sqrt{1-\ell^2}}$. 
Let us mention that $s(t)$ is not always unique.

 Let $f:s\mapsto \frac{s+\ell y_1(s)}{\sqrt{1-\ell^2}}$. Then $f$ is continuous on $(s^-,s^+)$.  We distinguish two cases.
 
 \EMPH{Case 1} $u$ is global. Then by finite speed of propagation, 
 $$\lim_{R\to\infty}\limsup_{t\to +\infty}\int_{|x|\geq t+R} |\nabla u(t,x)|^2+(\partial_tu(t,x))^2\,dx=0,$$
 which implies
 $|y(s)|\leq |s|+M$ for a large constant $M$. Thus 
 $$\lim_{s\to \pm\infty} \frac{s+\ell y_1(s)}{\sqrt{1-\ell^2}}=\pm\infty$$
 and the result follows by the intermediate value theorem.
 
 \EMPH{Case 2} $u$ is not global, say $s^+<\infty$. 
The maximal interval of existence of $u_{\ell}$ is $(-\infty,t^+)$, where $t^+=\frac{s^++\ell y_1^+}{\sqrt{1-\ell^2}}$.  
 As before, we have:
 $$\lim_{s\to -\infty} \frac{s+\ell y_1(s)}{\sqrt{1-\ell^2}}=-\infty$$
 Let $y^+$ be the blow-up point of $u$ at $s=s^+$. By \cite[Proof of Lemma 4.8]{KeMe08}, $y(s)$ is bounded as $s\to s^+$. Let $\{s_n\}_n$ be a sequence in $(s^-,s^+)$ that converges to $s^+$, and such that $\{y(s_n)\}_n$ converges in $\RR^N$. Then (see again \cite{KeMe08}) $\mu(s_n)\to 0$ as $n\to \infty$ and (extracting subsequences if necessary),
 $$\left(\mu^{N/2-1}(s_n)u(s_n,\mu(s_n)\cdot+y(s_n)),\mu^{N/2}(s_n)\partial_su(s_n,\mu(s_n)\cdot+y(s_n))\right)$$
 converges to a non-zero function as $n\to\infty$. Since the preceding function is supported in $\{y\in \RR^N\text{ s.t. }|\mu(s_n)y+y(s_n)-y^+|\leq |s_n-s^+|\}$, we get that $y(s_n)\to y^+$ as $n\to \infty$, and thus (since $\{s_n\}_n$ is an arbitrary sequence that converges to $s^+$), 
 $$\lim_{s\to s^+} y(s)=y^+.$$
 Thus
 $$ \lim_{s\to s^+}\frac{s+\ell y_1(s)}{\sqrt{1-\ell^2}}=t^+,$$
 and the statement follows again from the intermediate value theorem. Step 1 is complete.

\EMPH{Step 2} We let, for $t\in I_{\max}(u)$, 
$$\lambda(t)=\mu(s(t)),\quad x(t)=\frac{y(s(t))+\ell s(t)\ve_1}{\sqrt{1-\ell^2}}.$$
Let 
$$ K_{\ell}=\left\{\lambda^{N/2-1}(t)u_{\ell}(t,\lambda(t)\cdot+x(t)),\lambda^{N/2}(t)\partial_tu_{\ell}(t,\lambda(t)\cdot+x(t)),\quad t\in I_{\max}(u_{\ell})\right\}.$$
The aim of steps 2 and 3 is to show that $K_{\ell}$ has compact closure in $\hdot\times L^2$.

Let $\{t_n\}_n$ be a sequence in $I_{\max}(u_{\ell})$, $s_n=s(t_n)\in I_{\max}(u)$. Let
$$v_n(\tau,z)=\mu^{N/2-1}(s_n)u(s_n+\mu(s_n)\tau,y(s_n)+\mu(s_n)z).$$
Then (extracting subsequences if necessary), there exists $(v_0,v_1)\in \hdot\times L^2$ such that
\begin{equation}
 \label{B9}
 \lim_{n\to\infty}\left\|(v_n(0),\partial_tv_n(0))-(v_0,v_1)\right\|_{\hdot\times L^2}=0.
\end{equation} 
We let $v$ be the solution of \eqref{CP} with initial data $(v_0,v_1)$. Recall from Lemma \ref{L:Lorentz_continuity} that $v$ has the compactness property. In this step we prove that $0\in I_{\max}(v_{\ell})$. If $v$ is global, then $v_{\ell}$ is global and the result follows. We assume $v$ is not global. To fix ideas, we assume that 
$$I_{\max}(v)=(-\infty,S^+),\quad S^+>0.$$
We denote by $Y^+$ the blow-up point for $v$ at $s=S^+$. We have
\begin{equation}
 \label{B12}
 \supp (v_0,v_1)\subset\{|y-Y^+|\leq S^+\}.
\end{equation} 
By the choice of $y_1(s)$ in Claim \ref{C:choice_modulation},
\begin{multline*}
\int_{z_1\geq 0} |\nabla_{\tau,z}v_n(0,z)|^2\,dz=\int_{y_1\geq y_1(s_n)} |\nabla_{s,y}u(s_n,y)|^2\,dy\\
\geq \frac{1}{3} \int |\nabla_{s,y}u(s_n,y)|^2\,dy=\frac{1}{3}\int |\nabla_{\tau,z}v_n(0,z)|^2\,dz. 
\end{multline*}
Letting $n\to\infty$, we obtain
\begin{equation}
 \label{B10}
 \int_{z_1\geq 0} |\nabla_{\tau,z}v(0,z)|^2\,dz\geq \frac{1}{3}\int |\nabla_{\tau,z}v(0,z)|^2\,dz.
\end{equation} 
and similarly
\begin{equation}
 \label{B11}
 \int_{z_1\leq 0} |\nabla_{\tau,z}v(0,z)|^2\,dz\geq \frac{1}{3}\int |\nabla_{\tau,z}v(0,z)|^2\,dz.
\end{equation} 
We prove by contradiction $|Y_1^+|<S^+$. If for example $Y_1^+\leq -S^+$, then $y_1>0\Longrightarrow y_1-Y^+_1>S^+$ and \eqref{B12} implies
$$ \int_{y_1>0} |\nabla_{s,y} v(0,y)|^2\,dy=0,$$
which contradicts \eqref{B10}, since $v$ is not identically $0$, proving $Y_1^+> -S^+$. Similarly (using \eqref{B11}), we get $Y_1^+< S^+$. Recalling that 
$$T^+(v_{\ell})=\frac{S^++\ell Y_1^+}{\sqrt{1-\ell^2}},$$
we get $T^+(v_{\ell})>0$, which concludes Step 2.

\EMPH{Step 3}
By Step 2, $0\in I_{\max}(v_{\ell})$. By Lemma \ref{L:Lorentz_continuity}, $0\in I_{\max}(v_{n\ell})$ for large $n$ and 
\begin{equation}
\label{B13}
\lim_{n\to\infty} \left\|(v_{n\ell},\partial_tv_{n\ell})(0)-(v_{\ell},\partial_tv_{\ell})(0)\right\|_{\hdot\times L^2}=0.
 \end{equation} 
Moreover, letting $t_n=\frac{s_n+\ell y_1(s_n)}{\sqrt{1-\ell^2}}$, we have $x(t_n)=\frac{y(s_n)+\ell s_n\ve_1}{\sqrt{1-\ell^2}}$ and thus
\begin{multline*}
\lambda^{N/2-1}(t_n)u_{\ell}(t_n,\lambda(t_n)x+x(t_n))\\=\lambda^{N/2-1}(t_n)\unu\left( \frac{t_n-\ell(\lambda(t_n)x_1+x_1(t_n))}{\sqrt{1-\ell^2}},\frac{\lambda(t_n)x_1+x_1(t_n)-\ell t_n}{\sqrt{1-\ell^2}},\lambda(t_n)x'+x'(t_n) \right)\\
=\mu^{N/2-1}(s_n)\unu\left( s_n-\frac{\mu(s_n)\ell x_1}{\sqrt{1-\ell^2}},y_1(s_n)+\frac{\mu(s_n)x_1}{\sqrt{1-\ell^2}},\mu(s_n)x'+y'(s_n) \right)=v_{n\ell}(0,x)
\end{multline*}
and similarly,
\begin{equation*}
\lambda^{N/2}(t_n)\partial_t u_{\ell}\left(t_n,\lambda_n(t_n)x+x(t_n)\right)=\partial_tv_{n\ell}(0,x).
\end{equation*}
Combining with \eqref{B13}, we get that 
$$\left(\lambda^{N/2-1}(t_n)u_{\ell}(t_n,\lambda(t_n)x+x(t_n)),\lambda^{N/2}(t_n)\partial_t u_{\ell}\left(t_n,\lambda_n(t_n)x+x(t_n)\right)\right)$$ 
converges in $\hdot\times L^2$ as $n\to \infty$. The proof is complete.
\end{proof}
\begin{remark}
\label{R:Lorentz}
 From the proof of Lemma \ref{L:Lorentz_compactness}, we see that if $(u(s_n),\partial_tu(s_n))$ converges, up to scaling and space translation, to $(v_0,v_1)\in \hdot\times L^2$, then $(u_{\ell}(t_n),\partial_tu_{\ell}(t_n))$ converges (again up to scaling and space translation) to $(v_{\ell}(0),\partial_tv_{\ell}(0))$, where $v$ is the solution of \eqref{CP} with initial data $(v_0,v_1)$, and the sequence $(t_n)_n$ can be taken as in Step 1 of the proof of Lemma \ref{L:Lorentz_compactness}.
\end{remark}

\appendix

\section{Estimates on modulated functions}
As in Section \ref{S:stationary}, we denote by $A=(s,a,b,c)$ an element of $\RR^{N'}$, where $s\in \RR$, $a,b$ are in $\RR^{N}$ and $c$ in $\RR^{\frac{N(N-1)}{2}}$. Recall from \eqref{Pc} the definition of $P_c$.
\begin{lemma}
 \label{L:C1}
 Let $\psi\in \SSS(\RR^N)$ and $q>0$. Then the function
  $$ F:(x,A)\in \RR^N\times \RR^{N'}\longmapsto \left(\frac{e^{s/2}}{\left|\frac{x}{|x|}-|x|a\right|}\right)^q \psi\left( b+\frac{e^sP_c(x-|x|^2a)}{\left|\frac{x}{|x|}-|x|a\right|^2} \right)$$
can be extended to a $C^{\infty}$ function on $\RR^{N+N'}$. If $K$ is a compact subset of $\RR^{N'}$, there exists a constant $C_K>0$ such that
 \begin{multline}
 \label{C1}
 \forall x\in \RR^N,\; \forall A\in K,\quad |F(x,A)|+|\nabla_A F(x,A)|+|x\nabla_xF(x,A)|\\
 +|x\nabla_x\nabla_A F(x,A)|+|\nabla_A^2F(x,A)|\leq \frac{C_K}{(1+|x|)^q}.
 \end{multline}
 Furthermore
 \begin{align}
  \label{deriv_s}
  \frac{\partial F}{\partial s}(x,0)&=\frac{q}{2}\psi(x)+x\cdot\nabla \psi(x)\\
  \label{deriv_a}
  \frac{\partial F}{\partial a_j}(x,0)&=q x_j \psi(x)-|x|^2\frac{\partial f}{\partial x_j}(x)+2x_jx\cdot\nabla \psi(x)\\
  \label{deriv_b}
  \frac{\partial F}{\partial b_j}(x,0)&=\partial_{x_j}\psi(x)\\
  \label{deriv_c}
\frac{\partial F}{\partial c_j}(x,0)&= (x_{\ell}\partial_{x_k}-x_k\partial_{x_{\ell}})\psi(x),\quad \zeta(k,\ell)=j,
\end{align}
where $\zeta$ is introduced before \eqref{Pc}.
\end{lemma}
\begin{proof}
 Note that $\left|\frac{x}{|x|} -|x|a\right|^2=1-2\la a,x\ra +|a|^2|x|^2$, which is $>0$ if $a\neq x/|x|^2$. Thus $F$ can be extended to a smooth function on the open set 
 $$\left\{(x,A)\in \RR^N\times \RR^{N'}\text{ s.t. } x=0\text{ or }a\neq \frac{x}{|x|^2}\right\}.$$
 Next, we notice
 \begin{equation}
  \label{C6}
  F(x,A)=\frac{1}{|x|^q}G\left( \frac{x}{|x|^2}-a,A \right),
 \end{equation} 
 where 
 $$ G(y,A)=\frac{e^{q\,s/2}}{|y|^q} f\left( b+\frac{P_ce^sy}{|y|^2} \right),\quad G(0,A)=0.$$
Obviously, $G$ is smooth away from $y=0$. We prove that $G\in C^{\infty}(\RR^N\times \RR^{N'})$. Let us fix a large $M>0$. Let $\eps_K>0$ be a small constant, depending on $K$, to be specified. Using that $f\in \SSS(\RR^N)$, we get, if $A\in K$, $0<|y|<\eps_K$,
 $$|G(y,A)|\leq \frac{C_{K,M}}{|y|^q\,\left|b+\frac{P_c e^s y}{|y|^2}\right|^{q+M}} \leq \frac{C_{K,M}}{|y|^q\,\left(\frac{e^s |y|}{|y|^2} -|b|\right)^{q+M}}\leq \frac{C_{K,M}}{|y|^q} |y|^{q+M}=C_K|y|^{M}.$$
 As a consequence, $G$ is continuous also at $y=0$. Bounding similarly the derivatives of $G$, we deduce that $G$ is smooth and vanishes at infinite order at $y=0$. Going back to \eqref{C6}, we deduce that $F$ is smooth.

We next show the bound on $F$ in \eqref{C1}. Since $F$ is continous, it is bounded in $B^{N}(1)\times K$. To get a bound for $|x|\geq 1$, we use \eqref{C6}. Let $y=\frac{x}{|x|^2}-a$. If $|x|\geq 1$ and $A\in K$, then $|y|\leq C_K'=1+\max_{A\in K} |a|$. Let
$$M_K'=\sup_{\substack{|y|\leq C_K'\\ A\in K}} \left|G(y,A) \right|<\infty.$$
 Then:
$$\left|F(x,A)\right|\leq \frac{M'_k}{|x|^q},\quad |x|\geq 1,\; A\in K,$$
which completes the proof of the bound on $F$ in \eqref{C1}. The proof of the bounds on the derivatives of  $F$ in \eqref{C1} is similar and we omit it. 

Finally, \eqref{deriv_s},\eqref{deriv_a},\eqref{deriv_b} and \eqref{deriv_c} follow from explicit computations.
\end{proof}
As an immediate consequence of Lemma \ref{L:C1}, recalling
$$  \left(\theta^{-1}_{\vA}\right)^*(f)(x)=e^{\frac{(N+2)s}{2}} \left|\frac{x}{|x|} -|x|a\right|^{-(N+2)}f\left(b+\frac{e^sP_c(x-|x|^2a)}{1-2\langle a,x\rangle +|a|^2|x|^2}\right), \quad f\in \hdot,
$$
we obtain:
\begin{corol}
\label{C:C1function}
 There exists $\eps>0$ with the following property. Let $\psi\in \SSS(\RR^N)$. Then $A\mapsto (\theta^{-1}_A)^*\psi$ is a $C^1$ function from $B^{N'}(\eps)$ to $H^1(\RR^N)$. Its derivatives at $A=0$ are given by \eqref{deriv_s}, \eqref{deriv_a}, \eqref{deriv_b}, \eqref{deriv_c} with $q=N+2$.
\end{corol}
We finally prove the following estimate:
\begin{lemma}
 \label{L:C2}
 Let $f\in \SSS$ and $p\geq 1$. There exists $C,\eps>0$ such that for all $A\in \RR^{N'}$,
 $$ |A|<\eps\Longrightarrow \left\| f-\left(\theta_A^{-1}\right)^* f\right\|_{L^p}\leq C|A|.$$
\end{lemma}
\begin{proof}
 Indeed by the bound on $\nabla_A F$ in Lemma \ref{L:C1}, if $|A|\leq \eps$,
 $$ \left| \left(f-\left(\theta_A^{-1}\right)^* f\right)(y)\right|\leq \frac{C}{1+|y|^{N+2}}|A|,$$
 and the Lemma follows, integrating with respect to $y$.
\end{proof}
\section{Nonexistence of solutions converging fast to a stationary solution}
In this appendix we prove Claim \ref{C:uniq_expo}. By standard long-time perturbation theory, there exists $\eps_0>0$, $M>0$ such that, for all solution $v$ of \eqref{CP}, with initial data $(v_0,v_1)$  such that
$$\|v_0-S\|_{\hdot}+\|v_1\|_{L^2}=\eps<\eps_0,$$
we have
$$ [-1,+1]\subset (T_-(v),T_+(v))\text{ and } \sup_{t\in [-1,+1]} \left(\|v(t)-S\|_{\hdot}+\|\partial_tv(t)\|_{L^2}\right)\leq M\eps.$$
By induction, we deduce that for all integers $T\geq 1$, if $v$ satisfies
$$\|v_0-S\|_{\hdot}+\|v_1\|_{L^2}=\eps<\frac{\eps_0}{M^T},$$
then 
$$ [-T,+T]\subset (T_-(v),T_+(v))\text{ and } \sup_{t\in [-T,+T]} \left(\|v(t)-S\|_{\hdot}+\|\partial_tv(t)\|_{L^2}\right)\leq M^T\eps.$$
Let $\nu>0$ such that $e^{-\nu}M<1$. Let $u$ be a solution of \eqref{CP} such that \eqref{uniq_expo} holds. Then for any large integer $T$
$$\|u(T)-S\|_{\hdot}+\|\partial_tu(T)\|_{L^2}\leq Ce^{-\nu T}<\frac{\eps_0}{M^T}.$$
As a consequence,
$$\sup_{t\in [0,T]}\|u(t)-S\|_{\hdot}+\|\partial_t u(t)\|_{L^2}<Ce^{-\nu T}M^T\underset{T\to +\infty}{\longrightarrow}0.$$
We deduce that $(u_0,u_1)=(S,0)$ and thus by uniqueness that $u\equiv S$.
\section{Choice of the translation parameter}
\label{A:modulation}
In this appendix we prove Claim \ref{C:choice_modulation}. Let $\mu(s)>0$, $\ty(s)\in \RR^N$ such that
$$\tK=\left\{\left( \mu^{N/2-1}(s)u(s,\mu(s)\cdot+\ty(s)),\mu^{N/2}(s)\partial_su(s,\mu(s)\cdot+\ty(s))\right),\; s\in (s^-,s^+) \right\}$$
has compact closure in $\hdot\times L^2$. Let 
$$\Phi(s,y)=\frac{|\nabla_{s,y}u(s,y)|^2}{\int |\nabla_{s,z}u(s,z)|^2\,dz}.$$
Note that $\Phi$ is well-defined (since $u\not\equiv 0$), nonnegative, that $s\mapsto \Phi(s,\cdot)$ is continuous from $(s^-,s^+)$ to $L^1(\RR^N)$ and that $\int\Phi(s,y)\,dy=1$ for all $s$.  If $s\in(s^-,s^+)$, the function
$$F_s: Y_1\mapsto \int_{Y_1}^{+\infty} \left(\int_{\RR^{N-1}}\Phi(s,y)\,dy'+\frac{1}{3\sqrt{\pi}}e^{-y_1^2}\right)\,dy_1$$
is strictly decreasing and satisfies
$$ \lim_{Y_1\to -\infty} F_s(Y_1)=\frac{4}{3},\quad \lim_{Y_1\to+\infty} F_s(Y_1)=0.$$
We let $y_1(s)$ be the unique element of $\RR$ such that $F_s(y_1(s))=\frac 23$. We define 
$$y(s)=(y_1(s),\ty_2(s),\ty_3(s)).$$
Let us prove that $y,\mu$ satisfy points \eqref{I:modul1}, \eqref{I:modul2} and \eqref{I:modul3}.

\EMPH{Proof of \eqref{I:modul1}}
$$\frac{2}{3}=\int_{y_1(s)}^{+\infty} \int_{\RR^{N-1}} \Phi(s,y)\,dy'\,dy_1+\frac{1}{3\sqrt{\pi}}\int_{y_1(s)}^{+\infty}e^{-y_1^2}\,dy_1\leq \int_{y_1(s)}^{+\infty} \int_{\RR^{N-1}} \Phi(s,y)\,dy'\,dy_1+\frac{1}{3}.$$
Thus
$$\frac{2}{3}-\frac{1}{3}\leq \int_{y_1(s)}\int_{\RR^{N-1}}\Phi(s,y)\,dy'\,dy_1\leq \frac 23$$ 
and \eqref{I:modul1} follows.

\EMPH{Proof of \eqref{I:modul2}}

Let $s\in(s^-,s^+)$ and $\{s_n\}_n$ be a sequence in $(s^-,s^+)$ converging to $s$. We have
\begin{multline*}
\int_{y_1(s)}^{+\infty} \int_{\RR^{N-1}} \Phi(s,y)\,dy'\,dy_1+\frac{1}{3\sqrt{\pi}}\int_{y_1(s)}^{+\infty}e^{-y_1^2}\,dy_1=\frac{2}{3}\\
=\int_{y_1(s_n)}^{+\infty} \int_{\RR^{N-1}} \Phi(s_n,y)\,dy'\,dy_1+\frac{1}{3\sqrt{\pi}}\int_{y_1(s_n)}^{+\infty}e^{-y_1^2}\,dy_1
\end{multline*}
Thus
\begin{equation*}
 0=\int_{y_1(s)}^{y_1(s_n)} \left(\int_{\RR^{N-1}}\Phi(s,y)dy'+\frac{1}{3\sqrt{\pi}}e^{-y_1^2}\right)\,dy_1-\int_{y_1(s_n)}^{+\infty}\int_{\RR^{N-1}} \left(\Phi(s_n,y)-\Phi(s,y)\right)\,dy'\,dy_1.
\end{equation*}
Hence (since $\Phi(s,y)\geq 0$),
$$\left|\int_{y_1(s)}^{y_1(s_n)}e^{-y_1^2}\,dy_1\right|\leq 3\sqrt{\pi} \int |\Phi(s_n,y)-\Phi(s,y)|\,dy\underset{n\to\infty}{\longrightarrow}0,$$
which shows that $y_1(s_n)\to y_1(s)$ as $n\to \infty$, concluding the proof of the continuity of $s\mapsto y_1(s)$.

\EMPH{Proof of \eqref{I:modul3}}
We prove that there exists a constant $C>0$ such that 
\begin{equation}
\label{modul3_interm}
\forall s\in (s^-,s^+),\quad \ty_1(s)-C\mu(s)\leq y_1(s)\leq \ty_1(s)+C\mu(s).
\end{equation} 
If not, we can find a sequence $\{s_n\}_n$ in $(s^-,s^+)$ such that, for example
$$ \forall n,\quad y_1(s_n)>\ty_1(s_n)+n\mu(s_n).$$
We have
\begin{multline*}
 \int_{y_1\geq n} \mu^{N}(s_n)\left|\nabla_{s,y}u(s_n,\mu(s_n)y+\ty(s_n))\right|^2\,dy=\int_{z_1\geq \ty_1(s_n)+n\mu(s_n)} \left|\nabla_{s,y}u(s_n,z)\right|^2\,dz\\
 \geq \int_{z_1\geq y_1(s_n)}|\nabla_{s,y}u(s_n,z)|^2\,dz\geq \frac{1}{3}\int_{\RR^N}|\nabla_{s,y}u(s_n,z)|^2\,dz,
\end{multline*}
which gives a contradiction, since by the compactness of the closure of $\tK$,  the first term in the preceding inequalities goes to $0$ as $n\to\infty$. 

The compactness of $\overline{K}$ follows easily from \eqref{modul3_interm} and the compactness of $\overline{\tK}$. We omit the proof.

\section{Some space-time estimates}
\label{A:space-time}
In this appendix we prove Claim \ref{C:estimates}. By Proposition \ref{P:A1}, \eqref{I:harmonic}
\begin{multline*}
 \left\|S\chi_{r_0,t_0}\right\|_{L^{\frac{N+2}{N-2}}L^{\frac{2(N+2)}{N-2}}}\leq C\left(\int_{-\infty}^{+\infty} \left(\int_{|x|\geq r_0+|t-t_0|} \frac{1}{|x|^{2(N+2)}} \,dx\right)^{1/2}\,dt\right)^{\frac{N-2}{N+2}}\\
\leq C\left(\int \frac{1}{(r_0+|t-t_0|)^{\frac{N}{2}+2}}\,dt\right)^{\frac{N-2}{N+2}}\leq \frac{C}{r_0^{N/2-1}},
\end{multline*}
which yields the first inequality of the Claim.

By Lemma \ref{L:estim_Y}, 
\begin{multline*}
 \left\|e^{-\omega t}Y \chi_{r_0,t_0}\right\|_{L^{\frac{N+2}{N-2}}L^{\frac{2(N+2)}{N-2}}}= C\left(\int_{-\infty}^{+\infty} \left(\int_{|x|\geq r_0+|t-t_0|} |Y|^{\frac{2(N+2)}{N-2}}\,dx\right)^{\frac{1}{2}} e^{-\frac{N+2}{N-2}\omega t}\,dt\right)^{\frac{N-2}{N+2}}\\
\leq C\left(\int_{-\infty}^{+\infty} \frac{1}{(|t-t_0|+r_0)^{q_N/2}}\left(e^{- \frac{2(N+2)}{N-2}\omega(r_0+|t-t_0|)}\right)^{\frac{1}{2}} e^{-\frac{N+2}{N-2}\omega t}\,dt\right)^{\frac{N-2}{N+2}}.
\end{multline*}
We note that $e^{-\frac{N+2}{N-2}\omega|t-t_0|- \frac{N+2}{N-2}\omega t}\leq e^{-\frac{N+2}{N-2}\omega t_0}$. Hence (using that $q_N>2$)
\begin{equation*}
 \left\|e^{-\omega t}Y \chi_{r_0,t_0}\right\|_{L^{\frac{N+2}{N-2}}L^{\frac{2(N+2)}{N-2}}}\leq Ce^{-\omega (r_0+t_0)}\left(\int_{-\infty}^{+\infty} \frac{dt}{(r_0+|t-t_0|)^{q_N/2}}\right)^{\frac{N-2}{N+2}}\leq C e^{-\omega(r_0+t_0)},
\end{equation*} 
which yields the second inequality of the Claim.
\bibliographystyle{acm}
\bibliography{toto}

\end{document}